\documentclass[reqno,a4paper,12pt]{amsart}

\usepackage{mystyle}
\usepackage{etoolbox}

\usepackage[utf8]{inputenc}
\usepackage[T1]{fontenc}
\usepackage{lmodern}
\usepackage[british]{babel}
\usepackage[nopatch=eqnum]{microtype} %
\usepackage[headings]{fullpage}
\usepackage{parskip}

\usepackage[pdfencoding=auto,pdfusetitle]{hyperref}
\hypersetup{colorlinks=true, linkcolor=blue!30!black, citecolor=green!30!black, urlcolor=red!45!black}

\usepackage{microtype}

\usepackage{amsfonts,amssymb,bbm,stmaryrd}
\usepackage[scr=boondoxo,scrscaled=1.05]{mathalfa}
\usepackage{amsmath}

\usepackage{mathtools}
\usepackage{thmtools}
\hypersetup{hypertexnames=false}
\usepackage[noabbrev,capitalise]{cleveref}
\crefformat{equation}{(#2#1#3)}
\crefformat{enumi}{(#2#1#3)}
\crefformat{enumii}{(#2#1#3)}

\usepackage{centernot,float,subcaption,url,array,booktabs,colortbl}

\usepackage[shortlabels]{enumitem}

\usepackage{tikz}
\usetikzlibrary{matrix,positioning}
\usetikzlibrary{calc}
\usetikzlibrary{arrows}
\tikzset{> =stealth}

\usepackage{quiver}

\theoremstyle{plain}
\newtheorem{theorem}{Theorem}[section]
\newtheorem{lemma}[theorem]{Lemma}
\newtheorem{proposition}[theorem]{Proposition}
\newtheorem{corollary}[theorem]{Corollary}

\theoremstyle{definition}
\newtheorem{definition}[theorem]{Definition}
\theoremstyle{remark}
\newtheorem{remark}[theorem]{Remark}
\newtheorem{example}[theorem]{Example}

\newtheorem{excont}{Example}

\renewcommand{\epsilon}{\varepsilon}
\renewcommand{\phi}{\varphi}

\newcommand{\N}{\mathbb{N}}
\newcommand{\Z}{\mathbb{Z}}
\newcommand{\Q}{\mathbb{Q}}
\newcommand{\R}{\mathbb{R}}

\newcommand{\1}{{1}}

\newcommand{\Srpnsk}{\mathbb{S}}

\newcommand{\Set}{\mathbf{Set}}

\newcommand{\Frm}{\mathbf{Frm}}
\newcommand{\Loc}{\mathbf{Loc}}

\newcommand{\Cat}{\mathbf{Cat}}
\newcommand{\Grpd}{\mathbf{Grpd}}
\newcommand{\AnaCat}{\mathbf{ACat}}
\newcommand{\AnaGrpd}{\mathbf{AGrpd}}
\newcommand{\CAT}{\mathbf{CAT}}

\newcommand{\KHausLoc}{\mathbf{KHausLoc}}

\newcommand{\op}{{^\mathrm{\hspace{0.5pt}op}}}
\newcommand{\id}{\mathrm{id}}

\DeclareMathOperator{\Hom}{Hom}
\DeclareMathOperator{\Sh}{Sh}
\renewcommand{\O}{\mathcal{O}}

\newcommand{\Tmod}{\theory\text{-}\mathrm{Mod}}
\newcommand{\Tpmod}{\theory'\text{-}\mathrm{Mod}}
\newcommand{\Sstruct}{\Sigma\text{-}\mathrm{Struct}}

\newcommand{\slashedrightarrow}{\mathrel{\ooalign{\hss$\vcenter{\hbox{\tikz{\draw (0,4.5pt) -- (0,0) ;}}}\mkern2.75mu$\hss\cr$\to$}}}
\newcommand{\xslashedrightarrow}[1]{\mathrel{\ooalign{\hss$\vcenter{\hbox{\tikz{\draw (0,4.5pt) -- (0,0) ;}}}\mkern2.75mu$\hss\cr$\xrightarrow{{#1}}$}}}

\newcommand{\topos}{\mathcal{E}}

\newcommand{\G}{\mathcal{G}}
\renewcommand{\H}{\mathcal{H}}
\newcommand{\K}{\mathcal{K}}
\newcommand{\C}{\mathfrak{C}}
\newcommand{\D}{\mathcal{D}}
\newcommand{\X}{\mathcal{X}}
\newcommand{\theory}{\mathbb{T}}

\newcommand{\Topos}{{\bf Topos}}

\newcommand{\LocGrpd}{\mathbf{Grpd}(\Loc)}
\newcommand{\LocGrpdAna}{\mathbf{AGrpd}(\Loc)}
\newcommand{\LocCat}{\mathbf{Cat}(\Loc)}
\newcommand{\LocCatAna}{\mathbf{ACat}(\Loc)}

\newcommand{\ReprG}{{\rm G}}
\newcommand{\ReprE}{{\rm E}}
\newcommand{\rmLH}{{\rm LH}}
\newcommand{\rmPS}{{\rm PS}}

\newcommand{\Sub}{{\rm Sub}}
\newcommand{\Id}{\mathrm{Id}}
\newcommand{\Obj}{\mathbb{O}}
\newcommand{\paronto}{%
  \rightharpoondown\mathrel{\mspace{-15mu}}\rightharpoondown
}

\DeclareFontFamily{U}{mathb}{\hyphenchar\font45}
\DeclareFontShape{U}{mathb}{m}{n}{
<-6> mathb5 <6-7> mathb6 <7-8> mathb7
<8-9> mathb8 <9-10> mathb9
<10-12> mathb10 <12-> mathb12
}{}
\DeclareSymbolFont{mathb}{U}{mathb}{m}{n}
\DeclareMathSymbol{\pprec}{\mathrel}{mathb}{"CE}
 
\renewcommand{\le}{\leqslant}

\title{Generic bundles over a localic category}

\author[G. Manuell]{Graham Manuell}
\address{Department of Mathematical Sciences, Stellenbosch University, South Africa 
\newline
National Institute for Theoretical and Computational Sciences, Stellenbosch, South Africa}
\email{graham@manuell.me}

\author[J. L. Wrigley]{Joshua L. Wrigley}
\address{Université Paris Cité, CNRS, IRIF, F-75013, Paris, France}
\email{josh.l.wrigley@gmail.com}

\thanks{
The first author acknowledges some financial support from the Centre for Mathematics of the University of Coimbra (UIDB/00324/2020, funded by the Portuguese Government through FCT/MCTES). The authors are also grateful to Peter F.\ Faul, whose Thuthuka grant (TTK240408212906) from the National Research Foundation of South Africa provided further travel funding for the second author to visit the first.
}

\date{May 2026}


\begin{document}

\maketitle

\begin{abstract}
    In this paper we construct classifying localic categories and groupoids for various bundles equipped with logical structure.  When these bundles are local homeomorphisms, we recover the localic groupoids that classify geometric theories, demonstrating that these groupoids satisfy a stronger universal property than that of their corresponding classifying toposes.  We also prove a dual result that there exist classifying localic categories and groupoids for proper separated bundles satisfying a dual geometric theory. Thus, localic groupoids classify strictly more kinds of logical theories than toposes.  Our approach provides a concrete construction of the localic categories and the generic bundles involved in terms of generalised frame presentations.  To accommodate our approach, we prove \textit{en passant} a constructive, pointfree version of the Alexandroff--Hausdorff theorem and that internal functors that are fully faithful and effective descent morphisms on objects induce equivalences between the categories of discrete opfibrations over the source and target categories.
\end{abstract}
\setcounter{section}{-1}
\section{Introduction}

In algebraic topology, algebraic structure over a space can often be specified via maps into a particular \emph{classifying space}.  The theory of \emph{classifying toposes} was later developed in analogy with classifying spaces (e.g.\ \cite{hakim} and \cite[\S VIII]{SGL}), and they are now a central object of study in categorical logic --- in particular, there exist classifying toposes for any structure that can be axiomatised in {geometric logic} (for textbook accounts, see \cite{makkaireyes,Elephant,TST}).

Joyal and Tierney showed that every (Grothendieck) topos can be represented by a localic groupoid \cite{joyal1984galois}. In \cite{manuell2023representing}, we provided an explicit description of a localic groupoid representing the classifying topos of a given geometric theory.  As a consequence of the traditional theory of classifying toposes and the work of Moerdijk \cite{moerdijk1988}, this localic groupoid enjoys a universal property amongst \emph{étale-complete} localic groupoids.

Our first contribution is to show that this same groupoid satisfies a universal property amongst \emph{all} localic groupoids, for which we give a direct proof without relying on the theory of classifying toposes.

A key advantage of our approach is that it can easily be modified to give localic groupoids that classify other structures, leading to the second contribution of the present paper: we construct the classifying localic groupoids for proper separated bundles carrying logical structure. 

These latter groupoids might be understood as providing a mirror of topos theory where sets are replaced with compact Hausdorff spaces and geometric logic is replaced with a `dual' geometric logic (a positive infinitary first-order logic admitting infinitary \emph{conjunctions}, but only finite disjunctions, see \cref{sec:logic}). In particular, we obtain a classifying localic groupoid for compact Hausdorff locales, which was independently discovered by Henry and Townsend \cite{henry_townsend_classifying} using very different means.

In this way the bicategory of localic groupoids (and anafunctors) is seen to be very rich, containing not just all (Grothendieck) toposes, but also a kind of dual to toposes, and more besides. We believe that this bicategory is worthy of further study.

\subsection*{Classifying groupoids via forcing extensions}

The genesis of the present paper lies in our efforts to provide explicit descriptions of the localic constructions involved in the Joyal--Tierney paper.  The approach we adopt differs greatly from that of Henry and Townsend, which expresses the existence of an object classifier in the language of \emph{stacks}.  Our construction, by contrast, involves only the techniques of classifying locales and anafunctors. The core idea can also be understood in terms of the theory of \emph{forcing extensions} found in set theory.

As an illustrative example consider the problem of finding a locale $L[\Obj]$ whose points correspond to (isomorphism classes of) sets. If we bounded the size of the sets we wish to classify by some cardinality $\kappa$, then this would not be a difficult problem: we can take a locale with a distinct point for each cardinality below $\kappa$. However, since the number of points of a locale is a set, this is obviously impossible for the proper class of \emph{all} sets.

Nonetheless, for each non-empty set $X \in \Set$, there exists a forcing extension of the chosen model of set theory in which $X$ becomes countable. This means that arbitrary sets $X$ give rise to spans $1 \twoheadleftarrow I \to L[\Obj]$ where $I \twoheadrightarrow 1$ is an effective descent morphism and $L[\Obj]$ is restricted to countable sets. There are many such spans for each set $X$, but by keeping track of isomorphisms of sets it is possible to upgrade this to an equivalence of categories.

Using this intuition, we can build the classifying localic category for \emph{étale bundles}, and consequently for \emph{geometric theories}, by taking a localic category encoding the subquotients of $\N$ (as is done for groupoids in \cite{manuell2023representing}).
In a similar fashion, we will construct the classifying localic category of a \emph{proper separated bundle}, and consequently for \emph{dual geometric theories}, by observing that, for every compact Hausdorff locale $X$,
there is a forcing extension where $X$ is a subquotient of the Cantor space (see \cref{thm:alexandroff_hausdorff}).

\subsection*{Overview}

The paper proceeds as follows.
\begin{enumerate}
    \item We begin with a background section covering the concepts needed to understand the paper.  We cover locale theory, including local homeomorphisms and proper separated morphisms, in addition to internal categories and anafunctors, which are necessary to express the universal property of our classifying category.  We also briefly describe the elements of descent theory and topos theory we use in the paper.
    \item Whereas we give an explicit frame presentation of the generic local homeomorphic bundle, the generic bundle of a proper separated bundle is most naturally expressed in terms of \emph{generalised presentations} of locales, developed in \cite{manuellThesis,manuell2022spectrum}.  These are recalled in \cref{sec:gen_prens}, where we give detailed constructions of the instances that will be used throughout the paper.  We will pay special attention to \emph{dual} presentations, where the object of generators and relations are indexed by compact Hausdorff locales.  In particular, we prove that a compact Hausdorff locale can be presented canonically in terms of itself (\cref{lem:canonical_prens}), just as discrete locales can be. 
    \item In \cref{sec:forcing}, we prove two locale-theoretic forcing results which we will need to prove the universal property of our generic bundles.  The first expresses the well-known fact that, up to a forcing extension, `every set is subcountable'.  For the second result, we prove a constructive, pointfree version of the Alexandroff--Hausdorff theorem (\cref{thm:alexandroff_hausdorff}) that expresses that, up to a forcing extension, `every compact Hausdorff locale is a subquotient of the Cantor space'.
    \item In \cref{sec:syntax}, we describe how first-order logic is interpreted in the bundles over a localic category.  After delineating the theories of geometric and dual geometric logic, we describe the local homeomorphic models and proper separated models.  We show in \cref{sec:base_change} that local homeomorphic models and proper separated models are stable under base change. Thus, we obtain a pseudofunctor that assigns to each localic category its local homeomorphic models (respectively, the proper separated models). Finally, in \cref{subsec:generic_model}, we construct the generic local homeomorphic model of a geometric theory and the generic proper separated model of a dual geometric theory.
    \item Then, in \cref{sec:descent_dopf}, we demonstrate that the pseudofunctors constructed in \cref{sec:base_change} extend to the bicategory of localic categories and anafunctors.  To do so, we prove that internal discrete opfibrations descend along those internal functors that are fully faithful and surjective on objects.
    \item The above ingredients are combined in \cref{sec:univ_prop} to prove that the classifying localic categories and generic bundles constructed in \cref{subsec:generic_model} satisfy a universal property.  We then use the defining property of the core groupoid of an internal category to deduce a corresponding universal property of the classifying localic groupoid.
    \item We conclude with \cref{sec:concl} where we discuss possible extensions to our results and advocate for a theory of `dual toposes'.
\end{enumerate}

\section{Background}

\subsection{Locales}

Pointfree topology studies topological spaces by means of their lattices of open sets. The advantages of this approach lie in the existence of classifying locales and in the fact that pointfree topology has a satisfactory constructive theory. Here we give a very brief introduction. See \cite{PicadoPultr}, \cite[Part C]{Elephant} or \cite{manuell2023constructive} for more details.

\begin{definition}
 A \emph{frame} is a complete lattice satisfying the infinite distributivity law \[a \wedge \bigvee_\alpha b_\alpha = \bigvee_\alpha a \wedge b_\alpha.\]
\end{definition}
Frames are (infinitary) algebraic structures with constants $0$ and $1$, a binary operation $\wedge$, and a join operation $\bigvee$ for each cardinality. Frame homomorphisms are maps that preserve finite meets and arbitrary joins. We denote the category of frames by $\Frm$.

A topological space $X$ can be viewed as a frame via its lattice of open sets $\O X$, and a continuous map $f \colon X \to Y$ yields a frame homomorphism $f^* \colon \O Y \to \O X$ which takes the inverse image of opens.

We define the category $\Loc$ of \emph{locales} as the opposite of the category $\Frm$ of frames and frame homomorphisms.  The direction of locale morphisms is chosen to agree with that of continuous maps of topological spaces.

To avoid confusion about which category we are considering, we will distinguish between a locale $X$ and its frame of opens $\O X$. If $f\colon X \to Y$ is a locale morphism, we write $f^*\colon \O Y \to \O X$ for the corresponding frame homomorphism, in analogy with topological spaces. Moreover, we call the elements of $\O X$ the \emph{opens} of $X$.

A regular subobject of a locale is called a \emph{sublocale}. These correspond to quotient frames and should be understood as the pointfree analogue of subspaces. The category of locales has an image factorisation, meaning that every locale morphism $f \colon Y \to X$ factors as an epimorphism $Y \twoheadrightarrow f(Y)$ followed by a sublocale inclusion $f(Y) \hookrightarrow X$, which we call the \emph{image}.

The category of locales is order-enriched (in fact, enriched over directed-complete partial orders) using the normal pointwise order of the associated frame homomorphisms.

\subsubsection{Reasoning using points}\label{sec:reasoning_with_points}

Recall that the points of a topological space $X$ correspond to the continuous maps $1 \to X$ from the one-point space into $X$.  If $X$ is suitably well-behaved then its points correspond to the frame homomorphisms $\O X \to \O 1$.  

Therefore, we define a (global) point of a locale $X$ to be a morphism $p\colon 1 \to X$, where $1$ is given by the frame $\O 1$ from above, which is also the terminal object in $\Loc$. Just as in topology, we can use the points of a locale to `probe' for information about the locale.  For instance, if we wish to test whether two locale morphisms $f, g \colon X \rightrightarrows Y$ are distinct, it suffices to find a point $x \colon 1 \to X$ such that $fx \neq gx$. Famously, a locale might fail to have enough points: it may be the case that $fx = gx$ for all points $x \colon 1 \to X$, but $f \neq g$ (for instance, $X$ could be the locale defined in \cref{ex:partial_surjections_from_N}). Nonetheless, we \emph{can} still reason in a point-like fashion if we go about it in a more sophisticated way, as explained in \cite{vickers2007locales} and \cite[\S 5.3]{moerdijk1988}.

A \emph{generalised point} of $X$ is any locale morphism $A \to X$, and by the Yoneda lemma, generalised points determine the locale completely. Using the techniques of categorical logic it is possible to reason with generalised points largely as if they were points in the familiar (i.e.\ global) sense.

In fact, using generalised points can be more convenient than using global points: to define a continuous map $f \colon X \to Y$ between topological spaces, after specifying how $f$ acts on each point of $X$ we still need to check that $f$ is actually continuous; however, if we specify how $f$ acts on generalised points, continuity is automatic (as can be seen by considering the point $\id\colon X \to X$).  For an example of reasoning via points in action, see the Appendix of \cite{manuell2023representing}.

\subsubsection{Open and proper morphisms}

There are a number of important classes of locale morphisms (generalising notions from classical topology) that will be used in this paper.
\begin{definition}\label{def:open_and_proper}
    A morphism of locales $f\colon X \to Y$ is \emph{open} if the corresponding frame homomorphism $f^*\colon \O Y \to \O X$ has a left adjoint $f_!\colon \O X \to \O Y$ satisfying the Frobenius reciprocity condition
    \[f_!(f^*(u) \wedge v) = u \wedge f_!(v).\]
    
    A morphism of locales $f\colon X \to Y$ is \emph{proper} if the right adjoint $f_*\colon \O X \to \O Y$ of $f^*$ (which always exists) preserves directed joins and satisfies the Frobenius reciprocity condition
    \[f_*(f^*(u) \vee v) = u \vee f_*(v).\]
\end{definition}

If a sublocale inclusion $S \hookrightarrow X$ is open, we say $S$ is an \emph{open sublocale} of $X$. These are in order-preserving bijection with the elements of $\O X$. Moreover, the open sublocales are closed under arbitrary joins and finite meets. As one would expect, the image of an open sublocale under an open locale morphism is again an open sublocale (and the corresponding map of opens is given by the left adjoint).

If a sublocale inclusion $S \hookrightarrow X$ is proper, we say $S$ is a \emph{closed sublocale} of $X$. These are in order-reversing bijection with the elements of $\O X$. Indeed, open sublocales and closed sublocales are complements in the lattice of sublocales of $X$. Hence, closed sublocales are closed under arbitrary meets and finite joins. Proper morphisms are in particular \emph{closed} in the sense that they map closed sublocales to closed sublocales.

An open epimorphism is called an \emph{open surjection} and a proper epimorphism is called a \emph{proper surjection}.

Open morphisms and proper morphisms (and hence open and closed sublocales) are stable under pullback, as are open and proper surjections. Logically, open and proper morphisms are related to existential and universal quantification of opens.
Observe that if $X$ and $Y$ are sets and $\phi$ is a proposition on $X \times Y$ then the image of $\{(x,y) \in X\times Y \mid \phi(x,y)\}$ under $\pi_2\colon X \times Y \to Y$ is $\{y \in Y \mid \exists x \in X.\ \phi(x,y)\}$. Moreover, taking images is left adjoint to taking preimages. A similar situation arises if $\pi_2$ is an open map between locales and $\phi$ is viewed as an `open proposition'.  This logical interpretation is considered more in \cref{sec:dual}. For a more substantial discussion see \cite[Chapter 1]{manuellThesis}.
\begin{example}
    Recall that a locale $X$ is \emph{compact} if whenever there is a directed subset $\mathcal{U} \subseteq \O X$ for which $\bigvee \mathcal{U} = 1$, then $1 \in \mathcal{U}$. The unique map $! \colon X \to 1$ is proper if and only if $X$ is compact.
\end{example}
If $X$ is compact, then every product projection $\pi_2\colon X \times Y \to Y$ is proper and, viewing the elements of $\O X$ as closeds with the reversed order, we can understand $(\pi_2)_*$ as giving the images of these closeds. This can then be understood as existentially quantifying the closed propositions on $X \times Y$ over $X$, in a manner analogous to existentially quantifying the open propositions considered above.
\begin{definition}
    Since more creative terminology does not exist in the literature, we will be prosaic and say a morphism of locales $f\colon X \to Y$ has \emph{open diagonal} if the induced diagonal map to the pullback $X \to X \times_Y X$ is open. A morphism $f\colon X \to Y$ is said to be \emph{separated} if the diagonal map $X \to X \times_Y X$ is proper.

    If a morphism is both open and has open diagonal, it is said to be an \emph{étale} map or a \emph{local homeomorphism}.
\end{definition}
Local homeomorphisms can be understood as locale maps whose fibres are discrete. In particular, if ${!}\colon X \to 1$ is a local homeomorphism, then $X$ is a discrete locale --- that is, its frame is isomorphic to the powerset of a set.

Proper separated maps can similarly be viewed as locale maps whose fibres are compact Hausdorff locales. In particular, ${!}\colon X \to 1$ is proper and separated if and only if $X$ is compact Hausdorff. (Indeed, compare the usual definition of Hausdorff\-ness in terms of having a closed diagonal.)

\subsection{Internal categories}

The central objects of study in this paper are localic categories and groupoids. These are, respectively, internal categories and internal groupoids in the category of locales.
\begin{definition}\label{df:localic_category}
Let $\C$ be a category with all finite limits.
\begin{enumerate}
    \item An \emph{internal category} $\H$ in $\C$ is a diagram of the form
    \[\begin{tikzcd}
H_1 \times_{H_0} H_1 \ar[shift left = 4]{r}{\pi_{2}} \ar{r}{m} \ar[shift right = 4]{r}{\pi_{1}} & H_1  \ar[shift left = 4]{r}{t} \ar[shift right = 4]{r}{s} & \ar{l}[']{e} H_0 ,
\end{tikzcd}\]
satisfying the equations
\[s \circ e = t \circ e = \id_{H_0},\]
\[s \circ m = s \circ \pi_1, \ \ t \circ m = t \circ \pi_2,\]
\[m \circ (\id_{H_1} \times_{H_0} m) = m \circ (m \times_{H_0} \id_{H_1}), \]
\[m \circ (e \circ s, \id_{H_1}) = \id_{H_1} = m \circ (\id_{H_1}, e \circ t).\]
\item An \emph{internal groupoid} is an internal category with a further map $i \colon H_1 \to H_1$ satisfying the axioms:
    \[s \circ i = t, \ \ t \circ i = s, \]
    \[
    m \circ (\id_{H_1},i) = e \circ s, \ \ m \circ (i , \id_{H_1}) = e \circ t.
    \]
\end{enumerate}
\end{definition}
\begin{remark}
    In \cref{df:localic_category}, the morphisms $s$ and $t$ can be understood as taking the domain (or \emph{source}) and codomain (or \emph{target}) of an arrow, respectively. The map $e$ takes the identity arrow on an object, while the map $m$ takes the composite of two composable arrows. Note that in our convention the arguments of $m$ are given in `diagrammatic order'. For an internal groupoid, the map $i$ takes the inverse of an arrow.
\end{remark}
\begin{example}
    The internal categories of $\Set$ are just small categories.
\end{example}
\begin{example}
    Whenever $X$ is a locale, the diagram
    \[
    \begin{tikzcd}
    X \ar{r}{\id_X} \ar[shift left = 4]{r}{\id_X} \ar[shift right = 4]{r}{\id_X} & X \arrow[loop,  "\id_X "', distance=2em, in=305, out=235] \ar[shift left = 4]{r}{\id_X} \ar[shift right = 4]{r}{\id_X} & X \ar{l}[']{\id_X}
    \end{tikzcd}
    \]
    defines a localic groupoid, which is the  `categorically discrete' groupoid on $X$.
\end{example}
\begin{remark}
    We make the qualification that the above groupoid is discrete in the categorical sense because, given a set-based groupoid $X=(X_1 \rightrightarrows X_0)$, by taking the discrete locales corresponding to $X_1$ and $X_0$, we can also obtain a `topologically discrete' localic groupoid.
\end{remark}

There are also internal notions of functors and a natural transformations.
\begin{definition}
    Given two internal categories $\H$ and $\K$, an \emph{internal functor} $\Phi \colon \H \to \K$ consists of a pair of morphisms --- an action on objects $\Phi_0 \colon H_0 \to K_0$ and an action on arrows $\Phi_1 \colon H_1 \to K_1$ --- such that these commute with the structure of the internal categories, i.e.\ we have a commutative diagram:
   \[
        \begin{tikzcd}
	{H_1 \times_{H_0} H_1} & {K_1 \times_{K_0} K_1} \\
	{H_1} & {K_1} \\
	{H_0} & {K_0} \,.
	\arrow["{\Phi_0}", from=3-1, to=3-2, swap]
	\arrow["{\Phi_1}", from=2-1, to=2-2]
	\arrow["s"', shift right=3, from=2-1, to=3-1]
	\arrow["t", shift left=3, from=2-1, to=3-1]
	\arrow["t'", shift left=3, from=2-2, to=3-2]
	\arrow["s'"', shift right=3, from=2-2, to=3-2]
	\arrow["e'"{description}, from=3-2, to=2-2]
	\arrow["e"{description}, from=3-1, to=2-1]
	\arrow[dashed, from=1-1, to=1-2]
	\arrow["m"', from=1-1, to=2-1]
	\arrow["m'", from=1-2, to=2-2]
\end{tikzcd}
\]
\end{definition}
\begin{example}\label{ex:inclusion_of_identities}
    For every internal category $\H$, we can map the `categorically discrete' category on the object $H_0$ of objects of $\H$ to $\H$ via the internal functor $U \colon H_0 \to \H$, where $U_0 = \id_{H_0}$ and $U_1 = e$.
\end{example}

\begin{definition}\label{df:transformation}
    Given two internal functors $\Phi, \, \Psi \colon \H \rightrightarrows \K$ between internal categories, an \emph{internal transformation} $a \colon \Phi \Rightarrow \Psi$ is a morphism $a \colon H_0 \to K_1$ such that $s \circ a = \Phi_0$, $t \circ a = \Psi_0$, and the square
    \[\begin{tikzcd}
	{H_1} & {K_1 \times_{K_0} K_1} \\
	{K_1 \times_{K_0} K_1} & {K_1}
	\arrow["{(a \circ s, \Psi_1)}", from=1-1, to=1-2]
	\arrow["{(\Phi_1, a \circ t)}"', from=1-1, to=2-1]
	\arrow["m", from=1-2, to=2-2]
	\arrow["m"', from=2-1, to=2-2]
\end{tikzcd}\]
    commutes.  Intuitively, this expresses assigning an arrow $a_x \colon \Phi_0(x)  \to \Psi_0(x)$ in $K_1$ to each $x \in H_0$ such that this assignment is natural.
\end{definition}

We use $\Cat(\C)$ for the 2-category of internal categories, internal functors and internal transformations in $\C$, and $\Grpd(\C)$ for the restriction to internal groupoids.

\subsubsection{Discrete opfibrations}

We can impose some familiar properties on internal functors. See \cref{def:fullyfaithful} for the notions of full faithfulness and internal surjectivity of internal functors.
Here we consider discrete opfibrations.
\begin{definition}
    An internal functor $\Phi\colon \H \to \K$ is a \emph{discrete opfibration}
    if the square
    \[\begin{tikzcd}
	{H_1} & {H_0} \\
	{K_1} & {K_0}
	\arrow["s", from=1-1, to=1-2]
	\arrow["{\Phi_1}"', from=1-1, to=2-1]
	\arrow["{\Phi_0}", from=1-2, to=2-2]
	\arrow["s"', from=2-1, to=2-2]
    \end{tikzcd}\]
    is a pullback square.
\end{definition}

Discrete opfibrations into a $\Set$-category $\H$ correspond to functors from $\H$ into $\Set$. Analogously, we can think of internal discrete opfibrations as being a way to formalise the idea of a functor from an internal category to the category it lives in.

The word `discrete' in the term `discrete opfibration' refers the categorical discreteness of the fibres of $\Phi$. When working in $\Loc$ it is also possible to impose that $\Phi_0$ is a local homeomorphism, which is a topological discreteness condition on the fibres. See \cref{sec:toposes_as_spaces} for a different perspective on the conjunction of these two discreteness conditions.

We note that discrete opfibrations are stable under pullback in the 2-category of internal categories.

\subsubsection{The core groupoid}

Just as with ordinary categories, we can construct the \emph{core} $\H_{\cong}$ of an internal category $\H$, i.e.\ the internal groupoid of invertible morphisms.  Explicitly, $(H_{\cong})_0 = H_0$, while $(H_{\cong})_1$ is the regular subobject of $H_1 \times H_1$ defined by the equations
\[
t\circ \pi_1 = s\circ \pi_2, \, t\circ \pi_2 = s \circ \pi_1, \, m =  e\circ s \circ \pi_1, \, m \circ (\pi_2,\pi_1) = e \circ t \circ \pi_1.
\]
Intuitively, a morphism of $\H_{\cong}$ is a pair, where the first component is a morphism in $\H$ and the second component is its inverse. The structure maps of $\H_{\cong}$ are defined in the obvious way.

This defines the right adjoint to the inclusion of the 1-category of internal groupoids into the 1-category of internal categories. This 1-adjunction can be upgraded to a 2-adjunction by restricting to the sub-2-category $\Cat(\C)_{\cong}$ of $\Cat(\C)$ consisting of internal categories, internal functors and \emph{invertible} internal transformations --- that is, there is a natural isomorphism of categories $\Cat(\C)_{\cong}(\G,\H) \cong \Grpd(\C)(\G,\H_{\cong})$ for internal groupoids $\G$ and internal categories $\H$.

\subsection{Descent theory}

The concept of \emph{descent} will play an important role in this paper.  Descent theory was introduced by Grothendieck as a generalisation of the sheaf condition \cite{grothendieck_descent,grothendieck_descent_2}.  For an elementary introduction to descent theory motivated by topology, the reader is directed to \cite{facets_of_descent}.
From our perspective, descent can be understood in terms of discrete opfibrations.

Let $f\colon X \to Y$ be a morphism in a finitely complete category. Let $E \rightrightarrows X$ denote the kernel pair of $f$ (i.e.\ the pullback of $f$ along itself). This equips $X$ with the structure of an internal groupoid (in fact, an internal equivalence relation). Viewing $Y$ as a categorically discrete internal category, $f$ becomes a functor in a natural way.

\begin{definition}
    A morphism $f\colon X \to Y$ in a finitely complete category is said to be an \emph{effective descent morphism} if pullback along $f$ induces an equivalence of categories between the discrete opfibrations over $Y$ (viewed as a discrete category) and the discrete opfibrations over $X$ viewed as the kernel equivalence relation of $f$.
\end{definition}

That is to say, $f\colon X \to Y$ is an effective descent morphism if $Y$-indexed families are the same as $X$-indexed families that respect the equivalence relation induced by $f$. This is clearly a kind of surjectivity condition on $f$.

\begin{remark}
    An internal discrete opfibration $\K \to \H$ is the same data as an $H_1$-action on a bundle $q \colon K_0 \to H_0$ (see \cref{def:sheaf_over_localic_cat}).  Such an action can equivalently be repackaged as a \emph{descent datum}: a morphism $\theta \colon s^* K_0 \to t^* K_0$ satisfying $e^* \theta = \id_{K_0}$ and $m^* \theta = \pi_2^* \theta \circ \pi_1^* \theta$ (where $s^*$ denotes the pullback along the source map $s \colon H_1 \to H_0$, and similarly for other maps).  Given an $H_1$-action $\beta \colon K_0 \times_{H_0} H_1 \to K_0$, the corresponding descent datum $\theta_\beta \colon s^* K_0 \to t^* K_0$ is, in point-set notation, simply the map $(x,g) \mapsto (\beta(x,g),g)$.  For a detailed account of translating between discrete opfibrations/actions and descent data, see \cite[Appendix A]{manuell2023representing}.
\end{remark}

Effective descent morphisms satisfy a number of desirable properties: they are regular epimorphisms and they are preserved under pullback and composition (for further discussion, see \cite[\S B1.5]{Elephant}).
\begin{example}
Both open and proper surjections are of effective descent in the category of locales (see \cite{joyal1984galois,vermeulen1994proper}).
\end{example}

If $f\colon X \twoheadrightarrow Y$ is an effective descent morphism, we say that a pullback-stable class of morphisms \emph{descends} along $f$ if whenever the pullback of a morphism $p\colon Z \to Y$ along $f$ lies in the class, so does $p$ itself. This gives a restriction of the equivalence between the discrete opfibrations over $X$ and $Y$ to a certain subclass of discrete opfibrations and can be useful for proving such properties about maps into $Y$.

\subsection{Anafunctors}
Recall that the axiom of choice is equivalent to the theorem that every fully faithful functor which is essentially surjective on objects is part of an equivalence of categories. When working constructively or internal to a category that does not satisfy the axiom of choice (such as $\Loc$) it is often better to force this theorem to hold anyway by changing our morphisms between (internal) categories.
The idea is to \emph{localise} with respect the desired equivalences.

The \emph{localisation} of a category $\C$, or indeed a bicategory, at a class of morphisms $\Sigma \subseteq \C$ is a pseudofunctor $\C \to \C[\Sigma^{-1}]$ that is universal amongst all pseudofunctors that send the morphisms of $\Sigma$ to isomorphisms/equivalences. In other words, the localisation universally inverts the morphisms of $\Sigma$. In general localisations can be quite complicated, but they have a simple description in terms of spans when $\Sigma$ defines a \emph{calculus of fractions} (see \cite{gabriel_zisman,pronk1996bicategories,roberts2012anafunctors} for details).
\begin{definition}\label{def:fullyfaithful} Let $\H$ and $\K$ be internal categories.
\begin{enumerate}
    \item An internal functor $\Phi\colon \H\to \K$ is \emph{fully faithful} if the commutative square
  \[\begin{tikzcd}
	{H_1} & {K_1} \\
	{H_0 \times H_0} & {K_0 \times K_0}
	\arrow["{(s,t)}"', from=1-1, to=2-1]
	\arrow["{\Phi_1}", from=1-1, to=1-2]
	\arrow["{(s,t)}", from=1-2, to=2-2]
	\arrow["{\Phi_0 \times \Phi_0}"', from=2-1, to=2-2]
	\arrow["\lrcorner"{anchor=center, pos=0.125}, draw=none, from=1-1, to=2-2]
  \end{tikzcd}\]
   is a pullback.
  
    \item We say an internal functor $\Phi\colon \H \to \K$ is \emph{surjective on objects} if $\Phi_0 \colon H_0 \to K_0$ is an effective descent morphism.

    \item We call an internal functor that is both fully faithful and surjective on objects a \emph{surjective weak equivalence}.
\end{enumerate}
\end{definition}
\begin{remark}
    Note that the properties of being surjective on objects and fully faithful are both stable under taking pullbacks of internal functors; the former because effective descent morphisms are stable under pullback, and the latter because pullbacks commute with pullbacks.
\end{remark}
\begin{proposition}[{\cite[Theorem 7.1]{roberts2012anafunctors}, cf.\ also \cite[\S 7]{moerdijk1988}}]
    The class of fully faithful and surjective-on-objects functors form a right calculus of fractions on $\Cat(\C)$.
\end{proposition}

Therefore, by \cite{pronk1996bicategories} we can localise $\Cat(\C)$ at the class of surjective weak equivalences to obtain a bicategory whose morphisms are spans of functors where the left leg of the span is a surjective weak equivalence (see also the construction in \cite[\S 5]{roberts2012anafunctors}).

\begin{definition}\label{df:anacat}
    We denote the bicategory obtained by localising of the bicategory $\Cat(\C)$ at the class of surjective weak equivalences by $\AnaCat(\C)$.
    \begin{enumerate}
        \item The objects of $\AnaCat(\C)$ are internal categories.
        \item The arrows of $\AnaCat(\C)$ are called \emph{anafunctors}. An anafunctor from $\H$ to $\K$ is a span of internal functors $\H \xleftarrow{\Xi} \widetilde{\H} \xrightarrow{\Phi} \K$ where $\Xi$ is fully faithful and surjective on objects. We write $(\Xi,\Phi) \colon \H \slashedrightarrow \K$ to emphasise that $(\Xi, \Phi)$ is an internal anafunctor instead of an ordinary internal functor.
        \item Let $(\Xi_1,\Phi_1) \colon \H \slashedrightarrow \K$ and $(\Xi_2,\Phi_2) \colon \H \slashedrightarrow \K$ be a pair of anafunctors, and let $\Xi'_1$ and $\Xi'_2$ denote the respective pullback functors in the square
        \[
        \begin{tikzcd}
            \widetilde{\H}_1 \times_\H \widetilde{\H}_2 \ar{d}[']{\Xi'_2} \ar{r}{\Xi_1'} & \widetilde{\H}_2 \ar{d}{\Xi_2} \\
            \widetilde{\H}_1 \ar{r}{\Xi_1} & \H .
        \end{tikzcd}
        \]
        Following \cite[\S 9]{tommasini}, a 2-cell $(\Xi_1, \Phi_1) \Rightarrow (\Xi_2,\Phi_2)$ is defined as an equivalence class of pairs $(\Sigma,\tau)$ consisting of a surjective weak equivalence $\Sigma \colon \widetilde{\H}_3 \to \widetilde{\H}_1 \times_\H \widetilde{\H}_2$ and an internal transformation $\tau \colon \Phi_1 \circ \Xi'_2 \circ \Sigma \Rightarrow \Phi_2 \circ \Xi'_1 \circ \Sigma$.  Two such pairs $(\Sigma,\tau), (\Sigma',\tau')$ are equivalent if there exists a pair of surjective weak equivalences $\Omega \colon \widetilde{\H}_4 \to \widetilde{\H}_3$ and $\Omega' \colon \widetilde{\H}_4 \to \widetilde{\H}'_3$ and an internal natural isomorphism $\sigma \colon \Sigma' \circ \Omega' \Rightarrow \Sigma \circ \Omega$ such that
        \[
        \tau \Omega \circ \Phi_1 \Xi_2' \sigma = \Phi_2 \Xi'_1 \sigma \circ  \tau' \Omega'.
        \]     
    \end{enumerate}
    See \cite{pronk1996bicategories} and \cite{tommasini} for more details about the bicategorical structure.
    We denote by $\AnaGrpd(\C)$ the restriction of $\AnaCat(\C)$ to internal groupoids.
\end{definition}

\begin{remark}
 For simplicity, we have localised at fully faithful and surjective-on-objects functors. It would perhaps have been more intuitive to localise at the fully faithful and \emph{essentially} surjective functors, but by \cite[Proposition 8.1]{roberts2012anafunctors} the two constructions coincide.
\end{remark}

\begin{example}
    Every functor $\Phi\colon \H \to \K$ gives an anafunctor $(\id_\H, \Phi)\colon \H \slashedrightarrow \K$.
\end{example}
\begin{remark}\label{rem:ff_induced_domain}
  For an anafunctor $(\Xi,\Phi) \colon \H \slashedrightarrow \K$,
  the central category $\widetilde{\H}$ and the functor $\Xi$ are (essentially) determined by the underlying map of objects $\Xi_0\colon \widetilde{H}_0 \to H_0$ since, by the full faithfulness of $\Xi$, $\widetilde{H}_1$ and $\Xi_1$ can be constructed as the pullback of $\Xi_0 \times \Xi_0$ along $(s,t)\colon H_1 \to H_0 \times H_0$, while the composition of $\widetilde{\H}$ is also given by pulling back the composition of $\H$. Moreover, it can be shown that, if $\H$ is a groupoid, then $\widetilde{\H}$ must be a groupoid too, as we would expect for a category admitting a fully faithful functor into a groupoid.
\end{remark}

In our situation, the 2-cells in the bicategory of fractions $\AnaCat(\C)$ have a simpler description (see \cite{roberts2021elementary}).
\begin{proposition}\label{prop:2-cells_are_simplified}
    Let $(\Xi_1, \Phi_1)$ and $(\Xi_2,\Phi_2)$ be anafunctors from $\H$ to $\K$. Every 2-cell in $\AnaCat(\C)$ from $(\Xi_1, \Phi_1)$ to $(\Xi_2,\Phi_2)$ has a unique representative of the form $(\id_{\widetilde{\H}_1 \times_\H \widetilde{\H}_2},\tau)$. Thus, we may identify such 2-cells with internal transformations $\tau \colon \Phi_1 \circ \Xi'_2 \Rightarrow \Phi_2 \circ \Xi'_1$.
\end{proposition}
\begin{proof}
 Let $(\Sigma,\tau')$ be a representative of a 2-cell $(\Xi_1, \Phi_1) \Rightarrow (\Xi_2,\Phi_2)$. We claim that the underlying map of the transformation $\tau' \colon (\widetilde{H}_3)_0 \to K_1$ equalizes the kernel pair of $\Sigma_0 \colon (\widetilde{H}_3)_0 \to (\widetilde{H}_1)_0 \times_{H_0} (\widetilde{H}_2)_0$. Then since $\Sigma_0$ is, in particular, a regular epimorphism, $\tau'$ factors through $\Sigma_0$ to give a map $\tau\colon (\widetilde{H}_1)_0 \times_{H_0} (\widetilde{H}_2)_0 \to K_1$.

 We can see why $\tau'$ equalizes the kernel pair of $\Sigma_0$ by reasoning using points. Consider a pair of points $y,y' \in (\widetilde{H}_3)_0$ in the same fibre of $\Sigma_0$. Since $\Sigma$ is fully faithful, there is a (unique) morphism $y \to y'$ in $\widetilde{\H}_3$ that $\Sigma_0$ sends to the identity $\Sigma_0(y) = \Sigma_0(y')$. Then the naturality condition expressed in \cref{df:transformation} ensures that there is a commutative square
 \[
  \begin{tikzcd}
    (\Phi_1)_0 (\Xi'_2)_0 \Sigma_0(y) \ar[equals]{r} \ar{d}[']{\tau'(y)} & (\Phi_1)_0 (\Xi'_2)_0 \Sigma_0(y') \ar{d}{\tau'({y'})} \\
    (\Phi_2)_0 (\Xi'_1)_0 \Sigma_0(y) \ar[equals]{r} & (\Phi_2)_0 (\Xi'_1)_0 \Sigma_0(y'),
  \end{tikzcd}
 \]
 internal to $\K$, from which we conclude that $\tau'(y) = \tau'(y')$, as desired.

 The map $\tau\colon (\widetilde{H}_1)_0 \times_{H_0} (\widetilde{H}_2)_0 \to K_1$ defines an internal transformation $\tau\colon \Phi_1 \circ \Xi'_2 \Rightarrow \Phi_2 \circ \Xi'_1$. To see that the necessary naturality condition is satisfied, note that composing the desired commutative square from \cref{df:transformation} for $\tau$ with $\Sigma_1$ yields the corresponding square for $\tau'$, which commutes by assumption, and then use that $\Sigma_1$ is an epimorphism (since $\Sigma_0$ is a pullback-stable regular epimorphism and $\Sigma$ is fully faithful).

 We conclude that $(\id_{\widetilde{\H}_1 \times_\H \widetilde{\H}_2},\tau)$ also represents a 2-cell $(\Xi_1, \Phi_1) \Rightarrow (\Xi_2,\Phi_2)$. In fact, it lies in the same equivalence class as $(\Sigma, \tau')$, as is witnessed by the triple $(\Sigma, \Id_{\widetilde{\H}_3}, \id_\Sigma)$.

 We now show that such a representative $\tau$ is unique. Let $(\id_{\widetilde{\H}_1 \times_\H \widetilde{\H}_2},\rho)$ be another representative of the 2-cell under consideration. This means that there are weak equivalences $\Omega, \Omega'\colon \widetilde{\H}_4 \to \widetilde{\H}_1 \times_\H \widetilde{\H}_2$ and an internal natural isomorphism $\sigma\colon \Omega \to \Omega'$ such that $\tau \Omega \circ \Phi_1 \Xi_2' \sigma = \Phi_2 \Xi'_1 \sigma \circ  \rho \Omega'$. By the interchange law, we have $\tau \Omega \circ \Phi_1 \Xi_2' \sigma = \rho \Omega \circ \Phi_1 \Xi'_2 \sigma$ and so cancelling the isomorphism $\sigma$, we arrive at $\tau \Omega = \rho \Omega$. Finally, since $\Omega_0$ is an epimorphism we conclude that $\tau = \rho$, completing the proof.
\end{proof}
\begin{corollary}\label{prop:ff_on_2cells}
    The localisation pseudofunctor $\Cat(\C) \to \AnaCat(\C)$ that sends a functor $\Phi \colon \H \to \K$ to $(\id_\H,\Phi)$ is full and faithful on 2-cells.
\end{corollary}

\subsubsection{Localising respects the core}

Recall that for an internal category $\H$, we can form its core groupoid $\H_{\cong}$. We briefly argue that localising respects taking the core.
\begin{lemma}
    Let $\G$ and $\H$ be, respectively, an internal groupoid and an internal category.  Let $\AnaCat(\C)(\G,\H)_{\cong}$ denote the subcategory of $\AnaCat(\G,\H)$ consisting of only the invertible internal transformations.  Then there is an equivalence $\AnaCat(\C)(\G,\H)_{\cong} \simeq \AnaGrpd(\C)(\G,\H_{\cong})$ which is pseudonatural in $\G$ and $\H$.
\end{lemma}
\begin{proof}
    We describe the functors inducing the equivalence, but omit the calculation that these indeed give an equivalence of categories and the proof of pseudonaturality.  In the backward direction, we send an anafunctor $\G \slashedrightarrow \H_{\cong}$ to the composite
    \[\begin{tikzcd}
    \widetilde{\G} & {\H_{\cong}} & \H \\
    \G
    \arrow[from=1-1, to=1-2]
    \arrow["\sim"{anchor=south, rotate=90}, from=1-1, to=2-1]
    \arrow[from=1-2, to=1-3]
    \end{tikzcd}\]
    and a 2-cell, which by \cref{prop:2-cells_are_simplified} can be equated with a natural transformation $a \colon \widetilde{G}_0 \times_{G_0} \widetilde{G}'_0 \to (H_{\cong})_1$, is sent to its composite with the inclusion $(H_{\cong})_1 \to H_1$.
    
    For the forward direction, suppose we are given an anafunctor $\G \xleftarrow{\sim} \widetilde{\G} \to \H$. Note that since $\widetilde{\G}$ admits a fully faithful functor to an internal groupoid, $\widetilde{\G}$ itself will also be an internal groupoid.  Hence, by the universal property of the core, $\widetilde{\G} \to \H$ factors as $\widetilde{\G} \to \H_{\cong} \to \H$.  We send $\G \xleftarrow{\sim} \widetilde{\G} \to \H$ to the anafunctor $\G \xleftarrow{\sim} \widetilde{\G} \to \H_{\cong}$.  If $a \colon \widetilde{G}_0 \times_{G_0} \widetilde{G}'_0 \to H_1$ is an invertible natural transformation, representing a 2-cell of anafunctors, then $a$ factors through $(H_{\cong})_1 \to H_1$, thus yielding the action on arrows.
\end{proof}

\subsection{Toposes}

Underpinning and motivating many of our constructions is topos theory.  Because our focus will not be on the topos-theoretic consequences of our results (which were already well-known independently) we will not describe toposes in excessive detail here.  Instead, we will only recall the pertinent facts for this paper, and direct the interested reader to \cite{Elephant,SGL} for a broader introduction.  We use the word `topos' to unambiguously mean a Grothendieck topos, and we will use $\Topos$ to denote the bicategory of toposes, geometric morphisms and natural transformations between the inverse image functors.

Toposes can be understood as generalising the category of sets, while retaining many of the desirable properties of $\Set$, such as a rich internal logic.  Two aspects of toposes will be important for understanding our paper: the first is toposes as `generalised spaces', in which the category $\Set$ plays the role of the one-point space; the second is toposes as `mathematical universes', in which $\Set$ plays the role of a \emph{ground model} (cf.\ \cite[\S 14]{jech}).

\subsubsection{Toposes as generalised spaces}\label{sec:toposes_as_spaces}

We will make use of the fact that every locale $H$ yields a topos: namely, the slice category ${\bf LH}/H$ of local homeomorphisms over $H$ describes a topos, the \emph{topos of sheaves} on the locale $H$, which we denote by $\Sh(H)$.  Equivalently, this topos can be described as the category of functors $\O H \to \Set$ satisfying the \emph{sheaf condition} (see \cite[\S II]{SGL}).  In this manner, topos theory subsumes locale theory in the sense that this describes a fully faithful functor $\Sh \colon \Loc \hookrightarrow \Topos$.
\begin{example}
    A local homeomorphism over the one-point space $1$ is the same datum as a discrete locale or set, and so $\Sh(1) \simeq \Set$.
\end{example}
While not every topos is localic, i.e.\ of the form $\Sh(H)$ for some locale $H$, the extension of locales to all toposes is not so vast: in \cite{joyal1984galois}, Joyal and Tierney showed that every topos can be represented by a localic category in the following manner.
\begin{definition}\label{def:sheaf_over_localic_cat}
    Let $\H$ be a localic category.  A \emph{sheaf} over $\H$ consists of
    \begin{enumerate}
        \item a local homeomorphism $p \colon A \to H_0$,
        \item and an \emph{action} $\beta \colon A \times_{H_0} H_1 \to A$ (where the $A \times_{H_0} H_1$ is the pullback of $p$ and $s$) satisfying the equations (in point-set notation)
        \[
        p(\beta(x,g)) = t(g), \ \ \beta(\beta(x, g), h) = \beta(x, m(g,h)), \ \ \beta(x,ep(x)) = x.
        \]
    \end{enumerate}
    A morphism of sheaves $(A,p,\beta) \to (A',p',\beta')$ is a locale morphism $f \colon A \to A'$ such that $p = p' \circ f$ and $\beta(f(x),g) = f(\beta'(x,g))$. Together, this data defines a topos $\Sh(\H)$, the \emph{topos of sheaves} over $\H$.

    We note that this notion of sheaf over $\H$ is just a different way of describing a discrete opfibration over $\H$ whose underlying locale map is a local homeomorphism, and a morphism of sheaves corresponds to a functor between the domains of the discrete opfibrations making the triangle commute.

    In this way a sheaf over a categorically discrete localic category reduces to the usual notion of sheaf over on a locale and a sheaf over a topologically discrete localic category reduces to the usual notion of presheaf on a small category.

    The assignment of localic categories to their categories of sheaves is pseudofunctorial in that there is a pseudofunctor $\Sh \colon \LocCat \to \Topos$ which acts as follows.
    \begin{enumerate}
        \item For an internal functor $\Phi \colon \H \to \G$, the corresponding geometric morphism $\Sh(\Phi) \colon \Sh(\H) \to \Sh(\G)$ has, as its inverse image, the functor which sends a sheaf $(A,p,\beta)$ over $\G$ to the pullback
        \[\begin{tikzcd}
        {\Sh(\Phi)^* A} & A \\
        H_0 & G_0
        \arrow[from=1-1, to=1-2]
        \arrow[from=1-1, to=2-1]
        \arrow["\lrcorner"{anchor=center, pos=0.125}, draw=none, from=1-1, to=2-2]
        \arrow["p", from=1-2, to=2-2]
        \arrow["\Phi_0"', from=2-1, to=2-2]
        \end{tikzcd}\]
        endowed with the obvious $H_1$-action $\beta'$ defined, in point-set notation, by the equation
        \[
        \beta'((s(h),y), h) = (t(h),\beta(y,h)).
        \]
        \item For an internal transformation $a \colon \Psi \Rightarrow \Phi$ with underlying map $a \colon H_0 \to G_1$, the corresponding natural transformation $\Sh(\Phi)^\ast \Rightarrow \Sh(\Psi)^\ast$ has, as its component at $A \in\Sh(\G)$, the morphism $a_A \colon \Sh(\Phi)^\ast A \to \Sh(\Psi)^\ast A$ given, in point-set notation, by $a_A(x,y) = (x,\beta(a(x),y))$.
    \end{enumerate}
\end{definition}

\begin{example}
    Restricting $\Sh \colon \LocCat \to \Topos$ to the categorically discrete localic categories results in the usual embedding of locales into toposes $\Sh \colon \Loc \hookrightarrow \Topos$.
\end{example}

The landmark result of Joyal and Tierney \cite{joyal1984galois} asserts that every topos is of the form $\Sh(\H)$ for some localic category $\H$.  In fact, it says something stronger: every topos is of the form $\Sh(\H)$ for some localic \emph{groupoid} $\H$ which is both \emph{open}, meaning that the source and target morphisms $H_1 \rightrightarrows H_0$ are open, and moreover \emph{étale-complete} in the following sense.
\begin{definition}
    A localic groupoid $\G$ is \emph{étale-complete} if the square
    \[
    \begin{tikzcd}
        \Sh(G_1) \ar{r}{\Sh(s)} \ar{d}[']{\Sh(t)} \ar[draw = none]{rd}[anchor = center, pos = 0.125]{\lrcorner} & \Sh(G_0) \ar{d}{\Sh(U)} \\
        \Sh(G_0) \ar{r}[']{\Sh(U)} & \Sh(\G)
    \end{tikzcd}
    \]
    is a bipullback of toposes, where $U$ is the canonical functor defined in \cref{ex:inclusion_of_identities}.
\end{definition}
\begin{remark}\label{rem:etale_completion}
    Each generalised point $\alpha \colon X \to G_1$, with source and target $x$ and $y$ respectively, yields a natural isomorphism
    \[\begin{tikzcd}
	{\Sh(X)} & {\Sh(G_0)} \\
	{\Sh(G_0)} & {\Sh(\G)},
	\arrow["{\Sh(x)}", from=1-1, to=1-2]
	\arrow["{\Sh(y)}"', from=1-1, to=2-1]
	\arrow["\cong"{description}, draw=none, from=1-2, to=2-1]
	\arrow["{\Sh(U)}", from=1-2, to=2-2]
	\arrow["{\Sh(U)}"', from=2-1, to=2-2]
    \end{tikzcd}\]
    but in general there may be more natural isomorphisms $\Sh(U \circ x) \cong \Sh(U \circ y)$ than are instantiated by points of $G_1$.  Intuitively, a localic groupoid $\G$ is étale-complete when there is no difference between the natural isomorphisms `seen' by the topos $\Sh(\G)$ and the isomorphisms contained in the groupoid $\G$.

    As explained in \cite[\S 7.2]{moerdijk1988},
    every localic groupoid $\G = (G_1 \rightrightarrows G_0)$ can be completed to an \'{e}tale-complete one with the same locale of objects while taking the new locale of arrows to be the locale $\hat{G}_1$ for which the square
    \[
    \begin{tikzcd}
        & \Sh(\hat{G}_1) \ar{r} \ar{d} \ar[draw = none]{rd}[anchor = center, pos = 0.125]{\lrcorner} & \Sh(G_0) \ar{d}{\Sh(U)} \\
        & \Sh(G_0) \ar{r}[']{\Sh(U)} & \Sh(\G) 
    \end{tikzcd}
    \]
    is a (bi-)pullback.

    Expressing étale-completeness for localic groupoids requires consideration of (bi-)pullbacks in $\Topos$ since the corresponding square
    \[
    \begin{tikzcd}
         G_1 \ar{r}{s} \ar{d}[']{t} & G_0 \ar{d}{U} \\
        G_0 \ar{r}[']{U} & \G
    \end{tikzcd}
    \]
    in $\LocGrpd$ is \emph{always} a (bi)pullback (cf.\ \cite[Remark 5.3]{wrigley_25}).
\end{remark}

Let $\H$ and $\G$ be localic groupoids.  When $\G$ is étale-complete, the geometric morphisms between the sheaf toposes $\Sh(\H) \to \Sh(\G)$ can be described entirely in terms of the localic groupoids: they correspond to anafunctors $\H \slashedrightarrow \G$.  As shown in \cite[\S 7]{moerdijk1988}, we can use this observation and the localisation result of \cite[\S 1]{gabriel_zisman} to establish an equivalence between toposes and étale-complete localic groupoids with anafunctors.  Moreover, using the techniques of \cite{pronk1996bicategories}, this equivalence can be upgraded to a bi-equivalence:
\begin{theorem}[Theorem 7.7 \cite{moerdijk1988}, \cite{pronk1996bicategories}]
    The pseudofunctor $\Sh$ induces a bi-equivalence $\Topos_{\cong} \simeq \LocGrpdAna_{\text{\rm \'{E}C}}$, where $\LocGrpdAna_{\text{\rm \'{E}C}}$ is the restriction of $\LocGrpdAna$ to \'{e}tale-complete localic groupoids.
\end{theorem}

\subsubsection{Toposes as mathematical universes}

The second way in which we will use toposes is as `{mathematical universes}'.  It is well known that toposes possess rich enough structure to internalise all of higher order intuitionistic logic (see \cite[\S VI]{SGL}, \cite[\S II]{lambek_scott} or \cite[\S 14]{goldblatt}).  Thus, toposes behave as `universes' in which to perform constructive mathematics.

Just as reasoning in a point-like fashion about locales can simplify proofs as discussed in \cref{sec:reasoning_with_points}, we can perform a mathematical argument in the internal logic of topos that would otherwise be unwieldy or unfeasible when working externally.  Once the internal construction is obtained, we can `externalise' the result to obtain a valid external construction without needing to manipulate the explicit, external objects involved.  We give a simple example of working internally below --- further examples can be found, for instance, in \cite{ingo_phd}.
\begin{example}
    Let $X$ be an object of a topos $\topos$, and let ${\sim} \subseteq X \times X$ be an internal partial equivalence relation, i.e.\ an internal relation, which, in the internal logic of $\topos$, satisfies the axioms
    \begin{align*}
       & \forall x,y \in X.\ x \sim y \implies y \sim x, \\
       & \forall x,y,z \in X.\ (x\sim y) \land (y\sim z) \implies x \sim z.
    \end{align*}
    Then we can form the \emph{subquotient} $X / {\sim} \in \topos$, i.e.\ the quotient by (the restriction of) $\sim$ of the subobject $\{x \in X \mid x \sim x\} \subseteq X$.

    If we now suppose that $\sim$ is \emph{inhabited}, i.e.\ it satisfies $\exists x \in X.\ x \sim x$, and that it satisfies the axiom $\forall x, y \in X.\ (x \sim x) \land (y \sim y) \implies x \sim y$ expressing that all pairs in the subobject $\{x \in X \mid x \sim x\} \subseteq X$ are related to each other, then it is clear from the internal logic that there is a single equivalence class.  Externally, this entails that $X/ {\sim}$ is isomorphic to the terminal object.
\end{example}

In this paper, we will only ever need to work in the internal logic of the topos $\Sh(X)$ of sheaves on a locale $X$, which can be understood as performing $\O X$-valued set theory using the complete Heyting algebra structure on $\O X$ (see \cite{Fourman1979,bell}).  Unless $\O X$ is a complete Boolean algebra, the topos $\Sh(X)$ does not satisfy the law of the excluded middle, and so we must ensure that our arguments are constructive.  Particularly important will be the theory of internal locales of $\Sh(X)$, for which we recall the following:
\begin{enumerate}
    \item The objects of $\Sh(X)$, i.e.\ the `internal sets', are given by local homeomorphisms $ p \colon E \to X$.  This is one sense in which a local homeomorphism is a `discrete space' over the codomain.
    \item The category of internal locales in $\Sh(X)$ is equivalent to the slice category $\Loc/X$ (see \cite[Theorem C1.6.3]{Elephant}).
    \item Finally, the subcategory of internal compact Hausdorff locales of $\Sh(X)$ is identified as the subcategory ${\bf PS}/X$ of proper separated bundles over $X$.
\end{enumerate}

\section{Generalised presentations and classifying locales}\label{sec:gen_prens}

To prove the `dual' results of this paper we will make use of a very general notion of frame presentation. Before we describe this general construction we will first review the usual notion of presentation.

As algebraic structures, frames can be presented by generators and relations, with the frame obtained from a presentation $\langle G \mid R\rangle$ being the free frame on $G$ subject to the relations $R$. While relations are traditionally taken to be equalities, it is also possible to take them to be inequalities, and this is often more convenient. From a logical perspective, we can interpret the generators as atomic propositions and the relations as axioms in propositional geometric logic. The presented frame is then the Lindenbaum--Tarski algebra for the propositional geometric theory (see, for instance, \cite[\S 2]{vickers2007locales}). We say the corresponding locale is the \emph{classifying locale} for the theory. In particular, the points of the classifying locale correspond to models of the theory.
Moreover, generalised points with domain $A$ correspond to models in the sheaf topos $\Sh(A)$.

Let us look at some examples.

\begin{example}[Sierpiński space]\label{ex:sierpinski}
The theory consisting of only one proposition and no axioms yields the free frame on one generator.
The classifying locale $\Srpnsk$ is the Sierpiński space and its points correspond to truth values --- classically, $\bot$ and $\top$.

By the universal property of the free frame on one generator, locale maps $X \to \Srpnsk$ correspond to opens of $X$. The order on these locale maps agrees with the usual order on the opens.
\end{example}

\begin{example}[The Dedekind reals]\label{ex:dedekindreals}
Recall the construction of the real numbers by Dedekind cuts. A Dedekind cut on $\Q$ is given by a pair $(L,U)$ of subsets of $\Q$ satisfying certain axioms. A Dedekind cut $(L,U)$ represents a real that is larger than the rationals in the `lower cut' $L$ and smaller than the rationals of the `upper cut' $U$.
The theory of Dedekind cuts can be expressed in geometric logic by taking an atomic proposition with the name $[p \in L]$ for each $p \in \Q$, an atomic proposition $[p \in U]$ for each $p \in \Q$ and the following axioms.
\begin{displaymath}
\begin{array}{r@{\hspace{1.5ex}}c@{\hspace{1.5ex}}l@{\quad}@{}l@{\qquad\quad}r@{}}
  {[q \in L]} &\vdash& {[p \in L]} & \text{ for $p \le q$} & \text{($L$ downward closed)} \\
  {[q \in L]} &\vdash& \bigvee_{p > q} {[p \in L]} & \text{ for $q \in \Q$} & \text{($L$ rounded)} \\
   &\vdash& \bigvee_{q \in \Q} {[q \in L]} && \text{($L$ inhabited)} \\
  {[p \in U]} &\vdash& {[q \in U]} & \text{ for $p \le q$} & \text{($U$ upward closed)} \\
  {[q \in U]} &\vdash& \bigvee_{p < q} {[p \in U]} & \text{ for $q \in \Q$} & \text{($U$ rounded)} \\
   &\vdash& \bigvee_{q \in \Q} {[q \in U]} && \text{($U$ inhabited)} \\
  {[p \in L]} \land {[q \in U]} &\vdash& \bot & \text{ for $p \ge q$} & \text{($L$ and $U$ disjoint)} \\
   &\vdash& {[p \in L]} \lor {[q \in U]} & \text{ for $p < q$} & \text{(locatedness)}
\end{array}
\end{displaymath}
Note that a sequent $\phi \vdash \psi$ is interpreted as saying that $\psi$ holds whenever $\phi$ does. If $\phi$ is missing it is understood to be $\top$ (i.e.\ true).
A model of such a theory assigns a truth value to each atomic proposition such that the sequents are satisfied. In this case, such a model gives a Dedekind cut $(L,U)$ where $L$ is the set of $p \in \Q$ for which $[p \in L]$ is true and $U$ is the set of $p \in \Q$ for which $[p \in U]$ is true.

The classifying locale of this theory is just the usual locale of reals. The frame presentation is obtained by taking a generator for each atomic proposition and a relation for each axiom. Each sequent $\phi \vdash \psi$ is interpreted as an \emph{inequality} $\phi \le \psi$. It can be made to be an equality by using the lattice structure: $\phi \wedge \psi = \psi$.
\end{example}

\begin{example}[Cantor space]
Consider the theory with two basic propositions, $Z_n$ and $O_n$, for each natural number $n \in \N$, and the axioms $\top \vdash Z_n \lor O_n$ and $Z_n \land O_n \vdash \bot$ for each $n$. This describes infinite sequences of Boolean values and the resulting locale is the Cantor space $2^\N$.
\end{example}

We also note that imposing an additional relation of the form $\top \vdash \phi$ to a presentation results in an open sublocale of the original locale, while adding a relation of the form $\phi \vdash \bot$ results in a closed sublocale.

\subsection{Generalised presentations}

In the usual notion of presentation of a frame, sets play a distinguished role, since we start with a \emph{set} of generators and a \emph{set} of relations.
In order to maintain a parallel between the construction of classifying locales for geometric and for dual geometric theories, it will be natural to consider a dual kind of presentation where the generators and relations are not sets, but more general locales, often compact Hausdorff locales.

In \cite{manuellThesis,manuell2022spectrum}, the first author introduced a way to present frames by \emph{locales} of generators and relations, called \emph{generalised presentations}. It is simplest when the locales of generators and relations are locally compact (as is the case when they are sets or compact Hausdorff locales).

Let us start by reviewing the usual notion of presentation. In the standard situation where generators and relations are indexed by sets $G$ and $R$, the frame $\langle G \mid R\rangle$ is the initial quotient of the free frame $\langle G\rangle$ on $G$ such that the relations indexed by $R$ hold. These relations are specified by two functions from $R$ into $\langle G\rangle$ --- explicitly, we map each relation ``$\phi \sim \psi$'' in $R$ (where $\phi , \psi \in \langle G \rangle$) to $\phi$ and $\psi$ respectively. Alternatively, by the universal property of free frames, they can be specified by a pair of frame homomorphisms from $\langle R\rangle$ to $\langle G\rangle$. The presented frame, $\langle G \mid R\rangle$, is the coequaliser of these frame homomorphisms.

Now recall that, since the free frame on one generator is the frame of opens of Sierpiński space $\Srpnsk$, the free frame $\langle G\rangle$ is the frame of opens of $\Pi_{g \in G} \Srpnsk = \Srpnsk^G$ (because, as a left adjoint, the free frame functor preserves coproducts). The locale corresponding to the presented frame is then obtained by an equaliser \[X \hookrightarrow \Srpnsk^G \rightrightarrows \Srpnsk^R.\]

\begin{example}
 Consider the presentation $\langle a, b \mid a \wedge b = 0, a \vee b = 1 \rangle$.
 Here $G = \{a,b\}$ and $R$ is a two-element set, say $\{\triangle,\square\}$, indexing the relations. The frame maps corresponding to $\Srpnsk^G \rightrightarrows \Srpnsk^R$ in this case are then defined (using the universal property of the free frame) by
 \begin{minipage}{0.4\textwidth}
 \begin{align*}
  \triangle &\mapsto a \wedge b \\
  \square &\mapsto a \vee b
 \end{align*}
 \end{minipage}
 \begin{minipage}{0.1\textwidth}
 and
 \end{minipage}
 \begin{minipage}{0.4\textwidth}
 \begin{align*}
  \triangle &\mapsto 0 \\
  \square &\mapsto 1.
 \end{align*}
 \end{minipage}
\end{example}

At this point we reinterpret the powers $\Srpnsk^G$ and $\Srpnsk^R$ as exponentials (viewing the sets $G$ and $R$ as discrete locales). Indeed, any discrete locale $G$ can be written a coproduct $\bigsqcup_{g \in G} 1$ in $\Loc$ (where here $G$ is viewed as an indexing set) and the exponential $\Srpnsk^{(-)}$ sends coproducts to products, so that the exponential $\Srpnsk^G$ is the same as $\prod_{g \in G} \Srpnsk$. Now by treating $G$ and $R$ as exponents, a similar construction can be carried out for arbitrary exponentiable (i.e.\ locally compact by \cite{hyland1981exponentials}) locales $G$ and $R$, not just discrete locales.
It is perhaps easiest to understand the maps $\Srpnsk^G \to \Srpnsk^R$ in terms of their exponential transposes $\Srpnsk^G \times R \to \Srpnsk$ as $R$-indexed families of opens of $\Srpnsk^G$.

\begin{remark}
 It will often be convenient to replace the (co)equaliser above with a (co)inserter so as to describe the basic relations in terms of inequalities instead of equalities, but the argument is essentially the same. (A coinserter of two frame maps takes the largest quotient so that the one map compares less than another; see \cite[Example B1.1.4(d)]{Elephant}. This can be used to enforce relations involving inequalities which commonly arise from the interpretation in terms of axioms in geometric logic.)
\end{remark}

\begin{remark}\label{rem:presentation_base_change}
    Note that since change of base preserves limits and exponentials of exponentiable locales (see \cite[Lemma A1.5.2(ii)]{Elephant}) it preserves generalised presentations. Thus, generalised presentations describe the same locale in every (localic) topos. %
\end{remark}

\subsection{Dual presentations}\label{sec:dual}

Let us consider generalised presentations in the case that the locales of generators and relations are compact Hausdorff. This case will be explored in more detail in future work.

First suppose there are no relations. Then $\Srpnsk^G$ (for a compact Hausdorff locale $G$) is the locale of maps from $G$ to $\Srpnsk$, which can be viewed as either the locale of open sublocales of $G$ or the locale of closed sublocales of $G$ (since $\Srpnsk$ classifies both of these and they are in correspondence by taking complements). While with standard presentations we take the former perspective, here it will be more convenient to take the latter one.

\begin{example}[Closed partial equivalence relations]\label{ex:closed_partial_equiv_rels}
Throughout this subsection, we will consider the following illustrative example: a description of the locale of closed partial equivalence relations on a compact Hausdorff locale $Y$. The locale of generators is $G = Y \times Y$. Before imposing any relations, this gives the locale $\Srpnsk^{Y \times Y}$ which classifies \emph{all} closed binary relations on $Y$.
\end{example}

Relations are described by maps $\Srpnsk^G \rightrightarrows \Srpnsk^R$, or equivalently, maps $\Srpnsk^G \times R \rightrightarrows \Srpnsk$, or even $R \rightrightarrows \Srpnsk^{\Srpnsk^G}$. Each point $g$ of $G$ corresponds to a basic closed $[g \in F]$ in $\Srpnsk^G$ (under the inclusion into the double dual $G \to \Srpnsk^{\Srpnsk^G}$) consisting of the closed sublocales $F$ of $G$ that contain $g$. Maps from $R$ to $\Srpnsk^{\Srpnsk^G}$ that factor through $G \to \Srpnsk^{\Srpnsk^G}$ then describe relations that force one basic closed to compare less than or equal to another.  For instance, if there were one relation ``$g \leqslant h$'', then the pair of maps would be given by
\[
\begin{tikzcd}
    R= 1 \ar[shift left]{r}{g} \ar[shift right]{r}[']{h} & G \ar{r} & \Srpnsk^{\Srpnsk^G}.
\end{tikzcd}
\]

\begin{excont}[Continued]
In our example, each point $(x,y) \in Y$ corresponds to a basic closed $[(x,y) \in E]$ consisting of the (closed) relations which contain the pair $(x,y)$. We can force the relations to be symmetric by imposing the condition $[(x,y) \in E] = [(y,x) \in E]$ (or equivalently, $[(x,y) \in E] \le [(y,x) \in E]$) for all $x,y \in Y$. This corresponds to taking $R = Y \times Y$ and taking the equaliser of the maps $\Srpnsk^{Y \times Y} \rightrightarrows \Srpnsk^{Y \times Y}$, whose corresponding maps $Y \times Y \rightrightarrows \Srpnsk^{\Srpnsk^{Y \times Y}}$ are given by the canonical inclusion into the double dual on the one hand, and on the other, the composite of this inclusion with the swap map $Y \times Y \to Y \times Y$, i.e.\ the map $(y_1, y_2) \mapsto (y_2, y_1)$.
\end{excont}

We can build up more complicated relations from these simpler ones using that fact that $\Srpnsk$ is an internal distributive lattice. (Note that we are referring \emph{not} to the lattice structure of the frame of opens of $\Srpnsk$, but to an internal lattice structure in $\Loc$.) Recall from \cref{ex:sierpinski} that, working classically, the two points of $\Srpnsk$ describe the two-element distributive lattice $\{\bot,\top\}$.  The join operation on these points, corresponding to logical `or', and the meet operation, corresponding to logical `and', yield morphisms $\Srpnsk^2 \rightrightarrows \Srpnsk$ that describe the structure of the internal distributive lattice on $\Srpnsk$.  If we were working intuitionistically, then we still obtain an internal distributive lattice structure on $\Srpnsk$, using that the points of $\Srpnsk$ correspond to the truth values in the ambient topos.

Since closeds have the reverse ordering compared to opens, the meet operation on $\Srpnsk$ induces a \emph{union} operation on the basic closeds, while the join operation induces an \emph{intersection} on basic closeds. These allow us to impose that an expression $\phi$, built from finite meets and finite joins of basic closeds, compares less than another such expression $\psi$, i.e.\ that $\phi \leqslant \psi$. (We can also handle the constants $0$ and $1$ in a similar way.)

\begin{excont}[Continued]
Returning to our example, let us consider how to impose that our relations are transitive. We want $[(x,y) \in E] \wedge [(y,z) \in E] \le [(x,z) \in E]$. Now $R = Y \times Y \times Y$. The right-hand side is represented by the composite of the projection onto the outer factors $Y \times Y \times Y \xrightarrow{\pi_{1,3}} Y \times Y$ and the map into the double dual $Y \times Y \to \Srpnsk^{\Srpnsk^{Y\times Y}}$. The map corresponding to the left-hand side is obtained by taking two similar maps where the projection is onto the first two factors and the last two factors respectively, pairing them to give a map into $\Srpnsk^{\Srpnsk^{Y\times Y}} \!\! \times \Srpnsk^{\Srpnsk^{Y\times Y}}$ and then composing with the map $(\vee)^{\Srpnsk^{Y\times Y}}$. (Note that we use $\vee$, not $\wedge$, to represent the meet due to the fact the order on closeds is reversed compared to the usual order on $\Srpnsk$. These join and meet operations from $\Srpnsk$ also agree with joins and meets computed using the order enrichment.) Finally, we take the \emph{inserter} of the maps representing the left-hand side and right-hand side. Again because of the order reversal, we actually force the left-hand-side map to become \emph{larger} than the right-hand side with respect to the usual order-enrichment on $\Loc$.

Note that these relations can now be combined with the symmetry ones by taking $R = Y \times Y \sqcup Y \times Y \times Y$ and using the universal property of the coproduct. Now we must use an inserter in both cases (or alternatively an equaliser in both cases). This then gives a locale of symmetric, transitive (closed) relations on $Y$, or (closed) partial equivalence relations.
\end{excont}

In addition to expressions built up using finitary operations, given a map $\Srpnsk^G \times R \times I \to \Srpnsk$, i.e.\ a closed of $\Srpnsk^G \times R \times I$, (where $I$ is also compact Hausdorff) we can ``existentially quantify'' over $I$ to give a map $R \rightrightarrows \Srpnsk^{\Srpnsk^G}$, i.e.\ a closed of $\Srpnsk^G \times R$. By ``existentially quantifying over $I$'', we mean taking the image of the closed sublocale under the proper projection $\Srpnsk^G \times R \times I \to \Srpnsk^G \times R$. (We can also think of this as a join indexed by the compact Hausdorff locale $I$.)

It might help to understand what the analogue of this is for usual (discrete) presentations.
In such a setting we might consider a join $\bigvee_{i \in I} a_i$ of certain expressions $a_i$ representing opens in $\Srpnsk^G$. The family of $a_i$'s can be represented as a function $a\colon \Srpnsk^G \times I \to \Srpnsk$ so that $a_i$ is the preimage of $\top$ under $a(-,i)$. The entire family is given by the open $a^*(\top)$ in $\Srpnsk^G \times I$. Then $\bigvee_{i \in I} a_i$ is equal to the image of $a^*(\top)$ under the projection $\pi_1\colon \Srpnsk^G \times I \to \Srpnsk^G$. This image is open, since here the projection is an open map. This is exactly how existential quantification over $I$ would be defined in categorical logic.

Our case is entirely dual. Instead of using open sublocales we use closed sublocales. The product projection is closed since the locale $I$ is compact. See also the expository note \cite{escardo_compactly_many} where it is explained how to make sense of intersections of opens indexed by compact spaces. (Our situation is obtained from that one by taking complements.)

\begin{remark}\label{rem:distributivity_over_generalised_joins}
  We note that finite meets distribute over such joins indexed by compact Hausdorff locales. This is precisely the Frobenius reciprocity condition from \cref{def:open_and_proper}. Indeed, it is well-known in categorical logic that Frobenius reciprocity is responsible for the equivalence between the expressions $\phi(y) \wedge \exists x.\ \psi(x,y)$ and $\exists x.\ (\phi(y) \wedge \psi(x,y))$. Here, since we are treating the existential quantification as a big join, we understand Frobenius reciprocity as a distributivity condition.
\end{remark}

Finally, if $R = Y \times Y$ we also have a map $\Srpnsk^G \times R \to \Srpnsk$ corresponding to the closed sublocale $\Srpnsk^G \times Y \hookrightarrow \Srpnsk^G \times Y \times Y$ induced by the diagonal on $Y$, which can be understood as corresponding to equality.

\begin{example}\label{ex:partial_surjs_KHaus}
At this point let us consider a second illustrative example. Let $Y$ and $Z$ be a compact Hausdorff locales. We can describe a locale of partial surjections from $Y$ to $Z$ in terms of closed relations. The points will be closed sublocales $F$ of $Y \times Z$ such that $[(y,z) \in F] \wedge [(y,z') \in F] \le \llbracket z = z' \rrbracket$ and $1 \le \bigvee_{y \in Y} [(y, z) \in F]$. The first condition is functionality/single-valuedness, while the second is surjectivity/right-totality.

Let us consider the second condition first. Here $G = Y \times Z$. Let $R = Z$ and let $\Srpnsk^G \times Y \times R \to \Srpnsk$ correspond to inclusion into the double dual $Y \times Z \to \Srpnsk^{\Srpnsk^{Y \times Z}}$. As above we may project out the compact Hausdorff locale $Y$ to give a map $\Srpnsk^G \times R \to \Srpnsk$ and hence a map $\Srpnsk^G \to \Srpnsk^R$. This corresponds to the right-hand side of the inequality. The left-hand side is given by the map induced by the constant $\bot$ map $\Srpnsk^G \times R \to 1 \xrightarrow{\bot} \Srpnsk$ (again due to the order reversal). The coinserter of these two maps gives the locale of right-total (closed) relations on $Y$.

We can also construct a locale of single-valued relations on $Y \times Z$. Here $G = Y \times Z$ and $R = Y \times Z \times Z$. The left-hand side is represented similarly to before as a composite 
\[\Srpnsk^G = \Srpnsk^{Y \times Z} \xrightarrow{(\Srpnsk^{\pi_{1,2}}, \Srpnsk^{\pi_{1,3}})} \Srpnsk^{Y \times Z \times Z} \times \Srpnsk^{Y \times Z \times Z} \xrightarrow{(\vee)^{\Srpnsk^{Y \times Z \times Z}}} \Srpnsk^{Y \times Z \times Z} = \Srpnsk^R.\]
The right-hand side is given by $\Srpnsk^{Y \times Z} \xrightarrow{{!}} 1 \xrightarrow{d} \Srpnsk^{Z \times Z} \xrightarrow{\pi_{2,3}} \Srpnsk^{Y \times Z \times Z}$ where $d$ corresponds to the closed diagonal of $Z \times Z$. The inserter of these maps gives the locale of single-valued relations.  By combining this with the right-totality construction above we obtain the locale of partial surjections from $Y$ to $Z$.
\end{example}

The category of compact Hausdorff locales is a pretopos (see \cite{karazeris2021}). Using the internal logic of a pretopos we can interpret what it means for a (closed) relation to satisfy a coherent formula. If we use the above to form a presentation using this coherent formula, the points of the resulting locale are precisely the relations satisfying the formula.
The upcoming work \cite{manuell2026presentations} will explain this in more detail. %

In this way we can construct a dual version of classifying locales for a coherent theory. We can further extend this to dual geometric theories
by taking arbitrary joins in the order-enrichment and recalling that these corresponds to arbitrary meets of the closed sublocales.

\begin{remark}\label{rem:beyond_dual}
In this paper we will sometimes need to consider presentations whose generators and relations are indexed by potentially infinitely many compact Hausdorff spaces. This takes us beyond what would strictly be consider dual presentations.

There are two approaches to dealing with this. Firstly, we can obtain a generalised presentation where the locales of generators and relations are coproducts of compact Hausdorff locales. This is possible since the coproduct of compact Hausdorff locales is still locally compact. We will still use the conventions and constructions of dual presentations. In particular, we think in terms of closeds instead of opens. The basic logical operations do not actually depend in any way on being compact Hausdorff. We do need to ensure that equality is only taken between variables ranging over Hausdorff locales and that we only take joins indexed by compact locales, but this will always be true in the cases we consider.
We will call these conventions, where we use the `closed orientation' for connectives, \emph{closed-type generalised presentations}.

Alternatively, we can avoid dealing with locally compact locales of generators and relations by considering equalisers involving \emph{products} of objects of the form $\Srpnsk^{G_\alpha}$ where each $G_\alpha$ is locally compact. Again, all of the constructions involving dual presentations are easily seen to generalise to this setting.
\end{remark}

\subsection{The canonical presentation of a compact Hausdorff locale}\label{sec:canonical_presentation}

A discrete locale $X$ can be presented by generators $[=x]$ for each $x \in X$ and relations $[=x] \wedge [=y] \vdash x = y$ for $x,y \in X$ and $\top \vdash \bigvee_{x \in X} [=x]$.
To understand why, note that the free frame on $X$ corresponds to the locale $\Srpnsk^X$ whose points are (open) subsets of $X$. The condition $[=x] \wedge [=y] \vdash x = y$ is satisfied by those subsets in which any two elements are equal, while the condition $\top \vdash \bigvee_{x \in X} [=x]$ forces the subsets to be inhabited. Thus, these relations together restrict us to precisely the singleton opens of $X$, which can be identified with the points of $X$ itself.

In summary, we have an equaliser diagram
\[X \hookrightarrow \Srpnsk^X \rightrightarrows \Srpnsk^{X \times X \sqcup 1}\]
which expresses the discrete locale $X$ in terms of itself.
Using dual presentations we can do something similar for compact Hausdorff locales.

\begin{lemma}\label{lem:canonical_prens}
    Let $X$ be a compact Hausdorff locale. Then $X$ admits a dual presentation with generators $[=x]$ for $x \in X$ and relations $[=x] \wedge [=y] \vdash x = y$ for $x,y \in X$ and $\top \vdash \bigvee_{x \in X} [=x]$.
    More explicitly, we have an equaliser diagram
    \[\begin{tikzcd}
	X & {\Srpnsk^X} &&& {\Srpnsk^{X \times X} \times \Srpnsk^1},
	\arrow[hook, ""{name=0, anchor=center, inner sep=0}, "i", from=1-1, to=1-2]
	\arrow[""{name=0, anchor=center, inner sep=0}, "{(\Srpnsk^{\pi_1} \vee \Srpnsk^{\pi_2},\, \bot \circ {!})}", shift left=2, from=1-2, to=1-5]
	\arrow[""{name=1, anchor=center, inner sep=0}, "{(d \circ {!},\, (\Srpnsk^{!})_* )}"', shift right=2, from=1-2, to=1-5]
    \arrow["\le"{marking, allow upside down}, draw=none, from=1, to=0]
\end{tikzcd}\]
    where $i\colon X \to \Srpnsk^X$ and $d\colon 1 \to \Srpnsk^{X \times X}$ are various exponential transposes of the map $1 \times X \times X \to \Srpnsk$ corresponding to the closed diagonal of $X$, while $(\Srpnsk^{!})_*$ denotes the right adjoint of $\Srpnsk^{!}\colon \Srpnsk^1 \to \Srpnsk^X$ and corresponds to taking images of the closeds. %
\end{lemma}
\begin{proof}
    It suffices to show that the generalised points of the equaliser agree with those of $X$.
    The points of the equaliser are the closed sublocales $S$ of $X$ that make two inequalities hold.

    The first inequality asks that $S \times X \cap X \times S \subseteq \Delta$ where $\Delta$ is the diagonal of $X$. Using the internal logic, it is now easy to see that ${!}\colon S \to 1$ is monic. Alternatively and more explicitly, consider the commutative square
    \[\begin{tikzcd}
	S & {S \times S} \\
	X & {X \times X}.
	\arrow["\Delta"', from=1-1, to=1-2]
	\arrow["s"', hook', from=1-1, to=2-1]
	\arrow["{\pi_i}"', shift right=2, from=1-2, to=1-1]
	\arrow["k"{pos=0.4}, dashed, from=1-2, to=2-1]
	\arrow["{s \times s}", hook', from=1-2, to=2-2]
	\arrow["\Delta"', from=2-1, to=2-2]
    \end{tikzcd}\]
    By hypothesis, $s \times s$ factors through $\Delta\colon X \to X \times X$ to give a monomorphism $k$. Then $\Delta k \Delta = (s \times s) \Delta = \Delta s$ and so the upper triangle commutes. For both projections $\pi_{1,2}\colon S \times S \to S$ we have $\pi_i \Delta = \id_S$. We now show that $\Delta \pi_i = \id_{S \times S}$. Consider $k \Delta \pi_i = s \pi_i = \pi_i (s \times s) = \pi_i \Delta k = k$ and thus $\Delta \pi_i = \id_S$. Hence, the square
    \[\begin{tikzcd}
	S & S \\
	S & 1
	\arrow[equals, from=1-1, to=1-2]
	\arrow[equals, from=1-1, to=2-1]
	\arrow["\lrcorner"{anchor=center, pos=0.125}, draw=none, from=1-1, to=2-2]
	\arrow["{{!}}", from=1-2, to=2-2]
	\arrow["{{!}}"', from=2-1, to=2-2]
    \end{tikzcd}\]
    is a pullback and thus ${!}\colon S \to 1$ is monic.
    
    The second inequality asks that the image of $S$ under ${!}\colon S \to 1$ is the whole of $1$. This is to say that $S \to 1$ is a regular epimorphism in $\KHausLoc$.
    
    Putting these together we have that ${!}\colon S \to 1$ is an isomorphism and so $S$ is a singleton. Such sublocales are patently in bijection with the points of $X$ via the map $i\colon X \to \Srpnsk^X$.
\end{proof}
\begin{remark}
Recall that $\Srpnsk^X$, without any imposed relations, classifies open/closed sublocales of $X$.  Recall also that, tautologically, any locale classifies the theory of its (generalised) points.  Thus, \cref{lem:canonical_prens} can be understood as expressing that, when $X$ is a compact Hausdorff locale, the theory of points of $X$ is equivalent to the theory of closed singletons of $X$.
\end{remark}
\subsection{Presentations indexed by subquotients}
There are situations, such as in the proof of \cref{prop:each-generic-sort-is-LH}, where we wish to introduce `dummy' generators into a dual presentation. We explain in the following lemma how to make the locales of generators or relations of a standard or dual presentation `larger' without changing the presented frame.
\begin{lemma}\label{lem:expanded_presentation}
    Let $G_{0}, G_1, G_2$ and $R_0, R_1, R_2$ all be discrete locales or all be compact Hausdorff locales and suppose we have
    \[G_0 \twoheadleftarrow G_1 \hookrightarrow G_2 \qquad \text{and} \qquad R_0 \twoheadleftarrow R_1 \hookrightarrow R_2.\]
    Let \[X \hookrightarrow \Srpnsk^{G_0} \rightrightarrows \Srpnsk^{R_0}\]
    be a (generalised) presentation.
    Then there exists a (generalised) presentation of $X$ with generators $G_2$ and relations $R_2 \sqcup G_2^2 \sqcup G_2$.
\end{lemma}
    Before embarking on the proof, we outline the core ideas by describing what happens in the simple case where $G_0, G_1, G_2$ are discrete locales and $R_0, R_1, R_2$ are all empty, i.e.\ we wish to introduce `dummy' variables into a free frame presentation. In this case, it is easy to construct a frame presentation of the free frame $X = \langle G_0 \rangle$ with generators $G_2$ and relations $G_2^2 \sqcup G_2$ --- by moving from $G_0$ to $G_1$, we are duplicating the generators in $G_0$ (according to the surjection $q \colon G_1 \twoheadrightarrow G_0$) which we then desire to equate in the presentation, and by moving from $G_1$ to $G_2$ we are adding new generators that we want to annihilate.  Such desiderata are handled by imposing on the set of generators $G_2$ the relations
    \[
    \begin{array}{r@{\hspace{1ex}}c@{\hspace{1ex}}l@{\quad}@{}l@{}}
     \llbracket q(g) = q(g') \rrbracket \land g & \vdash & g' & \text{ for all $g, g' \in G_2$}, \\
     g & \vdash & \llbracket g \in G_1 \rrbracket & \text{ for all $g \in G_2$},
    \end{array}
    \]
    where $\llbracket g \in G_1 \rrbracket$ and $\llbracket q(g) = q(g') \rrbracket$ are interpreted as the truth value of the statement inside the brackets, e.g.\ if $g \in G_2$ is contained in the subset $G_1 \subseteq G_2$, then $\llbracket g \in G_1 \rrbracket = \top$, and if not, then $\llbracket g \in G_1 \rrbracket = \bot$.
    Intuitively, the first relation (if also applied a second time with $g$ and $g'$ exchanged) makes generators identified by $q$ equal to each other, while the second one `zeros out' the generators not included in $G_1$.
\begin{proof}[Proof of \cref{lem:expanded_presentation}]
    Let $E$ be the kernel equivalence relation of $q\colon G_1 \twoheadrightarrow G_0$ so that $E \rightrightarrows G_1 \xrightarrow{q} G_0$ is a coequaliser. Applying $\Srpnsk^{(-)}$ we arrive at the equaliser diagram \[\Srpnsk^{G_0} \xrightarrow{\Srpnsk^q} \Srpnsk^{G_1} \rightrightarrows \Srpnsk^{E}.\]

    Now note that $\Srpnsk^q$ has a retraction: if $G_0$ and $G_1$ are discrete then $q$ is an open surjection and so $\Srpnsk^q$ has a left adjoint left inverse $(\Srpnsk^q)_!$; if $G_0$ and $G_1$ are compact Hausdorff then $q$ is a proper surjection and so $\Srpnsk^q$ has a right adjoint left inverse $(\Srpnsk^q)_*$. We can use this to combine the above equaliser with the original presentation for $X$ to obtain an equaliser
    \[X \hookrightarrow \Srpnsk^{G_1} \rightrightarrows \Srpnsk^{R_0} \times \Srpnsk^E \cong \Srpnsk^{R_0 \sqcup E},\]
    where the maps from $\Srpnsk^{G_1}$ to $\Srpnsk^{R_0}$ come from the composite of the splitting $\Srpnsk^{G_1} \to \Srpnsk^{G_0}$ and the maps $\Srpnsk^{G_0} \rightrightarrows \Srpnsk^{R_0}$ in the original equaliser.

    Now as before, if $G_1$ and $G_2$ are discrete, the inclusion $i\colon G_1 \hookrightarrow G_2$ is open and so the quotient map $\Srpnsk^i$ has a left adjoint $(\Srpnsk^i)_!\colon \Srpnsk^{G_1} \hookrightarrow \Srpnsk^{G_2}$. This left adjoint is, in fact, the equaliser of the identity $\id \colon \Srpnsk^{G_2} \to \Srpnsk^{G_2}$ and the map $(\Srpnsk^{i})_!\Srpnsk^{i}$, which takes in the intersection of a given open with the open sublocale $G_1$.  Similarly, if $G_1$ and $G_2$ are compact Hausdorff, we have a right adjoint $(\Srpnsk^i)_*\colon \Srpnsk^{G_1} \hookrightarrow \Srpnsk^{G_2}$, which is the equaliser of the identity $\id \colon \Srpnsk^{G_2} \to \Srpnsk^{G_2}$ and the map $(\Srpnsk^{i})_*\Srpnsk^{i}$, which takes the intersection of a given closed with the closed sublocale $G_1$.

    Again we can combine this with the previous equaliser (using the retraction $\Srpnsk^i$) to give an equaliser
    \[X \hookrightarrow \Srpnsk^{G_2} \rightrightarrows \Srpnsk^{R_0 \sqcup E} \times \Srpnsk^{G_2} \cong \Srpnsk^{R_0 \sqcup E \sqcup G_2}.\]

    Now we can apply a similar argument to the relations to find monomorphisms $\Srpnsk^{R_0} \hookrightarrow \Srpnsk^{R_1}$ and $\Srpnsk^{R_1} \hookrightarrow \Srpnsk^{R_2}$. Moreover, the open/closed inclusion $E \hookrightarrow G_1 \times G_1 \hookrightarrow G_2 \times G_2$ yields a monomorphism $\Srpnsk^{E} \hookrightarrow \Srpnsk^{G_2 \times G_2}$. Putting these together we obtain a monomorphism
    $\Srpnsk^{R_0 \sqcup E \sqcup G_2} \hookrightarrow \Srpnsk^{R_2 \sqcup G_2^2 \sqcup G_2}$.

    Finally, since adding a monomorphism to the end of an equaliser diagram again gives an equaliser diagram we arrive at the generalised presentation
    \[X \hookrightarrow \Srpnsk^{G_2} \rightrightarrows \Srpnsk^{R_2 \sqcup G_2^2 \sqcup G_2},\]
    completing the proof.
\end{proof}

\section{Forcing with locales}\label{sec:forcing}

In this section, we state two local-theoretic `forcing' results we will need later in this paper.  The first expresses that, up to a forcing extension, `every set is subcountable', and employs the familiar locale of partial surjections between two sets (cf.\ \cref{ex:partial_surjs_KHaus} or \cite[Example C1.2.8]{Elephant}).  The second expresses that, up to a forcing extension, `every compact Hausdorff space is a subquotient of the Cantor space'.

Our explicit results in this section state that certain locales of maps are nontrivial in a strong sense, but they can be understood as positing the existence of such maps after passing to certain (localic) toposes, and these toposes can be interpreted as examples of set theorists' forcing extensions. For further discussion and examples of forcing in this fashion, the reader is directed to \cite{scedrov_forcing}.

The locale theorist will understand our constructions as showing (constructively) that a certain presented locale has a well-behaved surjection to 1 and hence that the corresponding fibrewise construction over a base is also a well-behaved surjection, namely an effective descent morphism.

The set theorist can think of our desired subquotients as being generic objects in a forcing extension. The notions of forcing we consider in this paper are very `tame' in that, apart from forcing our condition to be satisfied, we impose no constraints on whether other parts of the structure are preserved, aside from asking that things do not trivialise completely (corresponding to the morphism to 1 being a well-behaved surjection). As a result, we can be remarkably carefree with the instances of \emph{iterated} forcing we will encounter in \cref{thm:alexandroff_hausdorff}, as we are ensured of a non-trivialising forcing extension by the fact that effective descent morphisms compose.

\subsection{Every set is subcountable}

The first example of a forcing extension we will use is the familiar scenario where we wish to `force' the cardinality of a given set to be smaller (but where we do not care if other cardinalities collapse into each other).
\begin{example}[Partial surjections from $\N$ to $X$]\label{ex:partial_surjections_from_N}
Fix a set $X$ and consider the following geometric theory of partial surjections from $\N$ to $X$. There is an atomic proposition $[f(n) = x]$ for each $n \in \N$ and $x \in X$, which we interpret to mean that the partial function maps $n$ to $x$. The axioms are as follows (where $\{\top \mid x = y\}$ is understood as a subset of the singleton $\{\top\}$)\footnote{We use this strange-looking construction so the definition is constructively valid, since we will need to use it internal to a topos. If we could assume the law of the excluded middle, we either have $x = y$, in which case the join is $\top$ and the axiom is superfluous, or $x \ne y$ and the right-hand side of the axiom can be replaced with $\bot$.}.
\begin{displaymath}
\begin{array}{r@{\hspace{1ex}}c@{\hspace{1ex}}l@{\quad}@{}l@{\quad}r@{}}
  [f(n) = x] \land [f(n) = y] &\vdash& \bigvee \{\top \mid x = y\} & \text{for $n \in \N$, $x,y \in X$} & \text{(functionality)} \\
   &\vdash& \bigvee_{n \in \N} [f(n) = x] & \text{for $x \in X$} & \text{(surjectivity)}
\end{array}
\end{displaymath}
We denote the resulting classifying locale by $[\N \paronto X]$.
\end{example}
\begin{remark}\label{rem:inserter_diagram_for_NtoX}
    Note that, via the discussion in \cref{sec:gen_prens}, the locale $[\N \paronto X]$ satisfies the universal property of the inserter in the diagram
    \[
    [\N \paronto X] \hookrightarrow \Srpnsk^{X \times \N} \rightrightarrows \Srpnsk^{X \times X \times \N \sqcup X} 
    \]
    where the two maps $\Srpnsk^{X \times \N} \rightrightarrows \Srpnsk^{X \times X \times \N \sqcup X}$ are precisely analogous to those constructed in \cref{ex:partial_surjs_KHaus}.
\end{remark}
Thus, the locale $[\N \paronto X]$ has a basis consisting of those formal meets of generators $[f(n_1) = x_1] \land \dots \land [f(n_\ell) = x_\ell]$ that encode a finite partial function $f \colon \N \rightharpoondown X$.  In other words, a basic generator is a finite approximation of the partial surjection $f \colon \N \paronto X$ we are forcing, as is often the case in poset-theoretic forcing (cf.\ \cite[p.\ 267]{SGL}).

If the set $X$ is countable, the locale $[\N \paronto X]$ is not so strange: it is (classically) a spatial locale whose points are the partial surjections from $\N$ to $X$. However, if $X$ is uncountable, then there are no partial surjections from $\N$ to $X$ and thus the locale has no global points. 

Now consider the topos $\Sh([\N \paronto X])$ of sheaves on this locale.  The sets $X$ and $\N$ are lifted to this topos by applying the inverse image of the canonical geometric morphism $\gamma \colon \Sh([\N \paronto X]) \to \Set$, or equivalently, by taking the local homeomorphisms of locales $X \times [\N \paronto X] \to [\N \paronto X]$ and $\N \times [\N \paronto X] \to [\N \paronto X]$ given by the product projections. We write $\underline{X}$, $\underline{\N}$ to denote the respective liftings. (The latter is, of course, the internalisation of the naturals in $\Sh([\N \paronto X])$). Internal to the topos $\Sh([\N \paronto X])$, there is now an internal partial surjection $f \colon \underline{\N} \paronto \underline{X}$, even in the case where $X$ was originally uncountable.  In other words, we have `forced' $X$ to become \emph{subcountable}.

Explicitly, the graph of $f$ is given by the subobject of $\underline{X} \times \underline{\N} \in \Sh([\N \paronto X]) $, i.e.\ the open sublocale of $ X \times \N \times [\N \paronto X]$, obtained by adding the relation
\[
\top \vdash \bigvee_{x \in X, n \in \N} [=n] \land [=x] \land [f(n)=x]
\]
to the presentation.

We now observe that the locale $[\N \paronto X]$ is nontrivial, in the sense that the unique map ${!}\colon [\N \paronto X] \to 1$ is of effective descent. It is shown in \cite[Example C3.3.11(c)]{Elephant} that $[\N \paronto X]$ is connected and locally connected. The following proposition follows (see \cite[Corollary C3.3.2(ii)]{Elephant}).
\begin{proposition}\label{prop:NtoX_is_opensurj}
    The unique morphism ${!} \colon [\N \paronto X] \to 1$ is an open surjection, and hence of effective descent.
\end{proposition}
Explicitly, the left adjoint to ${!}^*$ acts by sending a formal meet of generators $[f(n_1) = x_1] \land \dots \land [f(n_\ell) = x_\ell]$ to $1 \in \O 1$ if and only if it encodes a finite partial function $f \colon \N \rightharpoondown X$.
\begin{remark}
    Although it may seem more natural to `force' an uncountable set $X$ to be a quotient of $\N$, rather than a subquotient, for our purposes it is simpler to consider the latter: there will be instances later on where we wish to treat empty sets and inhabited sets uniformly, and for that reason it is preferable to consider subquotients.  This is especially important since, when working internally to a topos that does not satisfy the law of excluded middle, `internal sets' cannot be cleanly divided into those that empty and those that are inhabited.
\end{remark}
Our decision to express \cref{ex:partial_surjections_from_N} in an entirely constructive manner is now rewarded by the fact that the locale of partial surjections from the naturals to a set can be \emph{internalised} in any topos.  In particular, a local homeomorphism of locales $X \to H$ yields an object in the topos $\Sh(H)$, i.e.\ an internal `set', and thus there exists an \emph{internal locale} $[\N \paronto_H X]$ of $\Sh(H)$ of partial surjections from the internalisation of the naturals $\underline{\N} \in \Sh(H)$ to the object $X \in \Sh(H)$.

Recall that the internal locales of $\Sh(H)$ are just the locales over $H$, i.e.\ objects of the slice category $\Loc/H$, and so the internal locale $[\N \paronto_H X]$ corresponds to an external locale, which we also denote by $[\N \paronto_H X]$, together with a morphism $[\N \paronto_H X] \to H$.  From the above observations, and the fact that \cref{prop:NtoX_is_opensurj} internalises, we obtain the following result, which will become instrumental in \cref{sec:univ_prop}.
\begin{proposition}\label{prop:every_set_is_subcount}
    For any local homeomorphism of locales $p \colon X \to H$, there exists a locale $[\N \paronto_H X]$ and an open surjection $[\N \paronto_H X] \twoheadrightarrow H$ such that the (generalised) points of $[\N \paronto_H X]$ are in bijection with pairs of a point $h$ of $H$ and a partial surjection $\N \paronto p^{-1}(h)$ onto the corresponding fibre $p^{-1}(h)$, i.e.\ the pullback locale
    \[
    \begin{tikzcd}
        p^{-1}(h) \ar{r} \ar{d} & X \ar{d}{p} \\
        U \ar{r}{h}  & H.
        \arrow["\lrcorner"{anchor=center, pos=0.125}, draw=none, from=1-1, to=2-2]
    \end{tikzcd}
    \]
\end{proposition}
\begin{remark}
    The locale $[\N \paronto_H X]$ can be presented in a similar fashion as in \cref{rem:inserter_diagram_for_NtoX}, where we take exponents in $\Loc/H$ of the \emph{lifting} of the Sierpi\'nski space, i.e.\ the product projection $H \times \Srpnsk \to H$; e.g.\ the locale of generators becomes $\underline{\Srpnsk}^{X \times \underline{\N}} \cong ({\underline{\Srpnsk}}^{\underline{\N}})^X \cong (H \times \Srpnsk^\N \to H)^X$, where the locale $(H \times \Srpnsk^\N \to H)^X$ denotes the exponent in $\Loc/H$, which exists since, as $X \to H$ is a local homeomorphism, it is exponentiable.
\end{remark}

\subsection{Every compact Hausdorff space is a subquotient of Cantor space}

Just as there is a sense that every set is a quotient of $\N$, we will show in this section that `every compact Hausdorff locale is a subquotient of Cantor space'. Note that if we restrict to countably generated frames, then this is strictly true --- the Alexandroff--Hausdorff theorem states that every non-empty countably-based compact Hausdorff space is the image of a map whose domain is the Cantor space. A classical proof of this in the pointfree setting has been given in \cite{avila2022cantor}. We will provide a constructive result that is true without cardinality assumptions.

This proof uses a fairly substantial amount of pointfree topology. Let us briefly describe the concepts involved with citations in case the reader is interested in learning more about them.

Firstly, we have the notion of \emph{complete regularity}. This is the pointfree incarnation of the familiar separation axiom from topology indicating there are sufficiently many maps to $\R$ to separate points. It is usually defined in terms of the \emph{completely below} relation $\pprec$ where $u \pprec v$ in $\O X$ indicates that there is a locale map $f\colon X \to [0,1]$ such that $u \le f^*((\tfrac{1}{2},1]) \le f^* ((0,1]) \le v$. The locale $X$ is completely regular if every open is the join of those completely below it. More details can be found in \cite[Chapters V and XIV]{PicadoPultr}.

Next we use the notions of \emph{density}, \emph{strong density} and \emph{positivity}. Dense locale maps are analogous to dense continuous maps of topological spaces. A locale map $f\colon X \to Y$ is dense if $f^*(u) = 0$ entails $u = 0$. Strong density is classically equivalent, but asks that $u > 0$ entails $f^*(u) > 0$, where $u > 0$ means that $u$ is positive, i.e.\ every open cover of $u$ is inhabited. Finally, $X$ is \emph{overt} if $X$ has a base of positive elements (which is always true classically). A locale $X$ is overt (and positive) if and only if the unique map $X \to 1$ is open (and surjective). See \cite[Section 2]{henry2016}, \cite[Part C]{Elephant} or \cite[Chapter 1]{manuellThesis} for more.

Finally, we need the concept of a \emph{triquotient map}. Triquotient maps are a large class of effective descent morphisms which contain open and proper surjections, as well as retractions. They are stable under pullback and composition. The canonical reference is \cite{plewe1997localic}.

\begin{theorem}\label{thm:alexandroff_hausdorff}
The locale of subquotients $[2^\N \paronto X]$ from $2^\N$ to a compact Hausdorff locale $X$ (see \cref{ex:partial_surjs_KHaus}) admits a triquotient surjection, and hence an effective descent morphism, to $1$.
\end{theorem}
\begin{proof}
    It follows from \cite[Proposition 2.3.17 \& Proposition 2.6.3]{henry2016} that there is a positive overt locale $L$ such that in $\Sh(L)$, (the base change of) the compact regular locale $X$ is completely regular. We now work in $\Sh(L)$.
    
    Let $S = \Hom_\Loc(X,[0,1])$. There is a map $e: X \to [0,1]^S$ induced from the family $(f\colon X \to [0,1])_{f \in S}$ by the universal property of the product. We claim the map $e$ is an embedding.
    
    This is true, because $X$ is completely regular. If $c \in \O X$ is a cozero element --- that is, $c = f^*( (0,1] )$ for some $f\colon X \to [0,1]$ --- then $c = e^* \pi_f^*( (0, 1] )$. But for every $u \in \O X$, we have $u = \bigvee\{ v \in \O X \mid v \pprec u \}$. Practically by definition, $v \pprec u$ gives that there is a cozero element $c$ such that $v \le c \le u$. Hence $u = \bigvee\{c \text{ cozero} \mid c \le u\}$. Note that, because we take all possible $c$, we avoid the axiom of choice.
    Thus the image of $e^*$ generates $\O X$ under joins, and hence $e^*$ is surjective, as required.
    
    Now we pass the topos of sheaves over the (positive, overt) locale $[\N \paronto S]$, so that $S$ becomes a subquotient of $\N$. Thus, we have $T \twoheadrightarrow S$ for some $T \subseteq \N$.  Since $[0,1]^{(-)}$ sends colimits to limits, we have $[0,1]^S \hookrightarrow [0,1]^T$, and so composing this with $e$ we have $e'\colon X \hookrightarrow [0,1]^T$ for $T \subseteq \N$.
    
    We claim that the map $[0,1]^i\colon [0,1]^\N \to [0,1]^T$ induced by $i\colon T \hookrightarrow \N$ is a (proper) quotient map. Since the locales involved are compact and Hausdorff, it suffices to show it is dense. In fact, we will show it is strongly dense. Note that, as a subset of $\N$, $T$ has decidable equality. By the `localic axiom of choice' from \cite[Proposition 2.3.7]{henry2016corrected}, we know that $[0,1]^\N$ and $[0,1]^T$ are overt and that the positive opens of $[0,1]^T$ are given by finite meets of elements $\pi_t^*(u_t)$ with distinct $t \in T$ and each $u_t > 0$. The preimage under $[0,1]^i$ sends $\pi_t^*(u_t)$ to $\pi_{i(t)}^*(u_t)$. Hence a finite meet $\bigwedge_i \pi_{t_i}^*(u_i)$ is sent to $\bigwedge_i \pi_{i(t_i)}^*(u_i)$. If each $t_i$ is distinct, so is each $i(t_i)$, since $i$ is injective. Thus, $[0,1]^i$ maps positive elements to positive elements and is hence strongly dense.
    
    So far we have shown that $X$ is a closed sublocale of a proper quotient of $[0,1]^\N$.
    Now recall that there is a proper surjection $2^\N \twoheadrightarrow [0,1]$ sending binary expansions to their corresponding real. (It is a surjection, since it sends sequences with finitely many ones to dyadic rationals, which are dense in $[0,1]$.)
    This induces a map $2^\N \cong 2^{\N \times \N} \cong (2^\N)^\N \twoheadrightarrow [0,1]^\N$, which is seen to also be a proper surjection using the argument of \cite[Theorem 6.6]{clementino1996topology} applied to the fact that proper surjections into $[0,1]^\N$ are closed under products in the slice category (see \cite[Theorem 4.5]{vermeulen1994proper}).
    
    We can compose this with $[0,1]^\N \twoheadrightarrow [0,1]^T$ to obtain a quotient $2^\N \twoheadrightarrow [0,1]^T$.
    
    In summary, $X$ is a closed sublocale of a proper quotient of $2^\N$.
    Taking the pullback
    
    \[\begin{tikzcd}
    	\bullet & {2^\N} \\
    	X & {[0,1]^T}
    	\arrow[hook, from=1-1, to=1-2]
    	\arrow[two heads, from=1-1, to=2-1]
    	\arrow["\lrcorner"{anchor=center, pos=0.125}, draw=none, from=1-1, to=2-2]
    	\arrow[two heads, from=1-2, to=2-2]
    	\arrow[hook, from=2-1, to=2-2]
    \end{tikzcd}\]
    we see that $X$ is also a proper quotient of a closed sublocale of $2^\N$ --- that is, a subquotient of $2^\N$. This yields a point of the locale $[2^\N \paronto X]$. Thus, in particular, the unique map $!\colon [2^\N \paronto X] \to 1$ is a triquotient.
    
    The above result has taken place in a topos $\Sh(L')$ where $L'$ is a locale admitting a composite of open surjections $L' \twoheadrightarrow L \twoheadrightarrow 1$. Externalising we have
    \[\begin{tikzcd}
    	{[2^\N \paronto X] \times L'} & {[2^\N \paronto X]} \\
    	L' & 1
    	\arrow[two heads, from=1-1, to=1-2]
    	\arrow[two heads, from=1-1, to=2-1]
    	\arrow["\lrcorner"{anchor=center, pos=0.125}, draw=none, from=1-1, to=2-2]
    	\arrow[from=1-2, to=2-2]
    	\arrow[two heads, from=2-1, to=2-2]
    \end{tikzcd}\]
    where the double-headed arrows are all triquotients
    and hence $!\colon [2^\N \paronto X] \to 1$ is a triquotient (see \cite[Proposition 1.8]{plewe1997localic} or alternatively use that triquotients descend along open surjections).
\end{proof}

Note that as in the discrete case, we can also form a locale $[2^\N \paronto_{H} X]$ of fibrewise partial surjections for a proper separated map $X \to H$. This can be found by externalising the construction of $[2^\N \paronto X]$ internal to $\Sh(H)$.

\section{Syntactic constructions as bundles}\label{sec:syntax}

In this section, we turn to the interpretation of logic in bundles over a localic category.  We proceed as follows:
\begin{enumerate}
    \item Firstly, we discuss the two fragments of infinitary first-order logic, \emph{geometric} and \emph{dual geometric}, that we consider in this paper.
    \item We then recall the interpretation of geometric and dual geometric logic in certain bundles over a localic category.
    \item In \cref{sec:base_change}, we recall how this interpretation is pseudofunctorial with respect to the base localic category.
    \item Finally, we construct some important examples of models over localic categories.  Namely, for each (dual) geometric theory, we construct a localic category and a model of the theory over this category which, as will be shown in \cref{sec:univ_prop}, is \emph{generic} in a suitable sense.
\end{enumerate}

\subsection{Geometric logic and dual geometric logic}\label{sec:logic}

We assume familiarity with the syntax of (infinitary) first-order logic (Section D1 of \cite{Elephant} is a good introduction).  We highlight the two fragments of infinitary first-order logic considered in this paper:
    \begin{enumerate}
        \item The first is \emph{geometric logic}, which is the fragment of infinitary first order logic involving finite conjunction, infinitary disjunction, equality and existential quantification, i.e.\ the symbols $\{\,\land,\bigvee,\exists,=\,\}$, as studied extensively by topos theorists.  Recall that, in geometric logic, finite conjunctions distribute over arbitrary joins.

        \item The latter is \emph{dual geometric logic}, where we can take only finite joins but infinitary meets.  This is the fragment of infinitary first order logic involving the symbols $\{\,\bigwedge,\lor,\exists,=\,\}$.  We include in our syntax the infinitary distributivity axiom
        \[
        \varphi \lor \bigwedge_{i \in I} \psi_i \dashv \vdash_{\vec{x}} \bigwedge_{i \in I} (\varphi \lor \psi_i).
        \]
    \end{enumerate}
Because geometric and dual geometric logic lack an implication symbol, we employ a `Gentzen' style sequent calculus, where a  \emph{sequent} $\phi \vdash_{\vec{x}} \psi$ is understood to mean ``for all $\vec{x}$, if $\phi(\vec{x})$ then $\psi(\vec{x})$''.  A theory $\theory = \{\,\varphi_i \vdash_{\vec{x}_i} \psi_i \mid i \in I\,\}$ over a signature $\Sigma$ is a collection of such sequents, and is said to be \emph{geometric} (respectively, \emph{dual geometric}) if the formulae $\phi_i, \psi_i$ involved in each axiom of $\theory$ are geometric (resp., dual geometric).

For convenience, we will restrict to signatures only involving relation symbols.  This is possible since a function symbol can be replaced by a relation symbol and axioms expressing that the relation is the graph of the function (see \cite[Lemma D1.4.9]{Elephant}).
\begin{remark}\label{rem:make_everything_singlesorted}
    When working with geometric logic, we will often also make the convenient assumption that our signatures are single-sorted.  This is possible by combining all the sorts $\{\,A_i\mid i \in I\,\}$ into one and introducing new relation symbols $U_i$ such that $U_i(x)$ expresses that ``$x$ is in sort $A_i$'', and axioms such as $\vdash_x \bigvee_{i \in I} U_i(x)$ (expressing that every element of the single sort in our new signature comes from one of the previous sorts).  See \cite[Lemma D1.4.13]{Elephant} for more details.  
    
    Similarly, a dual geometric theory with a finite set of sorts is equivalent to one with a single sort; however, the same construction \emph{cannot} be performed for a dual geometric theory with an infinite set of sorts.
\end{remark}
\begin{remark}
    Note that both geometric and dual geometric logic subsume \emph{coherent} logic (also called \emph{positive} logic), and hence, via a process commonly called \emph{Morleyisation} \cite[Lemma D1.5.13]{Elephant}, they also contain classical first-order logic.
\end{remark}
\begin{example}
    Geometric logic is more expressive than standard finitary first-order logic.  For instance, consider the theory of \emph{fields which are algebraic over their prime subfield}, which adds to the usual axioms of a field the geometric sequent
    \[
    \top\vdash_x \bigvee_{q \in \Z[x] \setminus \{0\}} q(x) = 0.
    \]
    A model of the theory is thus a subfield of one of the algebraically closed fields $\overline{\mathbb{Q}}$, or $\overline{\mathbb{F}}_p$ for $p$ prime, depending on the characteristic.  This theory cannot admit a finitary axiomatisation because the size of each model is bounded, contradicting the upwards L\"owenheim-Skolem theorem (see \cite[Corollary 6.1.4]{hodges}).
\end{example}
\begin{example}[cf.\ Example D1.1.7(l) \cite{Elephant}]\label{ex:metric_theory} %
    Certain theories that resist a geometric axiomatisation can instead be axiomatised using dual geometric logic.  Consider the theory of \emph{extended (separated) metric spaces}: the single-sorted theory with a binary relation symbol $D_{\leqslant \epsilon} \subseteq X^2$ for each rational number $\epsilon > 0$ and axioms expressing that $D_{\leqslant \epsilon} $ is the set of pairs $(x,y) \in X^2$ whose distance is less than or equal to $\epsilon$, i.e.\ the following sequents.
    \begin{align*}
        x = y \dashv\! & \vdash_{x,y} \bigwedge_{\epsilon > 0} D_{\leqslant \epsilon}(x,y), \\
        D_{\leqslant \epsilon }(x,y) & \vdash_{x,y} D_{\leqslant \delta} (y,x) & \forall \delta \geqslant \epsilon > 0, \\
        D_{\leqslant \epsilon}(x,y) \land D_{\leqslant \delta}(y,z) & \vdash_{x,y,z} D_{\leqslant \epsilon + \delta}(x,z) & \forall \delta, \epsilon > 0.
    \end{align*}
    Given a set $X$ equipped with such relations $D_{\leqslant \epsilon} \subseteq X^2$, the metric $d \colon X^2 \to \R \cup \{\infty\}$ is given by $d(x,y) = \inf\{\,\epsilon \mid (x,y) \in D_{\leqslant \epsilon }\,\}$.  
\end{example}

\subsection{Interpreting logic in a bundle}
We now describe how (dual) geometric logic is interpreted over a localic category $\H$.
\begin{definition}\label{def:model_over_groupoid}
\begin{enumerate}
    \item Let $\Sigma$ be a signature.  A $\Sigma$-\emph{structure} $M$ \emph{over a locale} $H$ consists of a bundle $A^M \to H$ for each sort $A$ of $\Sigma$ and, for each relation symbol $R \subseteq A_1 \times \dots \times A_n$, an interpretation as a sublocale of the wide pullback
    \[
    \begin{tikzcd}
        R^M \ar[hook]{r} & A^M_1 \times_H \dots \times_H A^M_n.
    \end{tikzcd}
    \]

    \item If $H$ is the space of objects of a localic category $\H$, a $\Sigma$-\emph{structure over the localic category} $\H$ is a $\Sigma$-structure over the objects $H$ such that each carrier $A^M_i$ is equipped with an $H_1$-action $\theta^{A^M_i} \colon A^M_i \times_H H_1 \to A^M_i$, and each relation symbol $R^M$ is \emph{stable} under the action $\theta^{A^M_1} \times_H \dots \times_H \theta^{A^M_n}$, by which we mean that there is a factorisation
    \[\begin{tikzcd}
	{R^M \times_H H_1 \times_H \dots \times_H H_1} &&& {R^M} \\
	{A^M_1 \times_H \dots \times_H A^M_n \times_H H_1 \times_H \dots \times_H H_1} &&& {A^M_1 \times_H \dots \times_H A^M_n}
	\arrow[dashed, from=1-1, to=1-4]
	\arrow[hook, from=1-1, to=2-1]
	\arrow[hook, from=1-4, to=2-4]
	\arrow["{(\theta^{A^M_1} \times_H \dots \times_H \theta^{A^M_n}) \circ \sigma}", from=2-1, to=2-4]
\end{tikzcd}\]
    where $\sigma\colon A^M_1 \times_H \dots \times_H A^M_n \times_H H_1 \times_H \dots \times_H H_1 \to A^M_1 \times_H H_1 \times_H \dots \times_H A^M_n \times_H H_1$ is the obvious isomorphism,
    or in point-set notation, if $(x_1,\dots, x_n) \in R^M$, then $(h_1 \cdot x_1, \dots , h_n \cdot x_n) \in R^M$ for all compatible $h_i \in H_1$.

    \item\label{def:model:enum:intrepretation} Let $M$ be a $\Sigma$-structure and let $\phi$ be a formula of (dual) geometric logic over $\Sigma$.  The \emph{interpretation} $\phi^M$ in $M$ is defined recursively as follows.
    \begin{enumerate}
        \item If $\phi$ is just a relation symbol, then $\phi^M = R^M$.

        \item\label{def:model:enumenum:equals} If $\phi$ is the equality predicate $(x =^A y)$ in the sort $A$, then $\phi^M$ is defined as the diagonal $\Delta_{A^M} \colon A^M \hookrightarrow A^M \times_{H} A^M$.
        
        \item If $\phi \equiv \bigwedge_{i \in I} \psi_i$ (respectively, $\phi \equiv \bigvee_{i \in I} \psi_i$), then $\phi^M$ is the intersection (resp., union) of the sublocales $\psi^M_i$. For example, the interpretation $(\psi \land \chi)^M$ is given by the pullback
        \[
        \begin{tikzcd}
            (\psi \land \chi)^M \ar[hook]{r} \ar[hook]{d} & \psi^M \ar[hook]{d} \\
            \chi^M \ar[hook]{r} &  A^M_1 \times_H \dots   \times_H A^M_n .
        \end{tikzcd}
        \]
        \item\label{def:model:enumenum:exists} If $\phi \equiv \exists x_{n+1} \, \psi(x_1, \dots , x_{n+1})$, then $\phi^M$ is the image of the composite
        \[
        \begin{tikzcd}
            \psi^M \ar[hook]{r} & A_1^M \times_H \dots \times_H A_n \times_H A^M_{n+1} \ar{r}{\pi_{1,\dots, n}} &  A_1^M \times_H \dots \times_H A^M_{n},
        \end{tikzcd}
        \]
        where $\pi_{1, \dots, n}$ is the projection map to the first $n$ factors.
    \end{enumerate}
    \item Let $\theory$ be a theory over the signature $\Sigma$, i.e.\ $\theory$ consists of axioms of the form $\varphi(x_1, \dots , x_n) \vdash_{x_1:A_1, \dots x_n:A_n} \psi(x_1, \dots , x_n)$.  A $\Sigma$-structure is a $\theory$-model if for each axiom of $\theory$, there is an inclusion of sublocales $\varphi^M \subseteq \psi^M$.
    
     \item Let $M$ and $N$ be models of a theory $\theory$ over $\H$.  A \emph{model homomorphism} $M \xrightarrow{\tau} N$ consists of, for each sort $A$ of $\theory$, a morphism of bundles
        \[
        \begin{tikzcd}
            A^M \ar{rr}{\tau_A} \ar{rd} &&A^N \ar{ld} \\
            & H &
        \end{tikzcd}
        \]
        such that:
        \begin{enumerate}
            \item each $\tau_A$ is \emph{equivariant} with respect to the endowed $H_1$-actions, i.e.\ the square
            \[
            \begin{tikzcd}
                A^M \times_{H_0} H_1 \ar{rr}{\tau_A \times_{H_0} H_1} \ar{d}[']{\theta^{A^M}} && A^N \times_{H_0} H_1 \ar{d}{\theta^{A^N}} \\
                A^M \ar{rr}{\tau_A} && A^N
            \end{tikzcd}
            \]
            commutes, or in point-set notation, $h \cdot \tau_A(x) = \tau_A(h \cdot x)$;

            \item and $\tau$ preserves the $\Sigma$-structure of $M$ and $N$ in the sense that, for each relation symbol $R \subseteq A_1 \times \dots \times A_n$ of $\theory$, there is a factorisation
            \[
            \begin{tikzcd}
                R^M \ar[hook]{d} \ar[dashed]{rrr} &&& R^N \ar[hook]{d} \\
                A_1^M \times_{H_0} \dots \times_{H_0} A_n^M \ar{rrr}{\tau_{A_1} \times_{H_0} \dots \times_{H_0} \tau_{A_n}} &&& A_1^N \times_{H_0} \dots \times_{H_0} A_n^N.
            \end{tikzcd}
            \]
        \end{enumerate}
        We use $\Tmod(\H)$ to denote the category of models of $\theory$ over $\H$.
\end{enumerate}
Note that for this interpretation to be sound we require some restrictions on the models. The LH-models and PS-models defined below provide settings in which the interpretation is sound for geometric and dual geometric logic respectively.
\end{definition}
\begin{definition}\label{def:LH_models_PS_models}
   We differentiate between different species of models.
    \begin{enumerate}
        \item Suppose that $\theory$ is a geometric theory.  If, for each sort of $\Sigma$, the underlying carrier $A^M \to H$ is a \emph{local homeomorphism}, and the interpretation of each relation symbol $R^M \hookrightarrow A^M_1 \times_H \dots \times_H A^M_n$ is an \emph{open} sublocale, then we say that $M$ is a \emph{local homeomorphism} model, or an \emph{LH-model} for short.

        We will use $\Tmod^\rmLH(\H) \subseteq \Tmod(\H)$ to denote the subcategory of LH-models of $\theory$, and $\Tmod^\rmLH_{\cong}(\H)$ to denote the groupoid of all \emph{invertible} homomorphisms of LH-models of $\theory$.

        \item Suppose that $\theory'$ is a \emph{dual} geometric theory.  If, for each sort of $\Sigma$, the underlying carrier $A^M \to H$ is a \emph{proper} and \emph{separated} map, and the interpretation of each relation symbol $R^M \hookrightarrow A^M_1 \times_H \dots \times_H A^M_n$ is a \emph{closed} sublocale, then we say that $M$ is a \emph{proper separated} model, or a \emph{PS-model} for short.

        We will use $\Tpmod^\rmPS(\H) \subseteq \Tpmod(\H)$ to denote the subcategory of PS-models of $\theory'$, and $\Tpmod^\rmPS_{\cong}(\H)$ to denote the groupoid of {invertible} homomorphisms of PS-models of $\theory'$.
    \end{enumerate}
\end{definition}
\begin{remark}
    Let $M$ be an LH-structure over a locale $H$.  By \cite[Lemma C1.3.2]{Elephant}, the maps $\Delta_{A^M}$ and $\pi_{1,\dots,n}$, considered in \cref{def:model_over_groupoid}\cref{def:model:enumenum:equals} and \cref{def:model:enumenum:exists} respectively, are local homeomorphisms, and in particular open maps.  Hence, since open sublocales are closed under finite intersections and arbitrary unions, and the image of an open sublocale under an open map is again open, the interpretation of any \emph{geometric} formula $\phi^M \hookrightarrow A_1^M \times \dots \times A_n^M$ is an open sublocale, and moreover the composite $\phi^M \hookrightarrow A_1^M \times \dots \times A_n^M \to H$ is a local homeomorphism, since an open sublocale inclusion is a local homeomorphism.

    In a similar fashion, if $M$ is a PS-structure, the maps $\Delta_{A^M}$ and $\pi_{1,\dots,n}$ are closed by \cite[Proposition 4.1]{vermeulen1994proper}, and so the interpretation of any \emph{dual geometric} formula $\phi^M \hookrightarrow A_1^M \times \dots \times A_n^M$ is a closed sublocale and the composite $\phi^M \hookrightarrow A_1^M \times \dots \times A_n^M \to H$ is a proper separated map.
\end{remark}
\begin{remark}\label{rem:LH_model_in_topos}
    The LH-models of a geometric theory $\theory$ over a localic groupoid $\H$, as in \cref{def:LH_models_PS_models}, coincide with the usual notion of a $\theory$-model internal to the topos $\Sh(\H)$.  
\end{remark}
\begin{example}\label{ex:models_over_1}
    Let $\1$ denote the trivial (one-arrow) localic category.  Recall that an LH-bundle over $\1$ is the same datum as a discrete set, and so the LH-models of a geometric theory $\theory$ over $\1$ correspond to the usual (i.e.\ set-based) models of $\theory$.   
\end{example}
\begin{example}\label{ex:PS_models_over_1}
    Recall also that the proper separated bundles over the one-point space are the compact Hausdorff locales, and so, classically, the PS-models of a dual geometric theory $\theory'$ over $\1$ correspond to a (set-based) model of $\theory'$ equipped with a compact Hausdorff topology such that the interpretation of every basic relation is a closed subset.

    In particular, a PS-model over $\1$ of the theory of extended metric spaces from \cref{ex:metric_theory} consists of a set equipped with both an (extended) metric and a compact Hausdorff topology such that metric is lower semi-continuous with respect to the topology, i.e.\ the closed balls for the metric are closed in the compact Hausdorff topology.  These structures have been studied in \cite{abbadini_hofmann} under the name \emph{metric compact Hausdorff spaces} (see also \cite{hofmann_reis}). %
\end{example}

\subsection{Models under change of base}\label{sec:base_change}
We now describe how structures and models behave under a change of base.  The first observation to make is that the map that sends a localic category $\H$ to the category of $\Sigma$-structures over $\H$ extends to a pseudofunctor $\Sstruct(-) \colon \LocCat\op \to \CAT$.  In particular, given an internal functor $\Phi\colon \G \to \H$ between localic categories and a $\Sigma$-structure $M$ on $\H$, we obtain a $\Sigma$-structure over $\G$ by `pulling back' $M$ along $\Phi$.  We denote the \emph{pullback structure} by $\Phi^\ast M$.

Explicitly, if $(e\colon A^M \to H_0, \beta\colon A^M \times_{H_0} H_1 \to A^M)$ is the bundle over $\H$ interpreting a sort $A$ in $\Sigma$, then the interpretation of $A$ in the structure $\Phi^\ast M$ is given by the pullback bundle ($e'$, $\beta'$) where $e'$ is given by the pullback
\[\begin{tikzcd}
	{\Phi^* A^M} \ar{r}{\Phi'_0} & A^M \\
	{G_0} & {H_0},
	\arrow["{e'}"', from=1-1, to=2-1]
	\arrow[from=1-1, to=1-2]
	\arrow["e", from=1-2, to=2-2]
	\arrow["{\Phi_0}"', from=2-1, to=2-2]
	\arrow["\lrcorner"{anchor=center, pos=0.125}, draw=none, from=1-1, to=2-2]
\end{tikzcd}\]
and $\beta'$ is defined via the universal property of this pullback as in the following diagram.
\[\begin{tikzcd}
	\Phi^\ast A^M \times_{G_0} G_1 &&  A^M \times_{H_0} H_1 \\
	& {\Phi^\ast A^M} & {A^M} \\
	{G_1} & {G_0} & {H_0}
	\arrow["{\langle \Phi'_0, \Phi_1 \rangle}", from=1-1, to=1-3]
	\arrow["{\beta'}", dashed, from=1-1, to=2-2]
	\arrow["{\pi_2}"', from=1-1, to=3-1]
	\arrow["\beta", from=1-3, to=2-3]
	\arrow[from=2-2, to=2-3]
	\arrow["{e'}"', from=2-2, to=3-2]
	\arrow["\lrcorner"{anchor=center, pos=0.125}, draw=none, from=2-2, to=3-3]
	\arrow["e", from=2-3, to=3-3]
	\arrow["t"', from=3-1, to=3-2]
	\arrow["{\Phi_0}"', from=3-2, to=3-3]
\end{tikzcd}\]
(the commutativity of the outside square is ensured since $\beta$ defines an action).  In other words, $\Phi^* A^M \times_{G_0} G_1 $ has points $(x, a, f \colon a \to b)$ and $\beta'$ acts on points by $(x,a,f) \mapsto (\beta(x,\Phi_1(f)),b)$.
Similarly, for each relation symbol $R \subseteq A_1 \times \dots \times A_n$ in $\Sigma$, we define $R^{\Phi^\ast M}$ as the pullback
\[
\begin{tikzcd}
    \Phi^\ast R^M \ar{r} \ar[hook]{d} \ar[draw=none]{rdd}[anchor=center, pos=0.125]{\lrcorner} & R^M \ar[hook]{d} \\
    A_1^{\Phi^\ast M} \times_{G_0} \dots \times_{G_0} A_n^{\Phi^\ast M} \ar{d} \ar[dashed]{r} & A_1^M \times_{H_0} \dots \times_{H_0} A_n^M \ar{d} \\
    G_0 \ar{r}{\Phi_0} & H_0.
\end{tikzcd}
\]
We easily see that $R^{\Phi^\ast M}$ is stable under the $G_1$-action defined above.  This completes the description of $\Phi^\ast M$.

For each homomorphism $\sigma \colon M \to N$ of $\Sigma$-structures over $\H$, there is an induced homomorphism $\Phi^\ast \sigma$ whose component, at a sort $A$ of $\Sigma$, is induced by the universal property of the pullback as in the following diagram.
\[
\begin{tikzcd}
    \Phi^\ast A^M \ar{rr} \ar[bend right]{ddr} \ar[dashed]{dr}{\Phi^\ast \sigma_A} && A^M \ar{dr}{\sigma_A}  \ar[bend right]{ddr} & \\
    & \Phi^\ast A^N \ar[crossing over]{rr} \ar{d} \ar[draw=none]{rrd}[anchor=center, pos=0.125]{\lrcorner}&& A^N \ar{d} \\
    & G_0 \ar{rr}{\Phi_0} && H_0
\end{tikzcd}
\]
Since the $G_1$-actions on $\Phi^\ast A^M$ and $\Phi^\ast A^N$ are defined universally, it is straight-forward to show that $\Phi^\ast \sigma_A$ is equivariant, as well as the fact that $\Phi^\ast \sigma$ preserves the interpretation of relations.  Therefore, for each internal functor $\Phi \colon \G \to \H$ and signature $\Sigma$, we have constructed a functor $\Sstruct(\Phi) \colon \Sstruct(\H) \to \Sstruct(\G)$.

It remains to describe how $\Sstruct(-)$ acts on an internal natural transformation $\tau \colon \Phi \Rightarrow \Psi$, defined by a map $\tau \colon G_0 \to H_1$.  Once again, this is defined by the universal property of the pullback.  Let $M$ be a $\Sigma$-structure over $\H$.  First note that, by the pullback lemma, the left-hand square in the diagram
\[
\begin{tikzcd} 
    \Psi^\ast A^M \ar{r} \ar{d} & t^\ast A^M \ar{d} \ar{r} & A^M \ar{d} \\
    G_0 \ar{r}{\tau} \ar[bend right]{rr}[']{\Psi_0} & H_1 \ar{r}{t} & H_0
\end{tikzcd}
\]
is a pullback, for each sort $A$ of $\Sigma$, where $\Psi^\ast A^M \to t^\ast A^M$ is the morphism universally induced by the pullback square on the right.  Thus, there is a universally induced morphism
\[
\begin{tikzcd}
    \Phi^\ast A^M \ar{rr} \ar[bend right]{ddr} \ar[dashed]{dr}{\tau^\ast_{A^M}} && s^\ast A^M \ar{dr}{\theta}  \ar[bend right]{ddr}& \\
    & \Psi^\ast A^M \ar[crossing over]{rr} \ar{d} \ar[draw=none]{rrd}[anchor=center, pos=0.125]{\lrcorner} && t^\ast A^M \ar{d} \\
    & G_0 \ar{rr}{\tau} && H_1,
\end{tikzcd}
\]
where $\theta \colon s^\ast A^M \to t^\ast A^M$ is the descent datum associated with the action $ A^M \times_{H_0} H_1 \to A^M$ (see \cite[Appendix A]{manuell2023representing}) and $\Phi^\ast A^M \to s^\ast A^M$ is the morphism induced by the universal property of the pullback $s^\ast A^M$.  Together, the maps $\tau_{A^M}^\ast$ define a homomorphism of $\Sigma$-structures $\tau^\ast_M \colon \Phi^\ast M \to \Psi^\ast M$, which in turn are the components of a natural transformation $\tau^\ast \colon \Phi^\ast \Rightarrow \Psi^\ast$.

As $\Sstruct(\Phi)$ and $\Sstruct(\tau)$ are defined using the universal property of the pullback, they are automatically pseudofunctorial.  Thus, for each signature $\Sigma$, there is a pseudofunctor
\[
\Sstruct(-) \colon \LocCat\op \to \CAT
\]
that sends a localic category $\H$ to the category of $\Sigma$-structures over $\H$.

If $\theory$ is a theory over the signature $\Sigma$, then the pseudofunctor $\Sstruct(-)$ does \emph{not} necessarily restrict to a pseudofunctor $\Tmod(-) \hookrightarrow \Sstruct(-)$ that sends a localic category to the $\theory$-models over it.  This is because surjective locale morphisms are not necessarily stable under pullback.
\begin{example}
    For instance, a surjective locale map $p \colon X \twoheadrightarrow Y$ can be viewed as a model for the theory of \emph{inhabited objects} $\{\top \vdash \exists x\colon X.\ \top\}$.  If the pullback of $f$ along $h \colon Z \to Y$ is not a surjection, then $h^\ast X$ is no longer a model for this theory.
\end{example}
Nonetheless, we are able to obtain pseudofunctors of $\theory$-models by restricting to LH-structures or PS-structures.
\begin{lemma}\label{lem:change_of_base_for_formulae}
    Let $\Phi\colon \G \to \H$ be an internal functor, and let $M$ be a $\Sigma$-structure over $\H$.  If $M$ is an LH-structure, then $\Phi^\ast M$ is an LH-structure, and $\Phi^\ast$ preserves the interpretation of geometric formulae.  If $M$ is a PS-structure, then $\Phi^\ast M$ is a PS-structure and $\Phi^\ast$ preserves the interpretation of dual geometric formulae.
\end{lemma}
\begin{proof}
    First note that, since local homeomorphisms and open sublocales are stable under pullback, if $M$ is an LH-structure, then so too is $\Phi^\ast M$.  Similarly, as proper separated maps and closed sublocales are also preserved under pullback, if $M$ is a PS-structure then so too is $\Phi^\ast M$.

    We now turn to the preservation of interpretation of formulae.  The equality predicate $=$, interpreted by the diagonal map, is automatically preserved by pullback.

    Suppose $M$ is an LH-structure.  Recall that the induced map on sublocales $\Phi^\ast \colon \Sub(M) \to \Sub(\Phi^\ast M)$ preserves finite intersections and arbitrary unions of open sublocales.  Thus, $\Phi^\ast$ preserves the interpretation of the symbols $\land, \bigvee$.  Moreover, as open surjections of locales are preserved by pullback (see \cite[Proposition V.4.1]{joyal1984galois}), so too is the image of an open sublocale along an open map.  Thus, $\Phi^\ast$ also preserves the interpretation of $\exists$.

    Similarly, if $M$ is a PS-structure, $\Phi^\ast \colon \Sub(M) \to \Sub(\Phi^\ast M)$ preserves arbitrary intersections and finite unions of closed sublocales.  And moreover, as proper surjections are stable under pullback (see \cite[Proposition 4.2]{vermeulen1994proper}), $\Phi^\ast$ preserves the interpretation of $\exists$.
\end{proof}
\begin{corollary}
    If $M$ is an LH-model of a geometric theory $\theory$, so too is $\Phi^\ast M$.  Similarly, if $M$ is a PS-model of a dual geometric theory $\theory'$, then so is $\Phi^\ast M$.
\end{corollary}
\begin{proof}
    Suppose $M$ is an LH-model of $\theory$.  By \cref{lem:change_of_base_for_formulae}, $\Phi^\ast$ preserves the interpretation of a geometric formula, i.e.\ $\phi^{\Phi^\ast M}  = \Phi^\ast \phi^M $ for any geometric $\phi$.  As $M$ is a $\theory$-model, for each axiom $\phi(x_1, \dots, x_n) \vdash_{x_1 : A_1, \dots, x_n : A_n} \psi(x_1, \dots , x_n)$ of $\theory$, we have that $\phi^M \subseteq \psi^M$.  Hence, $\phi^{\Phi^\ast M} = \Phi^\ast \phi^M \subseteq \Psi^\ast \psi^M = \psi^{\Phi^\ast M}$, and so $\Phi^\ast M$ is a $\theory$-model as well.  The argument for PS-models of dual geometric theories is identical.
\end{proof}
Thus, for each geometric theory $\theory$ and dual geometric theory $\theory'$, by restricting $\Sstruct(-)$ we arrive at the pseudofunctors
\[
\Tmod^\rmLH(-) \colon \LocCat\op \to \CAT \text{ and } \Tpmod^\rmPS(-) \colon \LocCat\op \to \CAT.
\]
Note also that if an internal natural transformation $\tau \colon \Phi \to \Psi$ has an inverse (as is the case if the codomain $\Phi , \Psi \colon \G \rightrightarrows \H$ is a localic groupoid), then the induced morphism $\tau^\ast_M$ is invertible too.  Thus, we also have pseudofunctors
\[\Tmod^\rmLH_{\cong}(-) \colon \LocGrpd\op \to \CAT \text{ and }\Tmod^\rmPS_{\cong}(-) \colon \LocGrpd\op \to \CAT . \]

\subsection{The generic bundle of a theory}\label{subsec:generic_model}
We saw in \cref{ex:models_over_1} and \cref{ex:PS_models_over_1} how it is relatively simple to describe all the LH-models and PS-models over the one-point space.  We now consider a more involved example where, for each geometric theory $\theory$ and each dual geometric theory $\theory'$, we construct:
\begin{enumerate}
    \item a localic category $\ReprG^\rmLH(\theory)$ equipped with an LH-model $\ReprE^\rmLH(\theory)$ of $\theory$,
    \item and a localic category $\ReprG^\rmPS(\theory')$ equipped with a PS-model $\ReprE^\rmPS(\theory')$ of $\theory'$.
\end{enumerate}
We will eventually show in \cref{sec:univ_prop} that the models we construct are \emph{generic} in that any other LH-model (respectively, PS-model) can be obtained by `pulling back' the generic model $\ReprE^\rmLH(\theory)$ (resp., $\ReprE^\rmPS(\theory')$) along an anafunctor.

The idea behind the construction has been prefigured in \cref{sec:forcing}: since there is a sense in which `every set is subcountable', and `every compact Hausdorff space is a subquotient of the Cantor space', we can consider a localic category whose objects are models presented as subquotients of a fixed locale $P$.  For LH-models, this localic category, or more precisely its groupoid coreflection, was constructed in \cite{manuell2023representing}.  The generic model over the localic category is obtained by taking the bundle whose fibre over an object, i.e.\ a model presented as a subquotient of $P$, is given by the subquotient in question.  A similar idea, in the topological case, is pursued in \cite{breinerthesis}.

We construct the generic bundles of geometric theories and dual geometric theories as follows:
\begin{enumerate}
    \item First, we construct the locales of objects $\ReprG^\rmLH(\theory)_0$ and $\ReprG^\rmPS(\theory')_0$, as well as the generic bundles $\ReprE^\rmLH(\theory)$ and $\ReprE^\rmPS(\theory')$ and demonstrate that these are indeed LH-models and PS-models respectively.
    \item Next, we construct the localic category structure by describing the locales of arrows $\ReprG^\rmLH(\theory)_1$ and $\ReprG^\rmPS(\theory')_1$, and their actions on the bundles $\ReprE^\rmLH(\theory)$ and $\ReprE^\rmPS(\theory')$.
    \item Finally, we describe what new relations must be added to the locales $\ReprG^\rmLH(\theory)_1$ and $\ReprG^\rmPS(\theory')_1$ to obtain the core groupoid.
\end{enumerate}

To emphasise the inherent duality of the local homeomorphism case and the proper separated case, we will define the classifying localic category and the generic bundle for each simultaneously.
\begin{definition}[{cf.\ \cite[\S 2]{manuell2023representing}}]\label{def:repr-grpd-LH}
Let $J$ denote either $\rmLH$ or $\rmPS$.
Let $\theory$ be a geometric first-order theory if we set $J$ to be $\rmLH$ or a dual geometric first-order theory if we set $J$ to be $\rmPS$. Define $P^\rmLH$ to be the discrete locale of natural numbers $\N$ and $P^\rmPS$ to be the compact Hausdorff locale $2^\N$.  We will denote by $\ReprG^J(\theory)_0$ the locale with the following generalised frame presentation (see \cref{sec:gen_prens}). For $J = \rmLH$, this is an ordinary presentation, while for $J = \rmPS$, this is a closed-type generalised presentation.
\begin{enumerate}
  \item For each sort $A$ of $\theory$, there is a basic proposition ${[p \sim^A q]}$ for each $p,q \in P^J$ together with the following axioms for each $p,q,r \in P^J$:
  \begin{align*}
   {[p \sim^A q]} &\vdash {[q \sim^A p]}, \\
   {[p \sim^A q]} \land {[q \sim^A r]} &\vdash {[p \sim^A r]}.
   \end{align*}
  \item For each relation symbol $R \subseteq A_1 \times \dots \times A_k$ of $\theory$, and for each $p_1,\dots,p_k \in P^J$ and $q_1,\dots,q_k \in P^J$, we have a basic proposition ${[(p_1,\dots,p_k) \in R]}$ and axioms
  \begin{align*}
      {[(p_1,\dots,p_k) \in R]} \land {[p_1 \sim^{A_1} q_1]} \land \dots \land {[p_k \sim^{A_k} q_k]} & \vdash {[(q_1,\dots,q_k) \in R]},  \\
      {[(p_1,\dots,p_k) \in R]} & \vdash {[p_1 \sim^{A_1} p_1]} \land \dots \land {[p_k \sim^{A_k} p_k]}.
   \end{align*}
  \item For each axiom $\phi \vdash_{x_1\colon A_1,\dots,x_k\colon A_k} \psi $ of $\theory$, we add an axiom 
  \[\bigwedge_{i=1}^k {[p_i \sim^{A_i} p_i]} \wedge \phi_{p_1,\dots,p_k} \vdash \psi_{p_1,\dots,p_k}\]
  for each $p_1, \dots , p_k \in P^J$, where for a formula $\chi(x_1, \dots , x_k)$ we define $\chi_{p_1, \dots, p_k}$ inductively as follows:
  \begin{enumerate}
      \item If $\chi$ is a basic relation $R(x_1, \dots , x_k)$, then $\chi_{p_1, \dots , p_k}$ is defined to be the generator $[(p_1, \dots , p_k) \in R]$.
      \item If $\chi \equiv (x_1 =_A x_2)$ is the equality predicate of the sort $A$, then $\chi_{p_1,p_2}$ is defined to be the generator $[p_1 \sim^A p_2]$.
      \item If $\chi$ is a conjunction $\bigwedge_{i\in I} \chi_i$ or a disjunction $\bigvee_{i \in I} \chi_i$, then $\chi_{p_1, \dots , p_k}$ is defined as $\bigwedge_{i\in I} (\chi_i)_{p_1, \dots , p_k}$ or $\bigvee_{i \in I} (\chi_i)_{p_1, \dots , p_k}$ respectively.
      \item If $\chi \equiv \exists x_{k+1} \chi' (x_1, \dots , x_k, x_{k+1})$, then we define $\chi_{p_1, \dots , p_k}$ as the $P^J$-indexed join $\bigvee_{p_{k+1} \in P^J} \chi'_{p_1, \dots , p_{k+1}}$.
  \end{enumerate}
  \end{enumerate}
\end{definition}
\begin{example}[Object classifier]
    Consider the theory of objects $\Obj$, the theory with one sort, no relation symbols and no axioms.  Viewing $\Obj$ as a geometric theory, we can apply the construction of \cref{def:repr-grpd-LH} with $J = \rmLH$, which results in the familiar classifying locale of partial equivalence relations on $\N$.

    Now viewing $\Obj$ as a dual geometric theory and taking $J = \rmPS$ in \cref{def:repr-grpd-LH}, we obtain the classifying locale of closed partial equivalence relations on $2^\N$,
    as explicitly described in \cref{ex:closed_partial_equiv_rels}.
\end{example}
\begin{definition}\label{def:generic-LH-bundle}
   Let $\ReprE^J(\theory)$ denote the following $\Sigma$-structure over $\ReprG^J(\theory)_0$.
   \begin{enumerate}
       \item For each sort $A$ of $\Sigma$, take $A^{\ReprE^J(\theory)}$ to be the locale with the generalised presentation obtained by taking a copy of the generalised presentation for $\ReprG^J(\theory)_0$ and, for each $p \in P^J$, adding a new generator $[\equiv p]$, and the relations
       \begin{align*}
        [\equiv p] \land [p \sim q] \dashv\! & \vdash [\equiv q] \land [p \sim q], \\
        [\equiv p] \land [\equiv q] & \vdash [p \sim q], \\
         & \vdash \bigvee_{p \in P^J} [\equiv  p], \\
        [\equiv p] & \vdash [p \sim p].
    \end{align*}
    The interpretation of the sort $A$ is given by the obvious locale morphism 
    \[\rho_A \colon A^{\ReprE^J(\theory)} \to \ReprG^J(\theory)_0\]
    that acts on a generalised frame generator of $\ReprG^J(\theory)_0$ by sending it to the corresponding generator of $A^{\ReprE^J(\theory)}$.

    \item Let $R \subseteq A_1 \times \dots \times A_\ell$ be a relation symbol.  Observe that the wide pullback $A_1^{\ReprE^J(\theory)} \times_{\ReprG^J(\theory)_0} \dots \times_{\ReprG^J(\theory)_0} A_\ell^{\ReprE^J(\theory)}$ has a frame presentation consisting of one copy of the presentation for $\ReprG^J(\theory)_0$ and $\ell$ copies of the generators $[\equiv_j p]$, as well as corresponding relations for $A_j^{\ReprE^J(\theory)}$ as above.  The interpretation of the relation symbol $R$ is the (open/closed) sublocale $R^{\ReprE^J(\theory)} \hookrightarrow A_1^{\ReprE^J(\theory)} \times_{\ReprG^J(\theory)_0} \dots \times_{\ReprG^J(\theory)_0} A_\ell^{\ReprE^J(\theory)}$ given by adding the relation
    \[
     \vdash \! \bigvee_{(p_1,\dots, p_\ell) \in (P^J)^\ell} \!\!\! [(p_1, \dots , p_\ell) \in R] \land [\equiv_1 p_1] \land \dots \land [\equiv_\ell p_\ell].
    \]
    to this presentation of $ A_1^{\ReprE^J(\theory)} \times_{\ReprG^J(\theory)_0} \dots \times_{\ReprG^J(\theory)_0} A_\ell^{\ReprE^J(\theory)}$.
   \end{enumerate}
\end{definition}

\begin{remark}\label{rem:points_of_generic_bundle}
    Recall from \cite[Remark 2.3]{manuell2023representing} that the points of $\ReprG^\rmLH(\theory)_0$ correspond to models of $\theory$ where each sort is given by a subquotient of the natural numbers (as encoded by a partial equivalence relation).  Given such a point, i.e.\ a model $\N /{\sim} \vDash \theory$, a point of $\ReprE^\rmLH(\theory)$ in the fibre of $\N /{\sim} \vDash \theory$ is an equivalence class $\overline{n} \in \N/{\sim}$ of the partial equivalence relation. This intuition will be made precise in \cref{prop:each-generic-sort-is-LH}.

    On the other hand, a point of $\ReprG^\rmPS(\theory)$ corresponds to a $\theory$-model where each sort is presented as a subquotient of $2^\N$ (as encoded by a closed partial equivalence relation) where moreover each relation is interpreted by a closed sublocale.  Similarly, the fibre over such a point $2^\N / {\sim} \vDash \theory$ in the generic bundle $\ReprE^\rmPS(\theory)$ is an equivalence class $\overline{s} \in 2^\N / {\sim}$ of the closed partial equivalence relation.
\end{remark}
\begin{remark}
    Note that, for a general dual geometric theory $\theory'$, because the sorts, relations symbols and axioms of the theory $\theory'$ are set-indexed, i.e.\ we have a \emph{set} of sorts, a \emph{set} of relation symbols, a \emph{set} of axioms, the presentations of the locales $\ReprG^\rmPS(\theory')_0$ and $\ReprE^\rmPS(\theory')$ given above are not \emph{dual} presentations, but rather \emph{closed-type} generalised presentations as discussed in \cref{rem:beyond_dual}.  For instance, if $\theory'$ is a single-sorted dual geometric theory with a single unary predicate of arity $n$, then the locale on the generators, of which $\ReprG^\rmPS(\theory')_0$ is a sublocale, is given by $\Srpnsk^{2^\N \times {(2^\N)}^n}$.  It is also possible to consider a variant of dual geometric logic where we instead take a `compact-Hausdorff-locale-indexed collection' of sorts/relation symbols/axioms, in which case applying \cref{def:repr-grpd-LH} would yield a true dual frame presentation.
\end{remark}

\begin{proposition}\label{prop:each-generic-sort-is-LH}
    The locale morphism $\rho_A \colon A^{\ReprE^J(\theory)} \to \ReprG^J(\theory)_0$ is the externalisation of the (discrete/compact-Hausdorff) subquotient of $P^J$ induced by the partial equivalence relation on $P^J$ corresponding to $A$, internal to $\Sh(\ReprG^J(\theory)_0)$. Consequently, for $J = \rmLH$, $\rho_A$ is a local homeomorphism, and for $J = \rmPS$, $\rho_A$ is proper and separated.
\end{proposition}
\begin{proof}
    Since the theory $\ReprG^J(\theory)_0$ classifies contains a copy of the theory of (open/closed) partial equivalence relations on $P^J$ (corresponding to the sort $A$), the topos $\Sh(\ReprG^J(\theory)_0)$ contains a \emph{generic} (open/closed) partial equivalence relation on $P^J$, which we call $\sim$. (Explicitly, we can construct this by taking the open/closed sublocale of $\ReprG^J(\theory)_0 \times P^J \times P^J$ cut out by $\vdash \bigvee_{p,q \in P^J} [p \sim q] \wedge [\equiv_1 p] \wedge [\equiv_2 q]$ and equipped with the projection down to $\ReprG^J(\theory)_0$.)

    Inside the sheaf topos $\Sh(\ReprG^J(\theory)_0)$, we can now compute the subquotient $P^J/{\sim}$, which is of course, an (internal) discrete locale for $J = \rmLH$ and an (internal) compact Hausdorff locale for $J = \rmPS$. We can use the canonical self-presentation (see \cref{sec:canonical_presentation}) of $P^J/{\sim}$ to give an internal presentation with generators $P^J/{\sim}$ and relations $(P^J/{\sim})^2 \sqcup 1$ (this is an ordinary presentation for $J = \rmLH$ and a dual presentation for $J = \rmPS$).
    Explicitly, we have a generator $[= [p]]$ for each $[p] \in P^J/{\sim}$ and relations
    \begin{align*}
        [= [p]] \wedge [= [q]] &\vdash [p] = [q] & \text{for $[p],[q] \in P^J/{\sim}$} \\
        \top &\vdash \bigvee_{[p] \in P^J/\sim} [= [p]]
    \end{align*}
    This presentation would be a little tricky to externalise naively. Luckily, we can massage it into a form which externalises more readily: by \cref{lem:expanded_presentation} we may replace this with an internal presentation with generators $P^J$ and relations indexed by $(P^J)^2 \sqcup 1 \sqcup (P^J)^2 \sqcup P^J$. Hence, we have an equaliser diagram
    \[P^J/{\sim} \hookrightarrow \Srpnsk^{P^J} \rightrightarrows \Srpnsk^{(P^J)^2 \sqcup 1 \sqcup (P^J)^2 \sqcup P^J}\]
    in the category of internal locales of $\Sh(\ReprG^J(\theory)_0)$.
    
    We omit the detailed calculations, but writing $[\equiv p]$ (with $p \in P^J$) for the generators, the relations can be expressed as follows (where we use double square brackets to denote the embedding of the truth value of a proposition/the closed complement of a proposition into the frame). %
    \begin{align*}
        \llbracket p \sim p\rrbracket \wedge \llbracket q \sim q\rrbracket \wedge {} \qquad\qquad\qquad\qquad \\ \bigvee_{p',q' \in \N} \llbracket p \sim p'\rrbracket \wedge \llbracket q \sim q'\rrbracket \wedge [\equiv p'] \wedge [\equiv q'] 
        &\vdash \llbracket p \sim q \rrbracket & \text{for $p,q \in P^J$} \\
        \top &\vdash \bigvee_{p \in P^J} \llbracket p \sim p \rrbracket \wedge [\equiv p] \\
        \llbracket p \sim q \rrbracket \wedge [\equiv p] \dashv\!& \vdash \llbracket p \sim q \rrbracket \wedge [\equiv q] & \text{for $p,q \in P^J$} \\
        [\equiv p] &\vdash \llbracket p \sim p \rrbracket  & \text{for $p \in P^J$}
    \end{align*}
    The final two relations (and the axioms of a partial equivalent relation) then allow us to simplify the first two relations to give
    
    \begin{align*}
        [\equiv p] \wedge [\equiv q] &\vdash \llbracket p \sim q \rrbracket & \text{for $p,q \in P^J$} \\
        \top &\vdash \bigvee_{p \in P^J} [\equiv p] \\
        \llbracket p \sim q \rrbracket \wedge [\equiv p] \dashv\!& \vdash \llbracket p \sim q \rrbracket \wedge [\equiv q] & \text{for $p,q \in P^J$} \\
        [\equiv p] &\vdash \llbracket p \sim p \rrbracket  & \text{for $p \in P^J$}.
    \end{align*}

    Since change of base preserves locale exponentials, the equaliser above externalises to an equaliser
    \[{\bullet} \hookrightarrow \Srpnsk^{P^J} \times \ReprG^J(\theory)_0 \rightrightarrows \Srpnsk^{(P^J)^2 \sqcup 1 \sqcup (P^J)^2 \sqcup P^J} \times \ReprG^J(\theory)_0.\]
    in the category of (external) locales.
    It is easy to externalise the above relations by simply changing the internal truth value (or complement of a truth value) $\llbracket {\cdot} \sim {\cdot} \rrbracket$ into the generator $[ {\cdot} \sim {\cdot} ]$ of $\ReprG^\rmLH(\theory)_0$.  Combining this with the presentation for $\ReprG^J(\theory)_0$ gives a final presentation for the externalised bundle. This is precisely the presentation we gave for $A^{\ReprE^J(\theory)}$ in \cref{def:generic-LH-bundle} above.
\end{proof}
\begin{remark}
    In the above proof, we used an argument that dualises readily to cover both the local homeomorphism and proper separated case. Here we provide an alternative, more `ad hoc' proof that $\rho_A$ is a local homeomorphism which requires less sophisticated technology, but does not generalise to the proper separated case. This proof utilises the characterisation of local homeomorphisms in terms of restrictions to an open cover --- namely, $\rho_A$ is a local homeomorphism if and only if there exists an open cover $\mathcal{U}$ of $A^{\ReprE^\rmLH(\theory)}$ such that, for each $U \in \mathcal{U}$, the map $\rho_A$ restricts to an isomorphism between open sublocales
    \[\begin{tikzcd}
	U & {A^{\ReprE^\rmLH(\theory)}} \\
	V & {\ReprG^\rmLH(\theory)_0}.
	\arrow[hook, from=1-1, to=1-2]
	\arrow[dashed, from=1-1, to=2-1]
	\arrow["\sim"{marking, allow upside down}, shift right = 2, draw=none, from=1-1, to=2-1]
	\arrow["{\rho_A}", from=1-2, to=2-2]
	\arrow[hook, from=2-1, to=2-2]
    \end{tikzcd}\]
    \begin{proof}[Alternative proof]
    The generators $[\equiv n]$ constitute an open covering of $A^{\ReprE^\rmLH(\theory)}$ by the relation $\top \vdash \bigvee_{n \in \N} [\equiv n]$. Let $U$ be the open sublocale of $A^{\ReprE^\rmLH(\theory)}$ corresponding to the open $[\equiv n]$ and let $V$ be the open sublocale of $G^\rmLH(\theory)_0$ corresponding to the open $[n \sim^A n]$.

    We start by showing the map $\rho_A$ restricts to a map $p\colon U \to V$.
    Since the sublocales $U$ and $V$ are cut out by the relations $\vdash [\equiv n]$ and $\vdash [n \sim^A n]$ respectively,
    it suffices to show that, given $\vdash [\equiv n]$, we may deduce $\vdash \rho_A^*([n \sim^A n])$. Recall that the frame homomorphism $\rho_A^*$ sends $[n \sim^A n] \in \O G^\rmLH(\theory)_0$ to the corresponding generator of $\O A^{\ReprE^\rmLH(\theory)}$ and so the desired deduction follows from the relation $[\equiv n] \vdash [n \sim^A n]$.
    
    We now define a frame homomorphism $q^\ast \colon \O U \to \O V$ which is our candidate inverse to $p^*$. The map $q^*$ sends $[\equiv n']$ to $[n \sim^A  n']$ and the remaining frame generators of $\O U$ to the corresponding generators of $\O V$. Note that this indeed defines a frame homomorphism by the transitivity and symmetry axioms for $\sim^A$.
    To show that $p^\ast$ and $q^\ast$ are mutually inverse, it suffices to show that both composites $p^\ast q^\ast$ and $q^\ast p^\ast$ act as the identity on all frame generators.  For the composite $q^\ast p^\ast$, this is immediate.  For $p^\ast q^\ast$, the only non-trivial case is the generator $[\equiv n']$. This is sent under $p^\ast q^\ast$ to $[n \sim^A  n']$, which in the presence of the relations $\top \vdash [\equiv n]$, $[\equiv n] \land [n \sim^A  n'] \dashv \vdash [\equiv n'] \land [n \sim^A  n']$ and $[\equiv n] \wedge [\equiv n'] \vdash [n \sim^A n']$, is evidently equivalent to $[\equiv n']$, completing the proof.
    \end{proof}
\end{remark}

We now show that the generic bundles described above in the local homeomorphic and proper separated case are indeed, respectively, LH-models and PS-models, in the sense of \cref{def:model_over_groupoid}. It is clear they are $\Sigma$-structures of the appropriate form. It remains to prove that that there is an inclusion of the interpretation of formulae $\phi^{E^J(\theory)} \subseteq \psi^{E^J(\theory)}$ for each axiom $\phi \vdash_{x_1 : A_1, \dots , x_\ell : A_\ell} \psi$ of $\theory$.  To do so, we demonstrate that a geometric (or a dual geometric) formula $\phi$ is interpreted in $\ReprE^J(\theory)$ by the sublocale $\phi^{\ReprE^J(\theory)} \subseteq A_1^{\ReprE^J(\theory)} \times_{\ReprG^J(\theory)_0} \dots \times_{\ReprG^J(\theory)_0} A_\ell^{\ReprE^J(\theory)}$ obtained by adding to the generalised presentation the relation
    \[\top \vdash \bigvee_{p_1, \dots , p_\ell \in P^J} \phi_{p_1, \dots , p_\ell} \land [\equiv_1 p_1] \land \dots \land [\equiv_\ell p_\ell],\]
    where $\phi_{p_1,\dots,p_\ell}$ is the generator in the presentation of $A_1^{\ReprE^J(\theory)} \times_{\ReprG^J(\theory)_0} \dots \times_{\ReprG^J(\theory)_0} A_\ell^{\ReprE^J(\theory)}$ described in \cref{def:repr-grpd-LH}\footnote{Here we are conflating the open/closed combination of generators $\phi_{p_1, \dots , p_\ell}$ in the presentation of $\ReprG^J(\theory)_0$ with its image under $(\rho_{A_1} \times_{\ReprG^J(\theory)_0} \dots \times_{\ReprG^J(\theory)_0} \rho_{A_\ell})^\ast$.}.  
    For ease of understanding, this is done separately for $\rmLH$-{} and $\rmPS$-models in \cref{prop:generic_bundle_is_LH_model} and \cref{prop:generic_bundle_is_PS_model} below.
\begin{proposition}\label{prop:generic_bundle_is_LH_model}
    Let $\theory$ be a geometric theory.  The $\Sigma$-structure $\ReprE^\rmLH(\theory)$ described above is an LH-model of $\theory$ over $\ReprG^\rmLH(\theory)_0$.
\end{proposition}
\begin{proof}
    We claim that the interpretation of a geometric formula $\phi$ is the sublocale $\phi^{\ReprE^\rmLH(\theory)} \subseteq A_1^{\ReprE^\rmLH(\theory)} \times_{\ReprG^\rmLH(\theory)_0} \dots \times_{\ReprG^\rmLH(\theory)_0} A_\ell^{\ReprE^\rmLH(\theory)}$ obtained by adding the relation:
    \[\top \vdash \bigvee_{n_1, \dots , n_\ell \in \N} \phi_{n_1, \dots , n_\ell} \land [\equiv_1 n_1] \land \dots \land [\equiv_\ell n_\ell].\]
    We will argue by induction on the complexity of $\phi$.

    First, we must demonstrate that atomic formulae satisfy the inductive hypothesis.  If $\phi$ is a single relation symbol $R$, the claim is true by the definition of $R^{\ReprE^\rmLH(\theory)}$.
    
    The formula $(x=_A y)$ is interpreted in $\ReprE^\rmLH(\theory)$ by the (open) diagonal
    \[
    \Delta_{A^{\ReprE^\rmLH(\theory)}} \colon A^{\ReprE^\rmLH(\theory)} \hookrightarrow A^{\ReprE^\rmLH(\theory)} \times_{\ReprG^\rmLH(\theory)_0} A^{\ReprE^\rmLH(\theory)}.
    \]
    Note that $\Delta_{A^{\ReprE^\rmLH(\theory)}}^\ast$ acts by sending both of the generators $[\equiv_1 n], [\equiv_2 n] $ of $ A^{\ReprE^\rmLH(\theory)} \times_{\ReprG^\rmLH(\theory)_0} A^{\ReprE^\rmLH(\theory)}$ to $[\equiv n]$ and sending the remaining generators of the form $[(n_1, \dots , n_\ell) \in R]$ to the corresponding generators of $A^{\ReprE^\rmLH(\theory)}$.  Thus, $(x =_A y)^{\ReprE^\rmLH(\theory)}$, as a sublocale of $ A^{\ReprE^\rmLH(\theory)} \times_{\ReprG^\rmLH(\theory)_0} A^{\ReprE^\rmLH(\theory)}$, is obtained by adding, for all $n \in \N$, the relation $[\equiv_1 n] \dashv \vdash [\equiv_2 n]$.  We must show that an equivalent sublocale is obtained by adding the relation $\top \vdash \bigvee_{n,m \in \N} [n \sim^A m] \land [\equiv_1 n] \land [\equiv_2 m]$.

    Assume we have $[\equiv_1 n_1] \dashv \vdash [\equiv_2 n]$ for $n \in \N$, in addition to the other relations of $A^{\ReprE^\rmLH(\theory)} \times_{\ReprG^\rmLH(\theory)_0} A^{\ReprE^\rmLH(\theory)}$. Since we have $\top \vdash \bigvee_{n \in \N} [\equiv_1 n]$ and $\top \vdash \bigvee_{m \in \N} [\equiv_2 m]$, we find that $\top \vdash \bigvee_{n,m \in \N} {[\equiv_1 n]} \land {[\equiv_2 m]}$. Hence, by the assumption, $\top \vdash \bigvee_{n, m \in \N} [\equiv_1 n] \land [\equiv_1 m] \land [\equiv_2 m]$ and so $\top \vdash \bigvee_{n , m \in \N} [n \sim^A m] \land [\equiv_1 n] \land [\equiv_2 m]$. 
    
    Conversely, assuming $\top \vdash \bigvee_{n, m \in \N} [n \sim^A m] \land [\equiv_1 n] \land [\equiv_2 m] $, we obtain the derivation
    \begin{align*}
        [\equiv_1 n] & \vdash [\equiv_1 n] \land \!\!\bigvee_{n', m \in \N} [n' \sim^A m] \land [\equiv_1 n'] \land [\equiv_2 m], \\
        & \vdash \bigvee_{n' , m \in \N} [\equiv_1 n] \land [\equiv_1 n'] \land [n' \sim^A m] \land [\equiv_2 m], \\
        & \vdash \bigvee_{n',m\in \N} [n \sim^A n'] \land [n' \sim^A m] \land [\equiv_2 m], \\
        & \vdash \bigvee_{n', m \in \N} [n \sim^A m] \land [\equiv_2 m], \\
        & \vdash [\equiv_2 n],
    \end{align*}
    where we have used the relations $[n \sim^A n'] \land [n' \sim^A m] \vdash [n \sim^A m]$, $[\equiv_1 n] \land [\equiv_1 n'] \vdash [n \sim^A n']$ and $[n \sim^A m] \land [\equiv_2 m] \vdash [\equiv_2 n]$.  Via symmetry, we also obtain $[\equiv_2 n] \vdash [\equiv_1 n]$, completing the argument.

    Having argued for the base case of atomic formulae, we turn to formulae built using logical connectives.  Given formulae $\phi, \psi$, their conjunction $\phi \land \psi$ is interpreted by the intersection of sublocales
    \[\begin{tikzcd}
	{\phi^{\ReprE^\rmLH(\theory)} \cap \psi^{\ReprE^\rmLH(\theory)}} & {\psi^{\ReprE^\rmLH(\theory)}} \\
	{\phi^{\ReprE^\rmLH(\theory)}} & {A_1^{\ReprE^\rmLH(\theory)} \times_{\ReprG^\rmLH(\theory)_0} \dots \times_{\ReprG^\rmLH(\theory)_0} A_\ell^{\ReprE^\rmLH(\theory)}}
	\arrow[hook, from=1-1, to=1-2]
	\arrow[hook, from=1-1, to=2-1]
	\arrow[hook, from=1-2, to=2-2]
	\arrow[""{name=0, anchor=center, inner sep=0}, hook, from=2-1, to=2-2]
	\arrow["\lrcorner"{anchor=center, pos=0.125}, draw=none, from=1-1, to=0]
\end{tikzcd}\]
    which, by our inductive hypothesis, corresponds to the sublocale obtained by adding the relation
    \[
    \top \vdash \bigvee_{n_1, \dots , n_\ell \in \N} \phi_{n_1, \dots , n_\ell}
    \land [\equiv_1 n_1] \land \dots \land [\equiv_\ell n_\ell]
    \, \land \!\!\!\bigvee_{m_1 , \dots , m_\ell \in \N} \psi_{m_1 , \dots , m_\ell}
    \land [\equiv_1 m_1] \land \dots \land [\equiv_\ell m_\ell].
    \]
    Note that, in the presence of the other relations of $ A_1^{\ReprE^\rmLH(\theory)} \times_{\ReprG^\rmLH(\theory)_0} \dots \times_{\ReprG^\rmLH(\theory)_0} A_\ell^{\ReprE^\rmLH(\theory)}$, we have that
    \[
    \psi_{m_1 , \dots , m_\ell}  \land [\equiv_1 n_1] \land \dots \land [\equiv_\ell n_\ell] \land [\equiv_1 m_1] \land \dots \land [\equiv_\ell m_\ell ] \vdash \psi_{n_1 , \dots , n_\ell} .
    \]
Thus, there are equivalences
\begin{align*}
    & \bigvee_{n_1, \dots , n_\ell \in \N} \phi_{n_1, \dots , n_\ell} \land [\equiv_1 n_1] \land \dots \land [\equiv_\ell n_\ell] \, \land \!\!\! \bigvee_{m_1 , \dots ,m_\ell \in \N} \psi_{m_1 , \dots , m_\ell} \land [\equiv_1 m_1] \land \dots \land [\equiv_\ell m_\ell], \\
    \equiv & \bigvee_{n_1, \dots , n_\ell \in \N} \! \bigvee_{m_1 , \dots , m_\ell \in \N} \phi_{n_1, \dots , n_\ell} \land \psi_{m_1 , \dots , m_\ell} \land [\equiv_1 n_1] \land \dots \land [\equiv_\ell n_\ell]   \land [\equiv_1 m_1] \land \dots \land [\equiv_\ell m_\ell], \\
    \equiv & \bigvee_{n_1, \dots , n_\ell \in \N }  \phi_{n_1, \dots , n_\ell} \land \psi_{n_1, \dots , n_\ell} \land [\equiv_1 n_1] \land \dots \land [\equiv_\ell n_\ell] ,
\end{align*}
proving that $(\phi \land \psi)^{\ReprE^\rmLH(\theory)}$ is presented as a sublocale by adding the relation 
\[
\top \vdash  \bigvee_{n_1, \dots , n_\ell \in \N }  \phi_{n_1, \dots , n_\ell} \land \psi_{n_1, \dots , n_\ell} \land [\equiv_1 n_1] \land \dots \land [\equiv_\ell n_\ell],
\]
as desired.

    The formula $\bigvee_i \phi_i$ is interpreted as the union of sublocales $\bigcup_i \phi_i^{\ReprE^\rmLH(\theory)} \subseteq A_1^{\ReprE^\rmLH(\theory)} \times_{\ReprG^\rmLH(\theory)_0}  \dots \times_{\ReprG^\rmLH(\theory)_0}  A_\ell^{\ReprE^\rmLH(\theory)}$.  By our inductive hypothesis, $(\bigvee_i \phi_i)^{\ReprE^\rmLH(\theory)}$ is the sublocale defined by adding the relation
\begin{align*}
    \top & \vdash \bigvee_{i }  \bigvee_{n_1, \dots , n_\ell \in \N} (\phi_i)_{n_1, \dots , n_\ell} \land [\equiv_1 n_1] \land \dots \land [\equiv_\ell n_\ell], \\
    & \equiv \bigvee_{n_1, \dots , n_\ell \in \N} \left( \bigvee_i (\phi_i)_{n_1, \dots , n_\ell} \right) \land [\equiv_1 n_1] \land \dots \land [\equiv_\ell n_\ell],
\end{align*}
as desired.

Finally, the formula $\exists x_1 \in A_1.\ \chi(x_1,x_2, \dots, x_\ell) $ is interpreted as the image of the composite
    \begin{equation}\label{eq:composite_wedge_R_i_then_projection}
        \chi^{A^{\ReprE^\rmLH(\theory)}} \hookrightarrow A_1^{\ReprE^\rmLH(\theory)} \times_{\ReprG^\rmLH(\theory)_0} \cdots \times_{\ReprG^\rmLH(\theory)_0} A_\ell^{\ReprE^\rmLH(\theory)} \xrightarrow{\pi_{2,\dots,n}} A_2^{\ReprE^\rmLH(\theory)} \times_{\ReprG^\rmLH(\theory)_0} \cdots \times_{\ReprG^\rmLH(\theory)_0} A_\ell^{\ReprE^\rmLH(\theory)},
    \end{equation}
    where $\chi^{A^{\ReprE^\rmLH(\theory)}}$ is the sublocale obtained by adding the relation $\top \vdash \bigvee_{n_1, \dots, n_\ell \in \N} \chi_{n_1, \dots, n_\ell} \land {[\equiv_1 n_1]} \land \dots \land [\equiv_\ell n_\ell]$ by our inductive hypothesis. Recall from \cref{sec:dual} that taking the image along the map \eqref{eq:composite_wedge_R_i_then_projection} amounts to existentially quantifying over the first variable in the defining relation of $\chi^{A^{\ReprE^\rmLH(\theory)}}$.
    So $(\exists x_1 . A_1 \, \chi)^{\ReprE^\rmLH(\theory)}$ is obtained by taking the join
    \begin{align*}
        \top & \vdash \bigvee_{n_1 \in \N} \left( \bigvee_{n_1, \dots, n_\ell \in \N} \chi_{n_1, \dots, n_\ell} \land {[\equiv_1 n_1]} \land \dots \land [\equiv_\ell n_\ell] \right), \\ %
        & \equiv \bigvee_{n_2, \dots , n_\ell \in \N} \left( \bigvee_{n_1 \in \N} \chi_{n_1, \dots , n_\ell} \land [\equiv_1 n_1]\right) \land [\equiv_2 n_2] \land \dots \land [\equiv_\ell n_\ell], \\
        & \equiv \bigvee_{n_2, \dots , n_\ell \in \N} (\exists x_1.A_1 \, \chi)_{n_2, \dots , n_\ell}  \land [\equiv_2 n_2] \land \dots \land [\equiv_\ell n_\ell],
    \end{align*}
    where the last equivalence holds since $\bigvee_{n_1 \in \N} [\equiv_1 n_1] \equiv \top$ and by the definition of $(\exists x_1.A_1 \, \chi)_{n_2, \dots , n_\ell}$.
    
    Therefore, the interpretation of a geometric formula $\phi$ is the sublocale obtained by adding the relation $\top \vdash \bigvee_{n_1, \dots , n_\ell \in \N} \phi_{n_1, \dots , n_\ell} \land [\equiv_1 n_1] \land \dots \land [\equiv_\ell n_\ell]$ as claimed.  Thus, for each axiom $\phi \vdash_{x_1\colon A_1,\dots,x_k\colon A_k} \psi $ of $\theory$, by the axioms
  \[[\equiv_i n_i ] \vdash [n_i \sim^{A_i} n_i] \quad \text{ and } \quad \bigwedge_{i=1}^k {[n_i \sim^{A_i} n_i]} \wedge \phi_{n_1,\dots,n_k} \vdash \psi_{n_1,\dots,n_k}\]
    we deduce that $\phi^{\ReprE^\rmLH(\theory)} \subseteq \psi^{\ReprE^\rmLH(\theory)}$.  Combined with the fact that the carrier $\rho_A \colon A^{\ReprE^\rmLH(\theory)} \to \ReprG^\rmLH(\theory)_0$ is a local homeomorphism by \cref{prop:each-generic-sort-is-LH}, and the observation above that each basic relation is interpreted by an open sublocale, we deduce that $\ReprE^\rmLH(\theory)$ is indeed an LH-model of $\theory$ as desired.
\end{proof}

\begin{proposition}\label{prop:generic_bundle_is_PS_model}
    Let $\theory'$ be a dual geometric theory.  The $\Sigma$-structure $\ReprE^\rmPS(\theory')$ from above is a PS-model of $\theory'$ over $\ReprG^\rmPS(\theory')_0$.
\end{proposition}
\begin{proof}
    By following the analogous inductive argument to that provided in \cref{prop:generic_bundle_is_LH_model}, the PS case can be proved as well.  The only additional inductive step that is required in the PS case concerns infinite conjunctions.  Suppose that for each $i \in I$ the dual geometric formula $\phi_i$ is interpreted in the model $\ReprE^\rmPS(\theory')$ by the closed sublocale of $A^{\ReprE^\rmPS(\theory)}_1 \times_{\ReprG^\rmPS(\theory')_0} \dots \times_{\ReprG^\rmPS(\theory')_0} A^{\ReprE^\rmPS(\theory)}_1 $ given by adding to its closed-type generalised presentation the relation 
    \[\top \vdash \bigvee_{s_1, \dots , s_\ell \in 2^\N} {(\phi_i)}_{s_1, \dots , s_\ell} \land [\equiv_1 s_1] \land \dots \land [\equiv_\ell s_\ell].\]
    Finite meets work as in \cref{prop:generic_bundle_is_LH_model} and so it suffices to consider directed meets.
    The interpretation of the directed conjunction $\bigwedge_I \phi_i$, i.e.\ the directed intersection of sublocales $\bigcap_I {\phi_i}^{\ReprE^\rmPS(\theory')}$, is given by the relation 
    \begin{align*}
        \top & \vdash \bigwedge_{i \in I} \bigvee_{s_1, \dots , s_\ell \in 2^\N} {(\phi_i)}_{s_1, \dots , s_\ell} \land [\equiv_1 s_1] \land \dots \land [\equiv_\ell s_\ell], \\
        & \equiv \bigvee_{s_1, \dots , s_\ell \in 2^\N} \bigwedge_{i \in I}  {(\phi_i)}_{s_1, \dots , s_\ell} \land [\equiv_1 s_1] \land \dots \land [\equiv_\ell s_\ell], \\
        & \equiv \bigvee_{s_1, \dots , s_\ell \in 2^\N} \left( \bigwedge_{i \in I}  {\phi_i}\right)_{s_1, \dots , s_\ell} \land [\equiv_1 s_1] \land \dots \land [\equiv_\ell s_\ell],
    \end{align*}
    where the first equivalence follows from the fact that right adjoint of the inverse image of the proper projection defining the join over $2^\N$ preserves directed joins of opens, and hence, directed meets of closeds.
    This completes the proof.
\end{proof}

We now upgrade the locale $\ReprG^J(\theory)_0$ into a localic category $\ReprG^J(\theory)$ in order to capture morphisms between models. We will then give an action of this category on $\ReprE^J(\theory)$ making it into an LH-model/PS-model of $\theory$ over $\ReprG^J(\theory)$. This is the model that we will later show to universal amongst all such models.

\begin{definition}[{cf.\ \cite[\S 2]{manuell2023representing}}]\label{def:generic-LH-localic-category}
    Let $J$ be either $\rmLH$ or $\rmPS$ and let $\theory$ be a geometric theory or a dual geometric theory, as appropriate.  We define $\ReprG^J(\theory)$ to be the following localic category.
    \begin{enumerate}
        \item The locale of objects is the locale $\ReprG^J(\theory)_0$ from \cref{def:repr-grpd-LH}.
        \item The locale of arrows $\ReprG^J(\theory)_1$ has the following generalised frame presentation.
        \begin{enumerate}
            \item We take two copies of the presentation for $\ReprG^J(\theory)_0$ (including generators $[p \sim^A_1 p'], [q \sim^A_2 q']$ for each sort $A$ and $p,p',q,q' \in P^J$, and their associated relations).
            \item For each sort $A$ and $p, q \in P^J$, we add a generator $[\alpha^A(p) = q]$ together with the axioms
            \begin{align*}
                [\alpha^A(p) = q] \land [\alpha^A(p') = q'] \land [p \sim^A_1 p'] & \vdash [q \sim^A_2 q'], & \text{for $p,p',q,q' \in P^J$,} \\
                [p \sim^A_1 p] & \vdash \bigvee_{q \in P^J} [\alpha^A(p) = q], & \text{for $p \in P^J$,}
            \end{align*}
            and, for each relation symbol $R \subseteq A_1 \times \dots \times A_\ell$ and $p_1, \dots p_\ell, q_1, \dots , q_\ell \in P^J$, the axioms
            \[
[(p_1, \dots , p_\ell) \in R_1] \land \bigwedge_{i=1}^\ell [\alpha^{A_i}(p_i) = q_i] \vdash [(q_1, \dots , q_\ell) \in R_2].
\]
            Thus, the generators $[\alpha^A(p) = q]$ describe a function between the respective subquotients of $P^J$ that moreover preserves the $\theory$-model structure (see \cref{rem:points_of_classifying_cat}).
        \end{enumerate}
        \item The source, target and identity maps as defined as follows (where we think of $[p \sim^A p']$ as $[(p,p') \in R]$ for $R = ({=_A})$):
  \begin{align*}
      s^*\colon {[(p_1,\dots,p_k) \in R]} &\mapsto {[(p_1,\dots,p_k) \in R_1]}.
  \end{align*}
  \begin{align*}
      t^*\colon {[(q_1,\dots,q_k) \in R]} &\mapsto {[(q_1,\dots,q_k) \in R_2]}.
  \end{align*}
  \begin{align*}
      e^*\colon {[(p_1,\dots,p_k) \in R_1]} &\mapsto {[(p_1,\dots,p_k) \in R]}, \\
      e^*\colon {[(q_1,\dots,q_k) \in R_2]} &\mapsto {[(q_1,\dots,q_k) \in R]}, \\
      e^*\colon {[\alpha^A(p) = q]} &\mapsto {[p \sim^A q]}.
  \end{align*}
  \item The pullback locale $\ReprG^J(\theory)_1 \times_{\ReprG^J(\theory)_0} \ReprG^J(\theory)_1$ can be given a generalised presentation whose generating locale consists of three copies of the generators for $\ReprG^J(\theory)_0$ and further generators $[\beta^A(p) = q]$, $[\gamma^A(q) = r]$, for each sort $A$ and $p, q, r \in P^J$.  The composition map $m \colon \ReprG^J(\theory)_1 \times_{\ReprG^J(\theory)_0} \ReprG^J(\theory)_1 \to \ReprG^J(\theory)_1$ is defined by
  \begin{align*}
      m^*\colon {[(p_1,\dots,p_k) \in R_1]} &\mapsto {[(p_1,\dots,p_k) \in R_1]}, \\
      m^*\colon {[(q_1,\dots,q_k) \in R_2]} &\mapsto {[(q_1,\dots,q_k) \in R_3]} ,\\
      m^*\colon {[\alpha^A(p) = r]} &\mapsto \bigvee_{q \in P} {[\beta^A(p) = q]} \wedge {[\gamma^A(q) = r]}.
  \end{align*}
  (Intuitively, the map $m^\ast$ is sending $\beta^A$ and $\gamma^A$ to their relational composite.)
    \end{enumerate}
\end{definition}
\begin{remark}[{cf.\ \cite[Remark 2.7]{manuell2023representing}}]\label{rem:points_of_classifying_cat}
    Recall from \cref{rem:points_of_generic_bundle} that points of the locale of objects $\ReprG^\rmLH(\theory)_0$ correspond to models of $\theory$ where each sort is presented as a subquotient of $\N$.  A point of the locale of arrows $\ReprG^\rmLH(\theory)_1$ consists of a homomorphism $\alpha$ between two such models, with the functions $\alpha^A \colon \N / {\sim}^A_1 \to \N/{\sim}^A_2$ mapping between the respective subquotients.  Similarly, the points of $\ReprG^\rmPS(\theory)_1$ correspond to homomorphisms between $\theory$-models whose sorts are presented as (closed) subquotients of the Cantor space $2^\N$.
\end{remark}
We now describe the action of this category on our bundle $\ReprE^J(\theory)$.
\begin{definition}
    Note that the pullback locale $A^{\ReprE^J(\theory)} \times_{\ReprG^J(\theory)_0} \ReprG^J(\theory)_1$ has generators of the form $[(p_1,\dots,p_k) \in R_1]$, $[(q_1,\dots,q_k) \in R_2]$, $[\alpha(p) = q]$ and $[\equiv p]$.
    The action $\theta^{A^{\ReprE^J(\theory)}} \colon A^{\ReprE^J(\theory)} \times_{\ReprG^J(\theory)_0} \ReprG^J(\theory)_1 \to A^{\ReprE^J(\theory)}$ is defined by
\begin{alignat*}{2}
  {\theta^{A^{\ReprE^J(\theory)}}}^*\colon && {[(q_1,\dots,q_k) \in R]} &\mapsto {[(q_1,\dots,q_k) \in R_2]}, \\
  {\theta^{A^{\ReprE^J(\theory)}}}^*\colon && {[\equiv q]} &\mapsto \bigvee_{p \in P^J} {[\equiv p]} \wedge {[\alpha(p) = q]}.
\end{alignat*}
\end{definition}
\begin{lemma}
    For each relation symbol $R$ of the (dual) geometric theory $\theory$, the interpretation $R^{\ReprE^J(\theory)} \subseteq A_1^{\ReprE^J(\theory)} \times_{\ReprG^J(\theory)_0} \dots \times_{\ReprG^J(\theory)_0} A_n^{\ReprE^J(\theory)}$ is stable under the action of $\theta^{A^{\ReprE^J(\theory)}}$.
\end{lemma}
\begin{proof}
    For simplicity, we will deal with the case of a unary predicate $R \subseteq A$ (the $\ell$-ary case is almost identical).  To demonstrate that the action $\theta^{A^{\ReprE^J(\theory)}}$ restricts to a map $R^{\ReprE^J(\theory)} \times_{\ReprG^J(\theory)_0} \ReprG_{\cong}^J(\theory)_1 \to R^{\ReprE^J(\theory)}$, we must show that the image under $(\theta^{A^{\ReprE^J(\theory)}})^*$ of the relation $\top \vdash \bigvee_{q \in P^J} [q \in R] \land [\equiv q]$ that defines $R^{\ReprE^J(\theory)} \subseteq A^{\ReprE^J(\theory)} $ is derivable from the presenting relations of $R^{\ReprE^J(\theory)} \times_{\ReprG^J(\theory)_0} \ReprG^J(\theory)_1$.  That is, we must derive
    \[
    \top \vdash \bigvee_{q \in P^J} \bigvee_{p \in P^J} [\equiv p] \land [\alpha(p) = q] \land [q \in R_2]
    \]
    from the relations in \cref{def:generic-LH-localic-category} and $\top \vdash \bigvee_{p \in P^J} [\equiv p] \land [p \in R_1]$.  But this is immediate from the relations $[\equiv p] \vdash [p \sim_1 p]$, $[p \sim_1 p] \vdash \bigvee_{q \in P^J} [\alpha(p) = q]$ and $[p \in R_1] \land [\alpha(p) = q] \vdash [q \in R_2]$.
\end{proof}
\begin{corollary}\label{prop:E(T)-is-stable-under-action}
    With this action, if $\theory$ is a geometric theory, then $\ReprE^\rmLH(\theory)$ is an LH-model over the localic category $\ReprG^\rmLH(\theory)$, and similarly, for a dual geometric theory $\theory'$, $\ReprE^\rmPS(\theory')$ is a PS-model over $\ReprG^\rmPS(\theory')$.
\end{corollary}

Finally, we note how the above description given for the localic category $\ReprG^\rmLH(\theory)$ (respectively, $\ReprG^\rmPS(\theory)$) is easily modified to obtain instead the core localic groupoid $\ReprG^\rmLH_{\cong}(\theory)$ (resp., $\ReprG^\rmPS_{\cong}(\theory)$).  Just as in \cref{prop:E(T)-is-stable-under-action}, the $\Sigma$-structure $\ReprE^\rmLH(\theory)$ defines an LH-model of $\theory$ over $\ReprG^\rmLH_{\cong}(\theory)$ and $\ReprE^\rmPS(\theory')$ defines a PS-model over $\ReprG^\rmPS(\theory')$.
\begin{definition}[{\cite[\S 2]{manuell2023representing}}]\label{def:generic-localic-groupoid}
   We use $\ReprG^J_{\cong}(\theory)$ to denote the core groupoid of $\ReprG^J(\theory)$, which is obtained as follows:
    \begin{enumerate}
  \item The locale of arrows $\ReprG^J_{\cong}(\theory)_1$ is the locale obtained by adding to the presentation of $\ReprG^J(\theory)_1$ the axioms
  \begin{align*}
     [\alpha^A(p) = q] \land [\alpha^A(p') = q'] \land [q \sim^A_2 q'] &\vdash [p \sim^A_1 p'] & \text{for $p,p',q,q' \in P^J$}, \\
     [q \sim^A_2 q] &\vdash \bigvee_{p \in P^J} [\alpha^A(p) = q] & \text{for $q \in P^J$},
  \end{align*}
    for each sort $A$, and for each relation symbol $R \subseteq A_1 \times \dots \times A_\ell$ and points $p_1, \dots p_\ell, q_1, \dots , q_\ell \in P^J$, the axiom
            \[
            [(q_1, \dots , q_\ell) \in R_2] \land \bigwedge_{i = 1}^\ell [\alpha^{A_i}(p_i) = q_i] \vdash [(p_1, \dots , p_\ell) \in R_1].
            \]
  \item The inversion map $i \colon \ReprG^J_{\cong}(\theory)_1 \to \ReprG^J_{\cong}(\theory)_1$ is defined by:
  \begin{align*}
      i^*\colon {[(p_1,\dots,p_k) \in R_1]} &\mapsto {[(p_1,\dots,p_k) \in R_2]}, \\
      i^*\colon {[(q_1,\dots,q_k) \in R_2]} &\mapsto {[(q_1,\dots,q_k) \in R_1]} ,\\
      i^*\colon {[\alpha^A(p) = q]} &\mapsto {[\alpha^A(q) = p]}.
  \end{align*}
 \end{enumerate}
\end{definition}
\begin{remark}
    Note that a direct construction of the core groupoid would add generators to $\ReprG^J(\theory)$ of the form $[(\alpha^A)^{-1}(q) = p]$ satisfying similar axioms to $\alpha^A$ in addition to ones imposing that $(\alpha^A)^{-1} \circ \alpha^A = {\sim_1^A}$ and $\alpha^A \circ (\alpha^A)^{-1} = {\sim_2^A}$. From this we can deduce that $[(\alpha^A)^{-1}(q) = p] = [\alpha^A(p) = q]$. The axioms then reduce to those given in \cref{def:generic-localic-groupoid} above. This is exactly like how a function is invertible if and only if it is bijective.
\end{remark}

\section{Descent for discrete opfibrations}\label{sec:descent_dopf}

The aim of this section is to prove the following proposition:
\begin{proposition}\label{prop:mod_extends_to_anafunctors}
  The pseudofunctor $\Sstruct(-) \colon \LocCat\op \to \CAT$ from \cref{sec:base_change} sends weak equivalences of localic categories to equivalences in $\CAT$ and hence, by \cite[Theorem 21]{pronk1996bicategories}, extends to a pseudofunctor $\LocCatAna\op \to \CAT$, which we will also call $\Sstruct(-)$.
\end{proposition}
We will prove this using descent.
Recall that a morphism $f\colon X \to Y$ in $\Loc$ (say) is an effective descent morphism if pullback along $f$ induces an equivalence between the slice category $\Loc/Y$ (i.e.\ discrete opfibrations over $Y$ viewed as a categorically discrete localic category) and the category of bundles (i.e.\ discrete opfibrations) over the groupoid induced by the kernel pair of $f$.

In fact, this immediately essentially proves the result for weak equivalences $\Phi$ with categorically discrete codomain, since we have that $\Phi_0$ is an effective descent morphism and the induced category structure on $X$ making $\Phi$ fully faithful as in \cref{rem:ff_induced_domain} is precisely the kernel pair of $\Phi_0$.

To generalise this result to non-discrete codomain we will study the descent of internal discrete opfibrations in categories of internal categories.

\begin{lemma}\label{lem:ff_eso_eff_desc}
  Let $\Phi\colon \H \to \K$ be an internal functor in $\Cat(\C)$ for a finitely complete category $\C$. If $\Phi$ is fully faithful and $\Phi_0$ is of effective descent in $\C$, then $\Phi$ is an effective descent morphism in $\Cat(\C)$ for which $\Phi_1$ is also of effective descent in $\C$.
\end{lemma}
\begin{proof}
  By \cite[Corollary 3.3.1]{creurer1999descent} it suffices to show that $\Phi_0,\Phi_1,\Phi_2$ and $\Phi_3$ are effective descent morphisms, where 
  \begin{align*}
      \Phi_2\colon  H_1 \times_{H_0} H_1 & \to K_1 \times_{K_0} K_1, \\
      \Phi_3\colon  H_1 \times_{H_0} H_1 \times_{H_0} H_1 & \to K_1 \times_{K_0} K_1 \times_{K_0} K_1
  \end{align*}
  are the induced maps between composable pairs and composable triples respectively.
  
  Since $\Phi$ is fully faithful, the following diagram is a pullback.\begin{equation}\label{eq:fullfaithfullness}\tag{$\ast$}
    \begin{tikzcd}[ampersand replacement=\&]
	{H_1} \& {K_1} \\
	{H_0 \times H_0} \& {K_0 \times K_0}
	\arrow["{(s,t)}"', from=1-1, to=2-1]
	\arrow["{\Phi_1}", from=1-1, to=1-2]
	\arrow["{(s,t)}", from=1-2, to=2-2]
	\arrow["{\Phi_0 \times \Phi_0}"', from=2-1, to=2-2]
	\arrow["\lrcorner"{anchor=center, pos=0.125}, draw=none, from=1-1, to=2-2]
    \end{tikzcd}
  \end{equation}
  Noting that $\Phi_0 \times \Phi_0$ is an effective descent morphism (since it can be obtained as a composite of $\id_{H_0} \times \Phi_0$ and $\Phi_0 \times \id_{K_0}$, which are pullbacks of $\Phi_0$) we see that $\Phi_1$ is of effective descent too.
  
  To see that $\Phi_2$ is an effective descent morphism we will show that the square
  \[\begin{tikzcd}
	{H_1\times_{H_0}H_1} & {K_1 \times_{K_0} K_1} \\
	{H_0^3} & {H_0^3}
	\arrow["{\Phi_2}", from=1-1, to=1-2]
	\arrow["{\Phi_0^{\times 3}}"', from=2-1, to=2-2]
	\arrow["{(s\pi_1,t\pi_1,t\pi_2)}", from=1-2, to=2-2]
	\arrow["{(s\pi_1,t\pi_1,t\pi_2)}"', from=1-1, to=2-1]
\end{tikzcd}\]
   is a pullback.

   Let $(x_1,x_2,x_3)\colon X \to H_0^3$ and $(f_1,f_2)\colon X \to K_1 \times_{K_0} K_1$ be morphisms such that 
   \[
   \Phi_0^{\times 3} \circ (x_1,x_2,x_3) = (s\pi_1,t\pi_1,t\pi_2) \circ (f_1,f_2)
   \]
   --- that is, $\Phi_0 x_1 = s f_1$, $\Phi_0 x_2 = t f_1$ and $\Phi_0 x_3 = t f_2$. We must show that there is a unique map $(u_1,u_2)\colon X \to H_1 \times_{H_0} H_1$ such that $(s\pi_1,t\pi_1,t\pi_2) \circ (u_1,u_2) = (x_1,x_2,x_3)$ and $\Phi_2 (u_1,u_2) = (f_1,f_2)$ --- that is, $s u_1 = x_1$, $t u_1 = x_2$ and $t u_2 = x_3$, and $\Phi_1 u_1 = f_1$ and $\Phi_1 u_2 = f_2$.
  
  Compose $(x_1,x_2,x_3)$ with the map $(\pi_1,\pi_2)\colon H_0^3 \to H_0^2$ to give $(x_1,x_2)\colon X \to H_0^2$ and $(f_1,f_2)\colon X \to K_1$ with $\pi_1\colon K_1 \times_{K_0} K_1 \to K_1$ to give $f_1$. Note that $\Phi_0 x_1 = s f_1$ and $\Phi_0 x_2 = t f_1$. Thus, by the universal property of the pullback (\ref{eq:fullfaithfullness}) we have a unique map $u_1\colon X \to H_1$ such that $(x_1,x_2) = (s,t) \circ u_1$ (i.e.\ $x_1 = s u_1$ and $x_2 = t u_1$) and $f_1 = \Phi_1 u_1$.

  Similarly, we can compose $(x_1,x_2,x_3)$ with $(\pi_2,\pi_3)\colon H_0^3 \to H_0^2$ to give $(x_2,x_3)$ and $(f_1,f_2)\colon X \to K_1$ with $\pi_2\colon K_1 \times_{K_0} K_1 \to K_1$ to give $f_2$. Then $\Phi_0 x_2 = t f_1 = s f_2$ and $\Phi_0 x_3 = t f_2$ and so again by (\ref{eq:fullfaithfullness}) we have a unique $u_2\colon X \to H_1$ such that $x_2 = s u_2$, $x_3 = t u_2$, $f_2 = \Phi_1 u_2$.

  Now $u_1$ and $u_2$ induce a well-defined map into $H_1 \times_{H_0} H_1$ since $t u_1 = x_2 = s u_2$. Moreover, this map is unique since $u_1$ and $u_2$ were the unique maps satisfying their conditions.
  
  A similar approach shows that $\Phi_3$ is also an effective descent morphism.
\end{proof}

\begin{lemma}\label{lem:disc_opfibrations_descend} %
  Let $\Phi\colon \H \to \K$ be an effective descent morphism in $\Cat(\C)$ such that $\Phi_1$ is effective descent in $\C$. Then discrete opfibrations descend along $\Phi$.
\end{lemma}
\begin{proof}
  Let $p \colon \X \to \K$ be an internal functor and let $p'\colon \X' \to \H$ be the pullback along $\Phi$.  We want to show that if $p'$ is a discrete opfibration, then so is $p$ --- i.e.\  if the commutative square
  \[\begin{tikzcd}
	{X'_1} & {X'_0} \\
	{H_1} & {H_0}
	\arrow["s"', from=2-1, to=2-2]
	\arrow["{p'_0}", from=1-2, to=2-2]
	\arrow["{p'_1}"', from=1-1, to=2-1]
	\arrow["s", from=1-1, to=1-2]
	\arrow["\lrcorner"{anchor=center, pos=0.125}, draw=none, from=1-1, to=2-2]
    \end{tikzcd}\]
  is a pullback, then the square
    \[\begin{tikzcd}
	{X_1} & {X_0} \\
	{K_1} & {K_0}
	\arrow["s"', from=2-1, to=2-2]
	\arrow["{p_0}", from=1-2, to=2-2]
	\arrow["{p_1}"', from=1-1, to=2-1]	\arrow["s", from=1-1, to=1-2]
    \end{tikzcd}\]
  is also a pullback.
  
  We form the composite pullback
  \[\begin{tikzcd}
	{X'_1} & {X'_0} & {X_0} \\
	{H_1} & {H_0} & {K_0} \rlap{\,.}
	\arrow["s"', from=2-1, to=2-2]
	\arrow["{p'_0}"', from=1-2, to=2-2]
	\arrow["{p'_1}"', from=1-1, to=2-1]
	\arrow["s", from=1-1, to=1-2]
	\arrow["\lrcorner"{anchor=center, pos=0.125}, draw=none, from=1-1, to=2-2]
	\arrow["{\Phi_0}"', from=2-2, to=2-3]
	\arrow["{p_0}", from=1-3, to=2-3]
	\arrow["{\Phi'_0}", from=1-2, to=1-3]
	\arrow["\lrcorner"{anchor=center, pos=0.125}, draw=none, from=1-2, to=2-3]
  \end{tikzcd}\]
  This is equal to the composite
  \[\begin{tikzcd}
	{X'_1} & {X_1} & {X_0} \\
	{H_1} & {K_1} & {K_0},
	\arrow["s"', from=2-2, to=2-3]
	\arrow["{{p_0}}", from=1-3, to=2-3]
	\arrow["{{p_1}}"', from=1-2, to=2-2]
	\arrow["s", from=1-2, to=1-3]
	\arrow["{\Phi_1}"', two heads, from=2-1, to=2-2]
	\arrow["{p'_1}"', from=1-1, to=2-1]
	\arrow["{\Phi'_1}", two heads, from=1-1, to=1-2]
	\arrow["\lrcorner"{anchor=center, pos=0.125}, draw=none, from=1-1, to=2-2]
  \end{tikzcd}\]
  which is thus also a pullback, and so the rectangle and the left-hand square in this diagram are pullbacks. Since $\Phi_1$ is a pullback-stable extremal epimorphism, it follows that the right-hand square is also a pullback (see \cite[Proposition 1.6]{Janelidze_Sobral_Tholen_2003}) concluding the proof.
\end{proof}

We feel that the following result must somehow be well-known, especially to experts in stacks. However, we have not been able to locate any proof in the literature. Probably the most relevant paper is \cite{bungepare}, but it refers to a particular special case.
\begin{theorem}\label{prop:bundles_descend}
    If $\Phi \colon \H \to \K$ is an internal functor satisfying the hypothesis of \cref{lem:ff_eso_eff_desc}, then pullback along $\Phi$ induces an equivalence
    \[
    {\bf doFib}(\H) \simeq {\bf doFib}(\K)
    \]
    between the categories of discrete opfibrations over $\H$ and $\K$.
\end{theorem}
\begin{proof}
    By \cref{lem:ff_eso_eff_desc} we know $\Phi$ is an effective descent morphism of categories. Moreover, discrete opfibrations descend by \cref{lem:disc_opfibrations_descend} and so there is an equivalence between the discrete opfibrations over $\K$ and the discrete opfibrations over $\H$ equipped with `descent data' (i.e.\ an action by the kernel equivalence relation).  We will show that the descent data is superfluous in the sense that it is uniquely induced (and morphisms automatically respect it).
    
    Let $\Pi^{1,2} \colon \H \times_\K \H \rightrightarrows \H$ denote the kernel pair of $\Phi$, i.e.\ $ \H\times_\K \H $ is the category obtained as in the prism of pullbacks
    \[\begin{tikzcd}
	{H_1 \times_{K_1} H_1} &[-30pt]& {H_1} \\
	& {H_0 \times_{K_0} H_0} && {H_0} \\
	{H_1} && {K_1} \\
	& {H_0} && {K_0}
	\arrow["{\Pi^1_1}", from=1-1, to=1-3]
	\arrow["t"{description, pos=0.4}, shift right=2, shorten >=10pt, from=1-1, to=2-2]
	\arrow["s"{description}, shift left=2, shorten <=10pt, from=1-1, to=2-2]
	\arrow["{\Pi^2_1}"', from=1-1, to=3-1]
	\arrow["s"{description}, shift left=2, from=1-3, to=2-4]
	\arrow["t"{description}, shift right=2, from=1-3, to=2-4]
	\arrow["{\Phi_1}"{pos=0.7}, from=1-3, to=3-3]
	\arrow["{\Phi_0}", from=2-4, to=4-4]
	\arrow["{\Phi_1}"{pos=0.7}, from=3-1, to=3-3]
	\arrow["t"{description}, shift right=2, shorten >=0pt, from=3-1, to=4-2]
	\arrow["s"{description}, shift left=2, shorten <=5pt, from=3-1, to=4-2]
	\arrow["t"{description}, shift right=2, from=3-3, to=4-4]
	\arrow["s"{description}, shift left=2, from=3-3, to=4-4]
	\arrow["{\Phi_0}"', from=4-2, to=4-4]
 \arrow["{\Pi^2_0}"'{pos=0.3}, from=2-2, to=4-2, crossing over]
 \arrow["{\Pi^1_0}"{pos=0.3}, from=2-2, to=2-4, crossing over]
\end{tikzcd}\]
That is to say, as generalised points, the objects of $\H \times_\K \H$ are pairs $(a,b) \in H_0^2$ such that $\Phi_0(a) = \Phi_0(b)$, and arrows are pairs $(f,g) \in H_1^2$ such that $\Phi_1(f) = \Phi_1(g)$.  The functors $\Pi^{1,2}$ define an internal groupoid in the category of internal categories $\Cat(\C)$ as in the diagram
\[\begin{tikzcd}
	{\H \times_\K \H \times_\K \H} & {\H \times_\K \H} & \H .
	\arrow["{\Pi^{1,2}}", shift left=3, from=1-1, to=1-2]
	\arrow["{\Pi^{2,3}}"', shift right=3, from=1-1, to=1-2]
	\arrow["{\Pi^{1,3}}"{description}, from=1-1, to=1-2]
	\arrow["{\Pi^1}", shift left=3, from=1-2, to=1-3]
	\arrow["{\Pi^2}"', shift right=3, from=1-2, to=1-3]
	\arrow["\Delta"{description}, from=1-3, to=1-2]
\end{tikzcd}\]

Let $p \colon \X \to \H$ be a discrete opfibration. Then
\[\begin{tikzcd}
	{X_1} & {X_0} \\
	{H_1} & {H_0}
	\arrow["s", from=1-1, to=1-2]
	\arrow["{{p_1}}"', from=1-1, to=2-1]
	\arrow["\lrcorner"{anchor=center, pos=0.125}, draw=none, from=1-1, to=2-2]
	\arrow["{{p_0}}", from=1-2, to=2-2]
	\arrow["s"', from=2-1, to=2-2]
\end{tikzcd}\]
is a pullback.  The target map $t \colon  X_1 \cong X_0 \times_{H_0} H_1 \to X_0$ of $\X$ is precisely an $H_1$-action on the bundle $p_0 \colon X_0 \to H_0$.  Suppose that $p \colon \X \to \H$ is endowed with a descent datum for $\Phi$, or equivalently an action by $\H \times_\K \H$.  This is a functor $\chi \colon \X \times_\H \H \times_\K \H \cong \X \times_\K \H \to \X$ satisfying the action equations from \cref{def:sheaf_over_localic_cat}. We will show that $\chi$ is uniquely determined by $p$.

First, consider the form of the map $\chi_1 \colon X_1 \times_{K_1} H_1 \cong X_0 \times_{H_0} H_1 \times_{K_1} H_1 \to X_1 \cong X_0 \times_{H_0} H_1$.  This map sends a triple $(x,f,g) \in X_0 \times H_1^2$, where $p_0(x) = s(f)$ and $\Phi_1(f) = \Phi_1(g)$, to a pair $(y,h) \in X_0 \times_{H_0} H_1$.  We claim that $\chi_1$ is uniquely determined by $\chi_0$ and $p$.

Firstly, as $\chi$ is a functor, $\chi_0 s = s \chi_1$, and so we have that
\[
   \chi_0 s(x,f,g)  = s \chi_1 (x,f,g) = s(y,h) = y.
\]
Secondly, since $\chi$ is an action, the triangle
\[\begin{tikzcd}
	{X_1 \times_{K_1} H_1} && {X_1} \\
	& {H_1}
	\arrow["{{\chi_1}}", from=1-1, to=1-3]
	\arrow["{{\pi_2}}"', from=1-1, to=2-2]
	\arrow["{{p_1}}", from=1-3, to=2-2]
\end{tikzcd}\]
commutes.  But as $p_1 \colon X_1 \cong X_0 \times_{H_0} H_1 \to H_1$ is just the second projection, we have that
\[
    h = p_1 (y,h) = p_1 \chi_1(x,f,g) = \pi_2(x,f,g) = g.
\]
Hence, $\chi_1(x,f,g) = (\chi_0 s(x,f,g), g)$.

We now show that $\chi_0 \colon X_0 \times_{H_0} H_0 \times_{K_0} H_0 \to X_0$ is uniquely determined by $p$.  Once again using that $\chi$ is a functor, we have that $\chi_0 t = t \chi_1$, and so combining with the above we have that
\begin{align*}
   \chi_0(t(x,f),t(f),t(g)) & = \chi_0 t(x,f,g) , \\
   & =t\chi_1(x,f,g), \\
   &= t(\chi_0 s(x,f,g), g) = t(\chi_0(x,s(f),s(g)),g). 
\end{align*}
Let $(x,a,b) \in X_0 \times_{H_0} H_0 \times_{K_0} H_0$, i.e.\ $a = p_0(x)$ and $\Phi_0 a = \Phi_0 b$.  As $\Phi$ is fully faithful, there exists a unique arrow $a \xrightarrow{u_{a,b}} b \in H_1$ such that $\Phi_1(u_{a,b}) = \id_{\Phi_0 b}$.  Hence, by above,
\begin{align*}
    \chi_0 (x,a,b) & = \chi_0(t(x,\id_{p_0(x)}),t(\id_{p_0(x)}), t(u_{a,b})) , \\
    & = t(\chi_0(x,s(\id_{p_0(x)}),s(u_{a,b})),u_{a,b}) = t(\chi_0(x,p_0(x),p_0(x)),u_{a,b}).
\end{align*}
Now, since $\chi$ defines an action, we have that $\chi_0(x,p_0(x),p_0(x)) = x$.  Therefore, we can conclude $\chi_0(x,a,b) = t(x,u_{a,b})$.

Thus, a discrete opfibration $p \colon \X \to \H$ admits at most one descent datum.  That the definitions $\chi_0(x,a,b) = t(x,u_{a,b})$ and $\chi_1(x,f,g) = (t(x,u_{s(f),s(g)}),g)$ indeed yield a functor $\chi$ and a valid action is, for the most part, routine and we omit the details.  One part that requires a little more thought is checking that $t \chi_1 = \chi_0 t$. This becomes
  \[t(t(x,u_{s(f),s(g)}),g) = t(t(x,f),u_{t(f),t(g)}).\]
  Since $t \colon X_1 \cong X_0 \times_{H_0} H_1 \to X_0$ is an action by $H_1$, this is equivalent to
  \[t(x,g \circ u_{s(f),s(g)}) = t(x,u_{t(f),t(g)} \circ f).\]
  Thus, it suffices to show $g \circ u_{s(f),s(g)} = u_{t(f),t(g)} \circ f$.
  In fact, using that fact that $\Phi_1$ is (fully) faithful, we only need to check that $\Phi_1(g \circ u_{s(f),s(g)}) = \Phi_1(u_{t(f),t(g)} \circ f)$.
  By assumption, $\Phi_1(u_{s(f),s(g)}) = \id_{\Phi_0 s(g)}$, $\Phi_1(u_{t(f),t(g)}) = \id_{\Phi_0 t(g)}$ and $\Phi_1(f) = \Phi_1(g)$.  Hence, the desired equality holds.

Finally, since any morphism of discrete opfibrations $q \colon \X \to \X'$ commutes with the target map and by the universal property of the pullback used in the definition of $u_{a,b}$, every such morphism $q$ respects the uniquely determined descent data.
\end{proof}
As a consequence of \cref{prop:bundles_descend}, $\Sigma$-structures descend along weak equivalences of localic categories, completing the proof of \cref{prop:mod_extends_to_anafunctors}.  We therefore deduce the following.
\begin{corollary}\label{cor:Tmod_extends}
    Let $\theory$ be a geometric theory and $\theory'$ a dual geometric theory.
    \begin{enumerate}
        \item The pseudofunctor $\Tmod^\rmLH(-) \colon \LocCat\op \to \CAT$ sends weak equivalences of localic categories to equivalences in $\CAT$, yielding a pseudofunctor 
        \[\Tmod^\rmLH(-) \colon \LocCatAna\op \to \CAT.\]
        \item The pseudofunctor $\Tpmod^\rmPS(-) \colon \LocCat\op \to \CAT$ sends weak equivalences of localic categories to equivalences in $\CAT$, yielding a pseudofunctor 
        \[\Tpmod^\rmPS(-) \colon \LocCatAna\op \to \CAT.\]
    \end{enumerate}
\end{corollary}
\begin{proof}
    It remains to show that, if the $\Sigma$-structure $M$ is an LH-structure (respectively, PS-structure), or a model of the geometric theory $\theory$ (resp., dual geometric theory $\theory'$), this is inherited by the descended $\Sigma$-structure $M'$.

    The former holds since both open and proper morphisms descend along effective descent morphisms of locales (see \cite[Proposition 3.6]{henry_townsend_classifying}).  The latter is immediate as $\phi^M \subseteq \psi^M$ if and only if $\phi^{M'} \subseteq \psi^{M'}$.
\end{proof}

\section{The universal property of the generic bundles}\label{sec:univ_prop}

In this section, we arrive at the main result of the paper.  By combining the ingredients developed above, we will demonstrate that the models constructed in \cref{subsec:generic_model} are indeed the \emph{generic models}, in a sense that we now explain.

We will show that the localic category $\ReprG^\rmLH(\theory)$ \emph{classifies} the models for the geometric theory $\theory$ in that it is a representing object for the pseudofunctor
\[
\Tmod^{\rmLH}(-) \colon \LocCatAna\op \to \CAT
\]
that sends a localic category to the category of LH-models of $\theory$ over it (see \cref{sec:base_change,cor:Tmod_extends}). Moreover, the generic model $\ReprE^\rmLH(\theory)$ is the universal element. By following an entirely analogous argument, it can be shown that, given a dual geometric theory $\theory'$, $\Tpmod^\rmPS(-)$ has a representing object in $\ReprG^\rmPS(\theory')$ with universal element $\ReprE^\rmPS(\theory')$.  We will also discuss how to obtain representing \emph{groupoids} for the restricted pseudofunctors $\Tmod^\rmLH_{\cong}(-), \Tpmod^\rmPS_{\cong}(-) \colon \LocGrpdAna\op \to \CAT$.

The proof involves constructing an essentially surjective and fully faithful functor 
\[\overline{\zeta}_\H \colon \LocCatAna(\H,\ReprG^\rmLH(\theory)) \to \Tmod^\rmLH(\H),\]
for each localic category $\H$.  

Recall that for all LH-models $M$ of $\theory$ over a localic category $\K$, pulling back $M$ along an internal functor $\Phi \colon \H \to \K$ yields an LH-model $\Phi^\ast M$ of $\theory$ over $\H$.  Let $\zeta_\H \colon \LocCat(\H,\ReprG^\rmLH(\theory)) \to \Tmod^\rmLH(\H)$ denote the functor that sends an internal functor $\Phi \colon \H \to \ReprG^\rmLH(\theory)$ to the model $\Phi^\ast \ReprE^\rmLH(\theory)$, where $\ReprE^\rmLH(\theory)$ is the `generic model' constructed in \cref{subsec:generic_model}.  Since LH-models of $\theory$ descend along weak equivalences by \cref{cor:Tmod_extends}, the functor $\zeta_H$ extends to $\LocCatAna$ to give a functor
\[
\overline{\zeta}_\H \colon \LocCatAna(\H,\ReprG^\rmLH(\theory)) \to \Tmod^\rmLH(\H).
\]
(Indeed, $\overline{\zeta}_\H(\Phi)$ is simply $\Tmod^\rmLH(\Phi)(\ReprE^\rmLH(\theory))$.)

The intuition behind why $\overline{\zeta}_\H$ is an equivalence, or at least essentially surjective, is easiest to grasp in the case that $\H$ is the terminal localic category $\1$.  Recall that the points of $\ReprG^\rmLH(\theory)_0$ are all the subcountable models of $\theory$.  Given a model $M$ of $\theory$ over $\1$, i.e.\ an ordinary model in $\Set$, we wish to find a corresponding map from $1$ to $\ReprG^\rmLH(\theory)_0$.  If $M$ is subcountable, this is immediate.  If $M$ is not subcountable, we can `force' it to be so by performing a change of base, using \cref{prop:every_set_is_subcount}.  It is this change of base that necessitates the consideration of anafunctors.

\begin{proposition}\label{prop:ess_surj}
    The functor $\overline{\zeta}_\H \colon \LocCatAna(\H,\ReprG^\rmLH(\theory)) \to \Tmod^{\rmLH}(\H)$ is essentially surjective for each localic category $\H$.
\end{proposition}
\begin{proof}
    Assume for simplicity that $\theory$ is single-sorted.  Let $M$ be an LH-model of $\theory$ with underlying carrier $(X, p, \beta)$.  For each (generalised) point $h \in H_0$, the fibre $p^{-1}(h)$ is a $\theory$-model.  Recall that there is an open surjection $[\N \paronto_{H_0} X] \twoheadrightarrow H_0$, and that a point of $[\N \paronto_{H_0} X]$ is a pair consisting of $h \in H_0$ and a partial surjection from $\N$ onto the fibre $p^{-1}(h)$, or equivalently a choice of an isomorphism $\N /{\sim} \cong p^{-1}(h)$ where $\sim$ is a partial equivalence relation on $\N$.  A point of ${\ReprG^\rmLH(\theory)}_0$ is a subquotient of $\N$ with the structure of a $\theory$-model.  Thus, there is an evident span
    \[ H_0 \xleftarrow{\quad} [\N \paronto_{H_0} X] \xrightarrow{\quad} {\ReprG^\rmLH(\theory)}_0 \]
    where the map $[\N \paronto_{H_0} X] \to {\ReprG^\rmLH(\theory)}_0$ sends the pair $(h, \N/{\sim} \cong p^{-1}(h))$ to the subquotient $\N/{\sim} \vDash \theory$ with the induced $\theory$-model structure.

    The open surjection $[\N \paronto_{H_0} X] \to H_0$ induces a localic category $(\N \paronto_{\H} X)$ with objects $[\N \paronto_{H_0} X]$ so that the map $[\N \paronto_{H_0} X] \to H_0$ becomes the action on objects of a fully faithful surjective-on-objects functor $(\N \paronto_\H X) \to \H$ as in \cref{rem:ff_induced_domain}.
    
    A point of $(\N \paronto_\H X)_1$ is a triple $(h_1 \xrightarrow{f} h_2, \N/{\sim}_1 \cong p^{-1}(h_1), \N/{\sim}_2 \cong p^{-1}(h_2))$ where $f$ is a point of $H_1$, i.e.\ an arrow between the points $h_1, h_2 \in H_0$.
    Such an arrow $f$ induces a map between the fibres 
    \[\N/{\sim}_1 \cong p^{-1}(h_1) \xrightarrow{\beta({-},f)} p^{-1}(h_2)\cong \N/{\sim}_2.\] 
    Recall that this map is a homomorphism of the $\theory$-model structure on each fibre.  Therefore, $\beta({-},f)$ yields a homomorphism between $\theory$-models presented as subquotients of $\N$, which is a point of the locale ${\ReprG^\rmLH(\theory)}_1$.  This yields a locale morphism $(\N \paronto_\H X)_1 \to {\ReprG^\rmLH(\theory)}_1$, which, together with the morphism on objects $[\N \paronto_{H_0} X] \to {\ReprG^\rmLH(\theory)}_0$, defines an internal functor of localic categories $(\N \paronto_\H X) \to \ReprG^\rmLH(\theory)$.
    Showing that this is indeed a functor is routine.
    
    Thus, we have constructed an anafunctor $\H \slashedrightarrow \ReprG^\rmLH(\theory)$.  It remains to prove that the image of $\H \slashedrightarrow\ReprG^\rmLH(\theory)$ under the functor $\overline{\zeta}_\H$ is indeed isomorphic to the model $M$.  For this, it suffices to show that there is a locale $E'$ and morphisms making both squares in the diagram
    \[\begin{tikzcd}
	X & {E'} & {\ReprE^\rmLH(\theory)} \\
	{H_0} & {[\N \paronto_{H_0} X]} & {\ReprG^\rmLH(\theory)_0}
	\arrow["p"', from=1-1, to=2-1]
	\arrow[from=1-2, to=1-1]
	\arrow[from=1-2, to=1-3]
	\arrow[from=1-2, to=2-2]
	\arrow[from=1-3, to=2-3]
	\arrow[""{name=0, anchor=center, inner sep=0}, from=2-2, to=2-1]
	\arrow[""{name=1, anchor=center, inner sep=0}, from=2-2, to=2-3]
	\arrow["\lrcorner"{anchor=center, pos=0.125, rotate=-90}, draw=none, from=1-2, to=0]
	\arrow["\lrcorner"{anchor=center, pos=0.125}, draw=none, from=1-2, to=1]
\end{tikzcd}\]
    pullbacks, with compatible actions and $\theory$-model structure.

    Define $E'$ to be the pullback
    \[
    \begin{tikzcd}
       E' \ar{d} \ar{r} \ar[draw = none]{rd}[anchor = center, pos = 0.125]{\lrcorner}& \ReprE^\rmLH(\theory) \ar{d} \\
       {[\N \paronto_{H_0} X]} \ar{r} & \ReprG^\rmLH(\theory)_0 \rlap{\,.}
    \end{tikzcd}
    \]
    A point of $E'$ is a pair of a point $(\N/{\sim} \cong p^{-1}(h)) \in [\N \paronto_{H_0} X]$ and an equivalence class of natural numbers $\overline{n} \in \N/{\sim}$. %
    Similarly, a point in the pullback
    \[
    \begin{tikzcd}
       E'' \ar{d} \ar{r} \ar[draw = none]{rd}[anchor = center, pos = 0.125]{\lrcorner} & X \ar{d} \\
        {[\N \paronto_{H_0} X]} \ar{r} & H_0
    \end{tikzcd}
    \]
    is a pair of a point $(\N/{\sim} \cong p^{-1}(h)) \in [\N \paronto_{H_0} X]$ and a point of the fibre $x \in p^{-1}(h)$.
    Transporting $\overline{n} \in \N/{\sim}$ across the isomorphism $\N/{\sim} \cong p^{-1}(h)$ yields the desired isomorphism $E' \cong E''$.  This same translation between a point of the fibre $x \in p^{-1}(h)$ and an equivalence class $\overline{n} \in \N/{\sim}$ across the isomorphism $\N/{\sim} \cong p^{-1}(h)$ also witnesses the fact that $E'$ and $E''$ have the same $(\N \paronto_\H X)_1$-action and carry the same $\theory$-model structure, for which we omit the details.
\end{proof}

\begin{example}
    In \cref{prop:ess_surj}, we have argued in terms of classifying theories and generalised points.  To further illustrate some of the details, we include here an explicit description of the action on objects $\1 \xleftarrow{} [\N \paronto X] \to \ReprG^\rmLH(\theory)_0$ of the representing anafunctor $\1 \slashedrightarrow \ReprG^\rmLH(\theory)$ for a set-based $\theory$-model $M$ (i.e.\ an LH-model over $\1$) with underlying carrier $X$.

    The locale morphism $[\N \paronto X] \to \ReprG^\rmLH(\theory)_0$ constructed in \cref{prop:ess_surj} acts on the generators of $\ReprG^\rmLH(\theory)_0$ by:
    \begin{align*}
        [n \sim m] & \mapsto \bigvee_{x \in X} [f(n) = x] \land [f(m) = x], \\
        [(n_1, \dots , n_\ell) \in R] & \mapsto \bigvee_{(x_1, \dots , x_\ell) \in R^M } [f(n_1) = x_1] \land \dots \land [f(n_\ell) = x_\ell].
    \end{align*}
    The pullback
    \[
    \begin{tikzcd}
       E' \ar{d} \ar{r} \ar[draw = none]{rd}[anchor = center, pos = 0.125]{\lrcorner}& \ReprE^\rmLH(\theory) \ar{d} \\
       {[\N \paronto X]} \ar{r} & \ReprG^\rmLH(\theory)_0
    \end{tikzcd}
    \]
    has a presentation with generators: $[\equiv n]$ for each $n \in \N$, $[f(n) = x]$ for each $n \in \N$ and $x \in X$, $[n \sim m]$ for each $n, m \in \N$, and $[(n_1, \dots , n_\ell) \in R]$ for each $\ell$-ary relation symbol $R$ and $n_1, \dots , n_\ell \in \N$.  In addition to those relations coming from $[\N \paronto X]$ and $\ReprE^\rmLH(\theory)$, we add the relations $[n \sim m] = \bigvee_{x \in X} [f(n) = x] \land [f(m) = x]$ and $[(n_1, \dots , n_\ell) \in R] =  \bigvee_{(x_1, \dots , x_\ell) \in R^M } [f(n_1) = x_1] \land \dots \land [f(n_\ell) = x_\ell]$.

    Meanwhile, the pullback 
    \[
    \begin{tikzcd}
       E'' \ar{d} \ar{r} \ar[draw = none]{rd}[anchor = center, pos = 0.125]{\lrcorner} & X \ar{d} \\
        {[\N \paronto X]} \ar{r} & \1
    \end{tikzcd}
    \]
    has a presentation where the generators are $[=x]$ for each $x \in X$ and $[f(n)=x]$ for each $n \in \N$ and $x \in X$, subject to the relations coming from $[\N \paronto X]$ and $X$.

    The isomorphism $E' \cong E''$ described in \cref{prop:ess_surj} acts on generators by
    \begin{align*}
        [f(n) = x] & \mapsto [f(n) = x] , \\
        [=x] & \mapsto \bigvee_{n \in \N} [f(n) = x] \land [\equiv n].
    \end{align*}
    and its inverse acts on generators by
    \begin{align*}
        [f(n) = x] & \mapsto [f(n) = x] , \\
        [\equiv n] & \mapsto \bigvee_{x \in X} [f(n) = x] \land [=x].
    \end{align*}
    (The action of this inverse frame homomorphism on $[n \sim m]$ and $[(n_1, \dots , n_\ell) \in R]$ is then forced by the relations.)
\end{example}

We now argue that $\overline{\zeta_\H}$ is fully faithful. We start by the considering the functor $\zeta_\H$.
\begin{lemma}\label{lemma:f+f_with_common_domain}
    The functor $\zeta_\K \colon \LocCat(\K,\ReprG^\rmLH(\theory)) \to \Tmod^{\rmLH}(\K)$ is full and faithful for each localic category $\K$.
\end{lemma}
\begin{proof}
    Recall that if $\Phi \colon \K \to \ReprG^\rmLH(\theory)$ is an internal functor, then $\zeta_\K(\Phi)$ is the $\theory$-model given by the pullback
    \[
    \begin{tikzcd}
        \Phi^\ast \ReprE^\rmLH(\theory) \ar{r} \ar{d} \ar[draw = none]{rd}[anchor = center, pos = 0.125]{\lrcorner} & \ReprE^\rmLH(\theory) \ar{d} \\
        K_0 \ar{r}{\Phi_0} & \ReprG^\rmLH(\theory)_0.
    \end{tikzcd}
    \]
    A (generalised) point of $\ReprG^\rmLH(\theory)_0$ is a model of $\theory$ presented as a subquotient of $\N$, and a point of $\ReprE^\rmLH(\theory)$ is some equivalence class $\bar{n}$ contained in that subquotient.  The map $\Phi_0$ sends each $k \in K_0$ to such a model, which we denote by $\N / {\sim}_{\Phi_0(k)} \vDash \theory$.  A point of the pullback $\Phi^\ast \ReprE^\rmLH(\theory)$ is therefore a pair $(k,\bar{n} \in \N/ {\sim}_{\Phi_0(k)} \vDash \theory)$.

    Recall also that, given a pair of internal functors $\Phi, \Psi \colon \K \rightrightarrows \ReprG^\rmLH(\theory)$, an internal natural transformation $\alpha \colon \Phi \Rightarrow \Psi$ is a map $\alpha \colon K_0 \to \ReprG^\rmLH(\theory)_1$ that sends $k \in K_0$ to an arrow $\alpha_k \colon \Phi(k) \to \Psi(k) \in \ReprG^\rmLH(\theory)_1$, or rather a homomorphism of $\theory$-models $\alpha_k \colon \N / {\sim}_{\Phi_0(k)} \to\N / {\sim}_{\Psi_0(k)}$.  Furthermore, this internal natural transformation is sent by $\zeta_\K$ to the homomorphism of $\theory$-models
    \[
    \begin{tikzcd}
        \Phi^\ast \ReprE^\rmLH(\theory) \ar{rr}{\zeta_\K({\alpha})} \ar{dr} && \Psi^\ast \ReprE^\rmLH(\theory) \ar{dl} \\
        & K_0 &
    \end{tikzcd}
    \]
    which acts by sending $(k,\bar{n} \in \N/ {\sim}_{\Phi_0(k)} \vDash \theory) \in \Phi^\ast \ReprE^\rmLH(\theory)$ to $(k,\alpha_k(\bar{n}) \in \N/ {\sim}_{\Psi_0(k)} \vDash \theory) \in \Psi^\ast \ReprE^\rmLH(\theory)$.

    If $\zeta_\K(\alpha) = \zeta_\K(\beta)$, then for each $k \in K_0$ and $\bar{n} \in \N / {\sim}_{\Phi_0(k)}$, $\alpha_k(\bar{n}) = \beta_k(\bar{n})$. Hence, $\alpha = \beta$ and so $\zeta_\K$ is faithful (using the relation $\top \vdash \bigvee_{n \in \N} [\equiv n]$).  Conversely, if there is a homomorphism of $\theory$-models
    \[
    \begin{tikzcd}
        \Phi^\ast \ReprE^\rmLH(\theory) \ar{rr}{\sigma} \ar{dr} && \Psi^\ast \ReprE^\rmLH(\theory) \ar{dl} \\
        & K_0, &
    \end{tikzcd}
    \]
    then for each $k \in K_0$, there is a homomorphism of the $\theory$-model structures on the fibres of $k$, i.e.\ a homomorphism $\sigma_k \colon \N/{\sim}_{\Phi_0(k)} \to \N / {\sim}_{\Psi_0(k)}$.  The map $k \mapsto \sigma_k$ defines an internal natural transformation $\Tilde{\sigma} \colon K_0 \to \ReprG^\rmLH(\theory)_1$ for which $\zeta_\K(\Tilde{\sigma}) = \sigma$.  Thus, $\zeta$ is also full. 
\end{proof}

\begin{corollary}\label{coro:f+f_for_zeta}
    The functor $\overline{\zeta}_\H \colon \LocCatAna(\H, \ReprG^\rmLH(\theory))\to \Tmod^{\rmLH}(\H)$ is also full and faithful.
\end{corollary}
\begin{proof}
    Recall from \cref{df:anacat} and \cref{prop:2-cells_are_simplified} that, given anafunctors 
    \begin{align*}
         \H \xslashedrightarrow{(\Xi,\Phi)} \G  = (\H \xlongleftarrow{\Xi} \K \xlongrightarrow{\Phi} \G) \text{ and } \H \xslashedrightarrow{(\Sigma,\Psi)} \G  = (\H \xlongleftarrow{\Xi} \K \xlongrightarrow{\Psi} \G),
    \end{align*}
    a 2-cell $\alpha \colon (\Xi,\Phi) \Rightarrow (\Xi,\Psi)$ in $\LocCatAna$ is the same datum as an internal transformation $\alpha \colon \Phi \circ \Sigma' \Rightarrow \Psi \circ \Xi'$ in $\LocCat$, where $\Sigma'$ and $\Xi'$ denote the pullbacks of, respectively, $\Sigma$ and $\Xi$, as in the following diagram.
    \[\begin{tikzcd}
	& \K && \\
	\H && {\mathclap{\K \times_\H \K'}\phantom{\K\K'}} & \G \\
	& {\K'}
	\arrow["\Xi"', from=1-2, to=2-1]
	\arrow["\Phi", curve={height=-18pt}, from=1-2, to=2-4]
	\arrow["{{\Sigma'}}"', from=2-3, to=1-2]
	\arrow["\lrcorner"{anchor=center, pos=0.3, rotate=-135}, draw=none, from=2-3, to=2-1]
	\arrow["{{\Xi'}}", from=2-3, to=3-2]
	\arrow["\Sigma", from=3-2, to=2-1]
	\arrow["\Psi", curve={height=18pt}, from=3-2, to=2-4]
    \end{tikzcd}\]
    Note that $\Sigma'$ and $\Xi'$ are still weak equivalences of localic categories since these are stable under pullback.  Therefore, we have that
    \begin{align*}
        \LocCatAna(\H, \ReprG^\rmLH(\theory))\big[(\Xi,\Phi),(\Sigma,\Psi)\big]
         &= \LocCat(\K \times_\H \K', \ReprG^\rmLH(\theory))\big[\Phi \circ \Sigma' , \Psi \circ \Xi' \big].
    \end{align*}
    Now, by \cref{lemma:f+f_with_common_domain}, the map $\left[\alpha \colon \Psi \circ \Sigma' \Rightarrow \Psi \circ \Xi' \right] \mapsto \zeta_{\K \times_\H \K'}(\alpha)$ induces a bijection
    \[
    \LocCat(\K \times_\H \K', \ReprG^\rmLH(\theory))\big[\Phi \circ \Sigma' , \Psi \circ \Xi' \big] \cong \Tmod^{\rmLH}(\K \times_\H \K')\big[\zeta(\Phi \circ \Sigma') , \zeta(\Psi \circ \Xi') \big].
    \]
    The categories $\Tmod^\rmLH(\H)$ and $\Tmod^\rmLH(\K \times_\H \K)$ are equivalent, using either of the weak equivalences $\Xi \circ \Sigma'$ or $\Sigma \circ \Xi'$ with \cref{cor:Tmod_extends}.
    Under this equivalence, $\zeta(\Phi \circ \Sigma')$ is identified with $\overline{\zeta}(\Xi,\Phi)$, and $\zeta(\Psi\circ \Xi')$ with $\overline{\zeta}(\Sigma,\Psi)$. %
    Thus, the functor $\overline{\zeta}_\H$ acts as a bijection
    \[
    \LocCatAna(\H, \ReprG^\rmLH(\theory))\big[(\Xi,\Phi),(\Sigma,\Psi)\big] \cong \Tmod^{\rmLH}(\H)\big[\overline{\zeta}(\Xi,\Phi), \overline{\zeta}(\Sigma,\Psi)\big],
    \]
    and so $\overline{\zeta}_\H$ is fully faithful.
\end{proof}

Combining \cref{prop:ess_surj} and \cref{coro:f+f_for_zeta}, we have that $\overline{\zeta}_\H\colon \LocCatAna(\H,\ReprG^\rmLH(\theory)) \to \Tmod^\rmLH(\H)$ is an equivalence of categories for each localic category $\H$.
Moreover, the functors $\zeta_\H$ are pseudonatural in $\H$ by the universal property of pullback (or the pseudofunctoriality of $\Tmod^\rmLH$).
Hence, by the universal property of the bicategory of fractions, the functors $\overline{\zeta}_\H$ are also pseudonatural in $\H$ and we arrive at our main result.
\begin{theorem}
 There is a pseudonatural equivalence
    \[
    \overline{\zeta}\colon \LocCatAna(-,\ReprG^\rmLH(\theory)) \simeq \Tmod^{\rmLH}(-) \colon \LocCatAna \to \CAT.
    \]
 That is to say, $\ReprG^\rmLH(\theory)$ is the representing object for $\Tmod^\rmLH(-)$ and $\ReprE^\rmLH(\theory)$ is the corresponding universal element.
\end{theorem}

\subsection{PS-models}

We have argued that, for a geometric theory $\theory$, the localic category $\ReprG^\rmLH(\theory)$ classifies LH $\theory$-models.  By making very mild modifications to the argument, we can also deduce $\ReprG^\rmPS(\theory')$ classifies PS $\theory'$-models for every dual geometric theory $\theory'$.
\begin{theorem}
    Let $\theory'$ be a dual geometric theory. There is pseudonatural equivalence
    \[
    \overline{\zeta'} \colon \LocCatAna(-,\ReprG^\rmPS(\theory')) \simeq \Tpmod^\rmPS(-).
    \]
    That is, $\ReprG^\rmPS(\theory')$ is the representing object for $\Tpmod^\rmPS(-)$. Moreover, $\ReprE^\rmPS(\theory')$ is the corresponding universal element.
\end{theorem}
\begin{proof}
    The functor $\zeta'_\H$ is defined by sending an internal functor $\Phi \colon \H \to \ReprG^\rmPS(\theory')$ to the pullback $\Phi^\ast \ReprE^\rmPS(\theory')$ of the generic PS-model constructed in \cref{subsec:generic_model}. From then on, the argument closely mirrors that for the case of $\Tmod^{\rmLH}(-)$ provided in the proofs of \cref{prop:ess_surj}, \cref{lemma:f+f_with_common_domain} and \cref{coro:f+f_for_zeta}. The only real difference is that, when proving that $\zeta'_\H$ is essentially surjective, we use \cref{thm:alexandroff_hausdorff} in place of \cref{prop:every_set_is_subcount} to deduce that, up to a `forcing extension', every PS-model over $\H$ is a subquotient of the Cantor space.
\end{proof}

\subsection{Representing groupoids}

We can use the universal property of the core of a localic category to deduce that there are also representing groupoids for the pseudofunctors $\Tmod^\rmLH_{\cong}(-), \Tpmod^\rmPS_{\cong}(-) \colon \LocGrpdAna\op \to \CAT$.
\begin{corollary}\label{coro:univ_prop_grpd}
    For each geometric theory $\theory$ and dual geometric theory $\theory'$, there are pseudonatural equivalences
    \begin{align*}
    \Tmod^\rmLH_{\cong}(-) & \simeq \LocGrpdAna(-,\ReprG^\rmLH_{\cong}(\theory)) , \\
    \Tpmod^\rmPS_{\cong}(-) & \simeq \LocGrpdAna(-,\ReprG^\rmPS_{\cong}(\theory')).
    \end{align*}
\end{corollary}
\begin{proof}
    Let $\theory$ be a geometric theory.  For each localic groupoid $\H$, have that
    \begin{align*}
        \LocGrpdAna(\H,\ReprG^\rmLH_{\cong}(\theory)) & \simeq \LocCatAna_{\cong}(\H,\ReprG^\rmLH(\theory)) \\
        & \simeq \Tmod^\rmLH_{\cong}(\H).
    \end{align*}
    The case for a dual geometric theory is identical.
\end{proof}

\begin{remark}[cf.\ \cite{manuell2023representing}]
    The universal property of the localic groupoid $\ReprG^\rmLH_{\cong}(\theory)$ expressed in \cref{coro:univ_prop_grpd} provides an alternative construction of the classifying topos of a geometric theory in terms of localic groupoids, rather than using syntactic categories (as is found in most textbook accounts, e.g.\ \cite[\S X]{SGL} or \cite[\S 2]{TST}).  If $\topos \simeq \Sh(\H)$ is the topos of sheaves on a localic groupoid (and by Joyal--Tierney \cite{joyal1984galois}, every topos is of this form), then a geometric morphism $\Sh(\H) \to \Sh(\ReprG^\rmLH_{\cong}(\theory))$ is the same information as an anafunctor $\H \slashedrightarrow \ReprG^\rmLH_{\cong}(\theory)$ by \cite{moerdijk1988}, which by \cref{coro:univ_prop_grpd} corresponds to an LH-model of $\theory$ over $\H$, or equivalently, a model of $\theory$ internal to the topos $\Sh(\H)$.
\end{remark}

\begin{remark}
 It is actually possible to recover even the non-invertible 2-cells from these localic groupoids by making use of the concept of a Sierpiński homotopy. See \cite{henry_townsend_classifying} for more details.
\end{remark}

\section{Conclusion}\label{sec:concl}

To summarise, in this paper we have constructed classifying localic categories for étale bundles (i.e.\ local homeomorphisms) endowed with the structure of a model for a theory of geometric logic, and for proper separated bundles with dual geometric logical structure.  The universal property of these classifying categories asserts that any such bundle is obtained by `pulling back' a generic bundle over the classifying category.  We now discuss our results within a wider context of classifying groupoids and classifying toposes.

\subsection*{Toposes and `dual toposes' within localic groupoids}

It is well-understood that toposes provide a rich domain in which to understand and study geometric logic --- every topos `embodies' a geometric theory, and two classifying toposes are equivalent if and only if the corresponding geometric theories are \emph{Morita equivalent}.  Similarly, given two geometric theories $\theory_1$ and $\theory_2$, their classifying groupoids $\ReprG^\rmLH_{\cong}(\theory_1)$ and $\ReprG^\rmLH_{\cong}(\theory_2)$ are equivalent in $\LocGrpdAna$ if and only if $\theory_1$ and $\theory_2$ are Morita equivalent.  Thus, by sending a topos $\topos$ to the classifying groupoid of any theory classified by $\topos$, we obtain a full embedding of the 2-category $\Topos_{\cong}$ of toposes, geometric morphisms and invertible natural transformations into $\LocGrpdAna$. Moreover, this inclusion has a left adjoint, given by the \'etale completion (\cref{rem:etale_completion}).

Our results suggest that there is a notion of `dual topos' which relates to dual geometric logic in the same way. These would correspond to the full sub-bicategory $\D$ of $\LocGrpdAna$ spanned by the classifying groupoids of dual geometric theories constructed in \cref{subsec:generic_model}.  The sub-bicategory $\D$ has many of the analogous properties that make $\Topos_{\cong}$ suitable for the study of geometric theories, such as, tautologically, each dual geometric theory having a classifier and every object of $\D$ classifying a dual geometric theory. We can also show that this bicategory has some favourable closure properties:
\begin{proposition}
    The sub-bicategory $\D \subseteq \LocGrpdAna$ of dual geometric classifiers has all (weighted) bilimits.
\end{proposition}
\begin{proof}[Proof sketch]
    For any diagram of dual geometric classifiers, we can describe a dual geometric theory that is classified by the limit.  For example, a morphism from $\H$ into a product $\prod_{i \in I} \ReprG^\rmPS_{\cong}(\theory_i)$ is the data of a collection of morphisms $\{\H \to \ReprG^\rmPS_{\cong}(\theory_i) \mid i \in I\}$, i.e.\ a PS-model of $\theory_i$ over $\H$ for each $i \in I$.  Equivalently, this is a PS-model for the dual geometric theory $\coprod_{i \in I} \theory_i$ obtained by taking a disjoint copy of the sorts, relation symbols and axioms of each $\theory_i$.  Thus, the product exists and is given by $\prod_{i \in I} \ReprG^\rmPS_{\cong}(\theory_i) \cong \ReprG^\rmPS_{\cong}(\coprod_{i \in I} \theory_i)$.
\end{proof}
We have only taken the first steps to understanding the potential of the bicategory $\D$ in the study of dual geometric logic and many further questions remain.  For example, does there exist a left adjoint to the inclusion $\D \hookrightarrow \LocGrpdAna$, giving an analogue of the étale completion of a localic groupoid?

There is also much to be said about the even `purer' kind of dual topos where the sorts, relations and axioms are all indexed by compact Hausdorff locales.
For example, we anticipate a classifier for principal $G$-bundles for a compact Hausdorff group $G$ to be among these. One might also ask if there is a description of such dual toposes in terms of universes of compact Hausdorff locales (cf.\ \cite{marra2020characterisation}) analogous to how toposes axiomatise universes of sets.

\subsection*{Other classifiers}

Our approach is general enough to consider other classes of bundles too. For example, one might ask which localic groupoid classifies bundles which are both proper separated \emph{and} local homeomorphisms. Since these morphisms correspond to (Bishop-)finite sets in the sheaf topos of the codomain, we might expect them to be classified by the (topologically discrete) category of finite sets. We believe this is indeed the case. Indeed, this is related to the correspondence between (finite) covering spaces and actions of the fundamental groupoid (on finite sets) from algebraic topology (see \cite[\S 10.6]{brown2006topology}). %
As another example, we could extend from discrete or compact Hausdorff locales to all locally compact locales. In this case, we believe the points of the category of objects correspond to certain subquotients of $\Srpnsk^\N$ specified by certain dcpo endomorphisms.
Christopher Townsend has been working on constructing a classifying category for locally compact bundles independently by different means.
Finally, other examples inspired by algebraic topology could be considered, such as the classifying localic category for vector bundles.

\bibliographystyle{abbrv}
\bibliography{references}

\end{document}